\documentclass[oneside,english,reqno]{amsart}
\usepackage[T1]{fontenc}
\usepackage[latin9]{inputenc}
\usepackage{geometry}
\usepackage{verbatim}
\usepackage{amsthm}
\usepackage{amstext}
\usepackage{amssymb}
\usepackage{graphicx}
\usepackage{esint}
\usepackage{url}
\usepackage{yhmath}
\usepackage[all]{xy}
\usepackage{color}
\usepackage{epsfig}

\usepackage{shuffle}

\usepackage{enumitem}
\usepackage{scalerel}

\usepackage{thm-restate}

\makeatletter

\numberwithin{equation}{section}
\numberwithin{figure}{section}
\theoremstyle{plain}
\newtheorem*{thm*}{\protect\theoremname}
\theoremstyle{plain}
\newtheorem{thm}{\protect\theoremname}[section]
\theoremstyle{plain}
\newtheorem{lem}[thm]{\protect\lemmaname}
\theoremstyle{plain}
\newtheorem{rem}[thm]{\protect\remarkname}
\theoremstyle{plain}

\theoremstyle{plain}
\newtheorem*{prop*}{\protect\propositionname}
\theoremstyle{plain}
\newtheorem{prop}[thm]{\protect\propositionname}
\theoremstyle{plain}
\newtheorem*{cor*}{\protect\corollaryname}
\theoremstyle{plain}
\newtheorem{cor}[thm]{\protect\corollaryname}
\theoremstyle{plain}

\theoremstyle{plain}
\newtheorem{defn}[thm]{Definition}

\theoremstyle{definition}
\newtheorem{example}[thm]{\protect\examplename}

\usepackage{mathrsfs}
\usepackage{setspace,wrapfig}

\AtBeginDocument{

}

\makeatother

\usepackage{babel}
\providecommand{\corollaryname}{Corollary}
\providecommand{\lemmaname}{Lemma}
\providecommand{\propositionname}{Proposition}
\providecommand{\remarkname}{Remark}
\providecommand{\theoremname}{Theorem}
\providecommand{\conjecturename}{Conjecture}
\providecommand{\notename}{Note}
\providecommand{\examplename}{Example}

\begin{document}


\global\long\def\sF{\mathcal{F}}
\global\long\def\sZ{\mathcal{Z}}
\global\long\def\sD{\mathcal{D}}
\global\long\def\sC{\mathcal{C}}
\global\long\def\sL{\mathcal{L}}
\global\long\def\sA{\mathcal{A}}
\global\long\def\sR{\mathcal{R}}
\global\long\def\sS{\mathcal{S}}
\global\long\def\sP{\mathcal{P}}
\global\long\def\sM{\mathcal{M}}

\global\long\def\bR{\mathbb{R}}
\global\long\def\bRpos{\bR_{> 0}}
\global\long\def\bRnn{\bR_{\geq 0}}
\global\long\def\bZ{\mathbb{Z}}
\global\long\def\bN{\mathbb{N}}
\global\long\def\bZpos{\mathbb{Z}_{> 0}}
\global\long\def\bZnn{\mathbb{Z}_{\geq 0}}
\global\long\def\bQ{\mathbb{Q}}
\global\long\def\bC{\mathbb{C}}

\global\long\def\PR{\mathbb{P}}
\global\long\def\EX{\mathbb{E}}

\global\long\def\one{\scalebox{1.1}{\textnormal{1}} \hspace*{-.75mm} \scalebox{0.7}{\raisebox{.3em}{\bf |}} }

\global\long\def\bD{\mathbb{D}}
\global\long\def\bH{\mathbb{H}}
\global\long\def\re{\Re\mathfrak{e}}
\global\long\def\im{\Im\mathfrak{m}}
\global\long\def\arg{\mathrm{arg}}
\global\long\def\ii{\mathfrak{i}}
\global\long\def\domain{D}
\global\long\def\bdry{\partial}
\global\long\def\cl#1{\overline{#1}}
\global\long\def\confmap{\varphi}
\global\long\def\zbar{\bar{z}}

\newcommand{\invbreve}[1]{\overset{\rotatebox{180}{$\breve{}\,$}}{#1}}

\global\long\def\diam{\mathrm{diam}}
\global\long\def\dist{\mathrm{dist}}

\global\long\def\OO{\mathcal{O}}
\global\long\def\oo{\mathit{o}}

\global\long\def\ud{\mathrm{d}}
\newcommand{\pder}[1]{\partial^+_{#1}}
\newcommand{\pderr}[1]{\partial^+_{#1}}

\global\long\def\SLE{\mathrm{SLE}}
\global\long\def\SLEk{\mathrm{SLE}_{\kappa}}

\global\long\def\caglad{\smash\gamma^\flat}
\global\long\def\cadlag{\smash\gamma^\sharp}
\global\long\def\newmathringcadlag{\mathring{\gamma}^\sharp}

\newcommand{\wt}[1]{\smash{\scalebox{.5}{$\widetilde{\scalebox{2}{$#1$}}$}}}

\global\long\def\Poisson{N}
\global\long\def\PoissonComp{\smash{\wt{\Poisson}}}

\global\long\def\PreGirsB{B}
\global\long\def\GirsB{\breve{B}}
\global\long\def\PreGirsN{N}
\global\long\def\GirsN{\breve{N}}
\global\long\def\GirsPR{\breve{\PR}}
\global\long\def\GirsEX{\breve{\EX}}
\global\long\def\RN{R}

\global\long\def\mgle{\smash{\check{M}}}

\global\long\def\microDriver{\smash{\wt{W}}}
\global\long\def\macroDriver{W}
\global\long\def\jump{v}

\global\long\def\growing{\xi^{\text{out}}}
\global\long\def\grown{\xi^{\text{in}}}
\global\long\def\growingpt{z^{\text{out}}}
\global\long\def\grownpt{z^{\text{in}}}

\global\long\def\maxjumpH#1{\smash{\lambda_{#1}^{\textnormal{h\"ol}}}}
\global\long\def\maxjumpT#1{\smash{\lambda_{#1}^{\textnormal{tr}}}}

\newcommand{\ThetaHmax}[2]{\smash{\theta_{#1,#2}^{\textnormal{h\"ol}}}}
\newcommand{\ThetaTmax}[2]{\smash{\theta_{#1,#2}^{\textnormal{tr}}}}
\newcommand{\ThetaTmaxbeta}[2]{\smash{\vartheta_{#1,#2}^{\textnormal{tr}}}}
\newcommand{\maxbetaT}[2]{\smash{\alpha_{#1,#2}^{\textnormal{tr}}}}

\global\long\def\rparamH{r}
\global\long\def\rparamT{r}

\global\long\def\Grid{\mathcal{G}}


\allowdisplaybreaks



\author{E.~Peltola and A.~Schreuder}

\

\vspace{2.5cm}

\begin{center}
\LARGE \bf \scshape{
On the geometry of locally growing Loewner chains
}
\end{center}

\vspace{0.75cm}

\begin{center}
{\large \scshape Eveliina Peltola}\\
{\footnotesize{\tt eveliina.peltola@aalto.fi}}\\
{\small{Department of Mathematics and Systems Analysis,}}\\
{\small{P.O. Box 11100, FI-00076, Aalto University, Finland}}\\
{\small{and}}\\
{\small{Division of Mathematics, University of Cologne}}\\
{\small{Weyertal 86-90, D-50931 Cologne, Germany}}\\
\bigskip{} \bigskip{} 
{\large \scshape Anne Schreuder}\\
{\footnotesize{\tt anne.schreuder@aalto.fi}}\\
{\small{Department of Mathematics and Systems Analysis,}}\\
{\small{P.O. Box 11100, FI-00076, Aalto University, Finland}}
\end{center}

\vspace{0.75cm}

\begin{center}
\begin{minipage}{0.85\textwidth} \footnotesize
{\scshape Abstract.}
Loewner chains are ubiquitous in the theory of slit mappings, and hence in the study of bounded conformal maps. 
They have attracted new interest in the past decades through their applications to statistical physics and fractal geometry, particularly in contexts involving randomness. In this article, we delve into topological features of the growing hulls obtained from Loewner chains with a general local growth property, inspired by the classical works of Loewner and Pommerenke. 

\qquad
We first revisit Loewner's theorem, associating to each locally growing collection of hulls a real-valued driving function $W$, possibly discontinuous. We then investigate the points chronologically added to the growing hulls, which may be part of a simply connected swallowed ``bubble,'' or a compact connected boundary set. For continuous driving functions, the Loewner chain can often be associated with a continuous curve (dubbed ``generating curve''). Motivated by this, we introduce a more general notion of a ``generating function'' $\eta$ for the Loewner chain, and characterize when there exists such a function $\eta$ (which can be continuous, c\`adl\`ag, c\`agl\`ad, or neither). We then investigate the necessity of left and right limits for $\eta$ from the point of view of the topology of the growing hulls. We find in particular that left-continuity implies path-connectedness and local connectedness of the hulls, as well as the existence of right limits, whereas failure of left-continuity leads to pathological boundary behavior.
\end{minipage}
\end{center}


\setcounter{tocdepth}{2}
\tableofcontents


\section{Introduction}
We shall investigate the general theory of function-driven Loewner chains --- a question which goes back to the origins of Loewner theory. 
In his seminal paper~\cite{Loewner:Untersuchungen_uber_schlichte_konforme_Abbildungen_des_Einheitskreises}, 
Charles Loewner introduced a concept, now known as \emph{Loewner chain}, as a method to study univalent (holomorphic, injective) functions $f \colon \bD \to \bD$.  
He derived an evolution equation, now known as \emph{Loewner's equation}, that describes the dynamical shrinking of the sets $f(\bD)$ in terms of a real-valued function, the \emph{driving function}.
This remarkably transforms the task of describing the evolution of the domains $f(\bD)$ in the plane into describing the evolution of one-dimensional functions, which are much easier to investigate a priori. 
In particular, he proved that Loewner chains with continuous driving functions form a dense subset of the space of all univalent maps $f \colon \bD \to \bD$ (normalized as $f(0)=0$ and $f'(0)=1$).
Notably, this includes $f(\bD)$ which are slit domains, and $f(\bD)$ which arise from complements of simple Jordan curves.
Using this theory, Loewner was able to show that the Bieberbach conjecture\footnote{The first coefficient estimate $\smash{\frac12 |f''(0)| \le 2}$ had previously been proven in~\cite{Bieberbach:Uber_die_Koeffizienten_derjenigen_Potenzreihen}. The second estimate $\smash{\frac16 |f'''(0)| \le 3}$ was new. 
The full Bieberbach conjecture (dubbed de Branges's theorem), stating that $\smash{\frac{1}{n!} |f^{(n)}(0)| \le n}$ for all $n \ge 2$,
was later proven by Louis de Branges in~\cite{DeBranges:Proof_of_Bieberbach_conjecture}, also using Loewner's theory.} 
holds for the first two coefficients.

However, Loewner was also aware that not every function-driven Loewner chain gives rise to a slit map. 
Nonetheless, he did not know what a counterexample would look like --- see~\cite[Section~5]{Loewner:Untersuchungen_uber_schlichte_konforme_Abbildungen_des_Einheitskreises}. 
Pommerenke later answered this question~\cite[Theorem~1]{Pommerenke:On_the_Loewner_differential_equation}: 
A Loewner chain with a continuous driving function precisely generates locally growing sets.  
This includes the graphs of all non-self-crossing curves, as well as some more intricate geometric objects (such as the spiral in Example~\ref{subsec: logarithmic spiral}).

In this work, we generalize these results as to not assume any regularity for the driving function a priori.
By general theory of Loewner chains, 
a driving function should at least be measurable in time. 
A closer study reveals that the minimal regularity of the driving function is in fact being c\`adl\`ag\footnote{The acronym ``c\`adl\`ag'' is short for \emph{continue \`a droite, limites \`a gauche}, i.e.,~right-continuous with unique left limits.}.   
This leads to a small relaxation in the definition of local growth:  
We do not require the local growth to be uniform in time. 
Geometrically, this condition forces left and right limits of the driving function to exist, but these limits do not have to coincide.  
For instance, graphs of self-crossing curves satisfy this new local growth assumption, whereas they fail the older, stricter version. 
We present a modified version of Loewner's result (for hulls\footnote{Loewner's theory can be viewed in terms of growing sets,
which are complements of the shrinking domains, and can be considered either in the original ``radial'' setting in the unit disc $\bD$, or almost equivalently in the ``chordal'' setting in the upper half-plane $\bH$. 
Analogous statements to our results could be proven in the radial setting via identical arguments.} 
in the upper half-plane $\bH$) 
in Theorems~\ref{thm: loewner intro}~\&~\ref{thm: Locally growing hulls solve LE intro}, 
and include a proof for it in Section~\ref{sec: Loewners Theorem}.

One noteworthy special case is when the driving function is a standard Brownian motion with diffusivity parameter $\kappa \geq 0$. Then, the hulls are generated by 
a random continuous fractal curve (Loewner trace)~\cite{LSW:Conformal_invariance_of_planar_LERW_and_UST, Rohde-Schramm:Basic_properties_of_SLE},
aka the celebrated Schramm-Loewner evolution ($\SLE_\kappa$),
which has turned out to be a universal and remarkably useful object in probability theory and mathematical 
physics\footnote{This is a vast topic nowadays. $\SLE_\kappa$ processes have played a key role in establishing rigorous results for scaling limits of many critical lattice models, e.g., 
in~\cite{Schramm:Scaling_limits_of_LERW_and_UST, Smirnov:Critical_percolation_in_the_plane, LSW:Conformal_invariance_of_planar_LERW_and_UST, Schramm:ICM, Schramm-Sheffield:Contour_lines_of_2D_discrete_GFF,  CDHKS:Convergence_of_Ising_interfaces_to_SLE},
and for important questions in probability theory and conformal geometry: 
Brownian intersection exponents 
\cite{Duplantier-Kwon:Conformal_invariance_and_intersections_of_random_walks, LSW:Brownian_intersection_exponents1, LSW:Brownian_intersection_exponents2, LSW:Brownian_intersection_exponents3, Werner:Girsanovs_transformation_for_SLE_kappa_rho_intersection_exponents_and_hiding_exponents} 
and Hausdorff dimension of the Brownian frontier 
\cite{LSW:The_dimension_of_the_planar_Brownian_frontier_is_four_thirds},
constructions of conformal restriction measures 
\cite{LSW:Conformal_restriction_the_chordal_case},
couplings with the Gaussian free field~\cite{Dubedat:SLE_and_free_field, Miller-Sheffield:Imaginary_geometry1, Miller-Sheffield:Imaginary_geometry2, Miller-Sheffield:Imaginary_geometry3, Miller-Sheffield:Imaginary_geometry4},
constructions of random metric or measure  spaces 
\cite{DMS:Liouville_quantum_gravity_as_mating_of_trees, BGS:Permutons_meanders_and_SLE-decorated_Liouville_quantum_gravity} (see also references therein),
and recent results concerning the relationship of fractal objects in random geometry (such as $\SLEk$ type paths) with conformal field theory, see~\cite{Peltola:Towards_CFT_for_SLEs, ARS:FZZ_formula_of_boundary_Liouville_CFT_via_conformal_welding} and references therein. 
Although we consider more general Loewner theory in the present work, these also serve as an important motivation.}.
It is this object as well as its many applications that have sparked a renewed interest in the study of deterministic function-driven Loewner chains. 
The interest focused mainly on the question whether a Loewner chain describes the complement of a curve, and on the regularity of this curve. 
It is notoriously difficult to prove that $\SLE_\kappa$ has a trace, i.e.,~that it is almost surely a curve --- 
see~\cite{LSW:Conformal_invariance_of_planar_LERW_and_UST,Rohde-Schramm:Basic_properties_of_SLE,Ambrosio-Miller:Continuous_proof_of_the_existence_of_SLE8_curve,KMS:Regularity_of_the_SLE4_uniformizing_map_and_the_SLE8_trace}


In~\cite{Marshall-Rohde:The_loewner_differential_equation_and_slit_mappings}, Marshall~\&~Rohde prove that a Loewner chain with a $1/2$-H\"older continuous driving function whose H\"older $1/2$-norm is smaller than an unknown critical value describes the complement of a (simple) quasi-arc\footnote{This condition narrowly fails for $\SLE_\kappa$. 
Namely, Brownian motion is almost surely $\alpha$-H\"older continuous for all $\alpha \in (0, 1/2)$, while its H\"older $1/2$-norm is almost surely infinite.}. 
Conversely, every quasi-arc can be represented by a Loewner chain whose driving function is $1/2$-H\"older continuous. 
In the same year, Lind proved in~\cite{Lind:Sharp_condition_for_Loewner_equation_to_generate_slits} that this constant equals $4$. 
Its sharpness was shown by providing counterexamples of driving functions with H\"older $1/2$-norms at least~$4$ whose Loewner chains describe complements of non-simple curves. 
In~\cite{LMR:Collisions_and_spirals_of_Loewner_traces} this was taken one step further: 
The Loewner chain with driving function $W(t) = 4 \sqrt{1 - t}$ is not the complement of any continuous curve, but a of spiral --- see also Example~\ref{subsec: logarithmic spiral}. 
In particular, the space of driving functions with continuous traces is not convex. 
See also~\cite{KNK:Exact_solutions_for_Loewner_evolutions, Kinneberg:Loewner_chains_and_Holder_geometry} for related results.

Loewner chains with driving functions are a special case of a general theory of measure-driven Loewner chains, investigated initially by Kufarev~\cite{Kufarev:On_one-parameter_families_of_analytic_functions}. 
The general Loewner chains are allowed to grow from multiple areas of the boundary simultaneously. 
We recommend~\cite[Chapter~6]{Pommerenke:Univalent_functions},~\cite[Section~6]{Miller-Sheffield:QLE} and~\cite[Chapter~5]{Beliaev:Conformal_maps_and_geometry} for further reading on this topic. 
Additionally, the article~\cite{Sola:Elementary_examples_of_Loewner_chains_generated_by_densities} contains many interesting examples of explicit Loewner chains and their driving measures.

\smallskip
\subsection{General theory of locally growing Loewner chains}

\begin{figure}[ht]
\begin{center}
\includegraphics[width=0.27\textwidth]{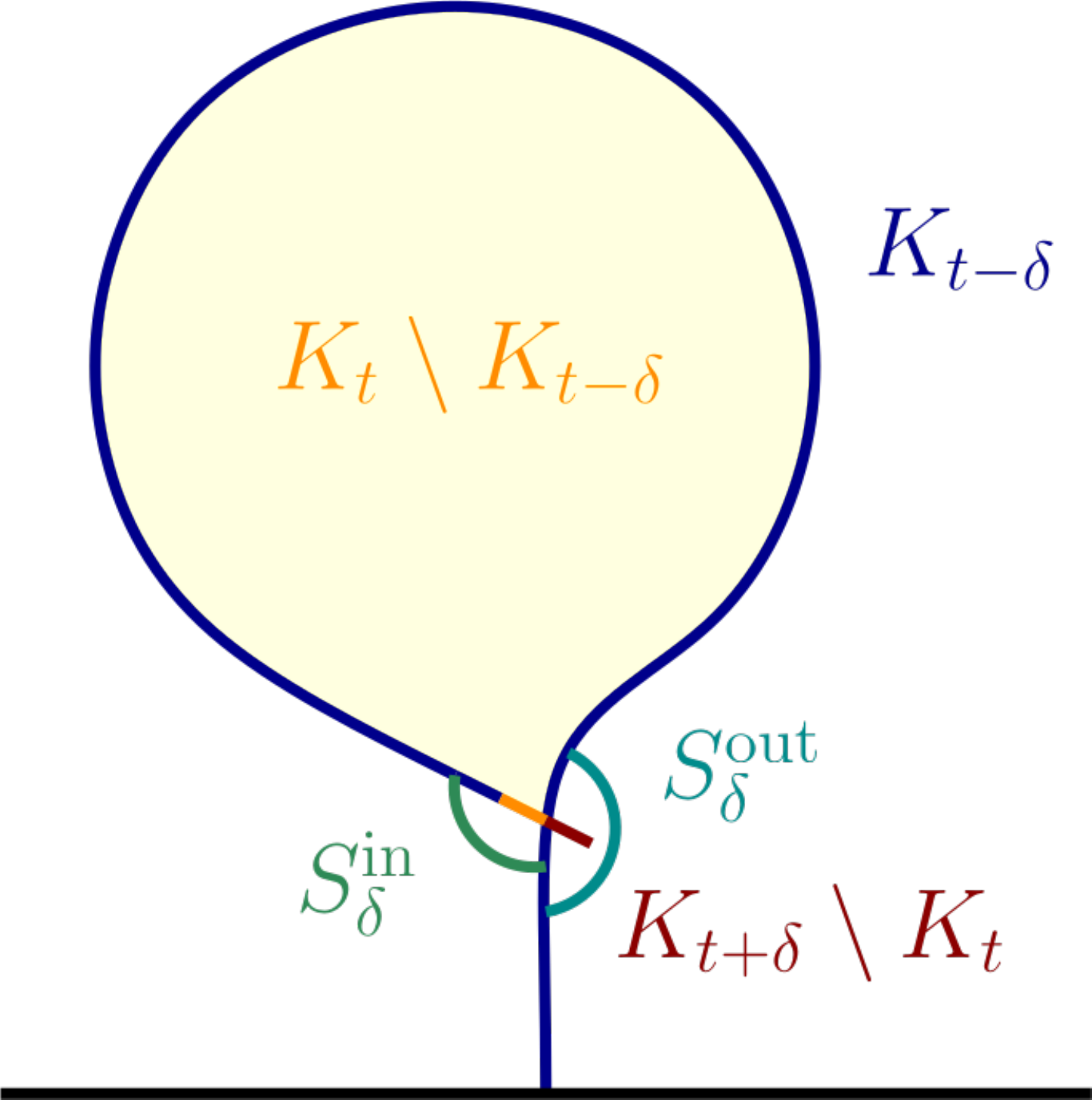}
\end{center}
\caption{\label{fig: left-local growth}Illustration of the local growth property (Definition~\ref{def: local growth}).}
\end{figure}

By a variation of the arguments leading to Loewner's classical theorem, we obtain a bijection between real-valued c\`adl\`ag functions and locally growing hulls in the following sense. 
This result was partly inspired by~\cite{Peltola-Schreuder:Loewner_traces_driven_by_Levy_processes}, where we study random Loewner chains whose driving functions are general L\'evy processes; see also~\cite{ROKG:Stochastic_Loewner_evolution_driven_by_Levy_processes,Guan-Winkel:SLE_and_aSLE_driven_by_Levy_processes,Chen-Rohde:SLE_driven_by_symmetric_stable_processes}.

\begin{restatable}{defn}{DefLocalGrowth}
\label{def: local growth}
Let $\boldsymbol{K} = (K_t)_{t \geq 0}$ be a growing family of hulls \textnormal{(}Definition~\ref{def: growth}\textnormal{)}. 
Let $t \geq 0$ be fixed.
\begin{enumerate}[label=\textnormal{(\arabic*):}, ref=\textnormal{(\arabic*)}]
\item\label{item: left-local growth} 
$\boldsymbol{K}$ are \emph{left-locally growing} at time $t$ if for every $\varepsilon > 0$ there exists 
$\delta = \delta(\varepsilon, t) > 0$ and a crosscut $S_{\delta}^{\mathrm{in}} \subset \bH \setminus K_{t-\delta}$  
with $\diam(S_{\delta}^{\mathrm{in}}) < \varepsilon$ 
such that $S_{\delta}^{\mathrm{in}}$ separates $K_{t} \setminus K_{t-\delta}$ from $\infty$ in $\bH \setminus K_{t-\delta}$. 

\smallskip

\item\label{item: right-local growth}
$\boldsymbol{K}$ are \emph{right-locally growing} at time $t$ if for every $\varepsilon > 0$ there exist 
$\delta = \delta(\varepsilon, t) > 0$ and a crosscut $S_{\delta}^{\mathrm{out}} \subset \bH \setminus K_t$  
with $\diam(S_{\delta}^{\mathrm{out}}) < \varepsilon$ 
such that $S_{\delta}^{\mathrm{out}}$ separates $K_{t + \delta} \setminus K_t$ from $\infty$ in 
$\bH \setminus K_t$. 
\end{enumerate}
Moreover, we say that the hulls $\boldsymbol{K}$ are \emph{locally growing} if they satisfy~\ref{item: left-local growth}~and~\ref{item: right-local growth}.
\end{restatable}

For each hull $K$, 
there exists a unique conformal bijection (the \emph{mapping-out function})
\begin{align} \label{eq: mof Laurent exp}
g_K \colon \bH \setminus K \to \bH ,
\qquad \qquad
g_K(z) = z +  \frac{\mathrm{hcap}(K)}{z} +
\sum_{n = 2}^{\infty} a_n(K) \, z^{-n} , \qquad |z| \to \infty ,
\end{align}
where the first coefficient $\mathrm{hcap}(K) \geq 0$ is called the \emph{half-plane capacity} of $K$. 
Growing hulls $\boldsymbol{K} = (K_t)_{t \geq 0}$ can be parameterized in various ways. 
The most common choice is to parameterize them by capacity.

\begin{restatable}{thm}{ThmLoewner}
\emph{(Loewner's theorem)}
\label{thm: loewner intro}
There is a bijection between the following objects.
\begin{enumerate}[label=\textnormal{(\arabic*):}, ref=\textnormal{(\arabic*)}]
\item \label{item: loewner hulls}
A family $(K_t)_{t \geq 0}$ of locally growing hulls parametrized by capacity, 
i.e.,~$\mathrm{hcap}(K_t) = 2 t$ for $t \geq 0$.

\smallskip

\item \label{item: loewner driver} 
A c\`adl\`ag function $W \colon [0, \infty) \to \bR$.  
\end{enumerate}
We call $W$ the \emph{driving function} of the hulls $\boldsymbol{K} = (K_t)_{t \geq 0}$. 
\end{restatable}

Using the same proof strategy, this generalizes known results for Loewner chains with continuous driving functions 
(in which case the local growth is uniform over time; cf.~Remark~\ref{remark: uniform local growth}).
We refer to~\cite{Loewner:Untersuchungen_uber_schlichte_konforme_Abbildungen_des_Einheitskreises,Pommerenke:On_the_Loewner_differential_equation}, 
and~\cite[Chapter~4]{Lawler:SLE},~\cite[Chapter~4]{Kemppainen:SLE_book}, and~\cite[Chapter~5]{Beliaev:Conformal_maps_and_geometry} for further reading. 

Importantly, the locally growing hulls $\boldsymbol{K}$ can be obtained from their driving function $W \colon [0, \infty) \to \bR$, because their mapping-out functions $(g_t)_{t \geq 0}$
(dubbed a ``Loewner chain'') 
solve the Loewner differential equation~\eqref{eq: LE}, which is an ODE for the time-evolution of the growth: 

\begin{restatable}{thm}{ThmLoewnerEquation}
\label{thm: Locally growing hulls solve LE intro}
Let $W \colon [0, \infty) \to \bR$ be a c\`adl\`ag function. 
Then, for all $z \in \bH$, the \emph{Loewner equation}
\begin{align}
\label{eq: LE}
\begin{split}
\pderr{t} g_t(z) &= \frac{2}{g_t(z) - W(t)} ,
\\
g_0(z) &= z ,
\end{split} \tag{LE}
\end{align} 
admits a unique absolutely continuous solution up to the blow-up time
\begin{align*}
\tau(z) := \sup \Big\{ s \geq 0 \; | \; \inf_{u \in [0, s]} | g_u(z) - W(u) | > 0 \Big\} \; \in \; [0,\infty].
\end{align*}
The sets $K_t := \big\{ z \in \overline{\bH} \mid \tau(z) \leq t \big\}$, for $t \geq 0$, define a family of locally growing hulls, denoted $\boldsymbol{K} = (K_t)_{t \geq 0}$.
Moreover, for $t \geq 0$ fixed, 
$g_t \colon \bH \setminus K_t \to \bH$ is the mapping-out function of $K_t$,
that is, the unique conformal map such that $|g_t(z)  -  z| \to 0$ as $|z| \to \infty$.
We call $(g_t)_{t \geq 0}$ a \emph{Loewner chain}.
\end{restatable}

Conversely, it is possible to obtain the driving function from the corresponding locally growing hulls:
\begin{align} 
\label{eq: shrink to points}
\{W(t)\} = \bigcap_{\delta > 0} \overline{g_t( K_{t+\delta} \setminus K_t )} \; \subset \; \bR
\qquad \textnormal{and} \qquad
\{W(t-)\} = \bigcap_{0 < \delta < t} \overline{g_{t-\delta}( K_t \setminus K_{t-\delta} )}  \; \subset \; \bR.
\end{align} 
We will make this statement more precise in Proposition~\ref{prop: characterise grown/growing end}, whereas the two equations in~\eqref{eq: shrink to points} themselves are proven in Proposition~\ref{prop: first construction driving measure}.
Note that these limits can be different, as we do not assume the driving function to be continuous. 
Therefore, the result is also not obvious.

Furthermore, without any additional assumption we establish the following decomposition of locally growing hulls. 
In particular, this shows that a newly added point $z \in K_t \setminus \bigcup_{s < t} K_s$ is either swallowed 
(i.e.,~fully disconnected from infinity by the hull $K_t$), or is part of a compact, connected set of boundary points. 
We have not seen this form of the result in the literature, and it may be of independent interest. 

\begin{figure}[ht]
\centering
\includegraphics[width=0.25\textwidth]{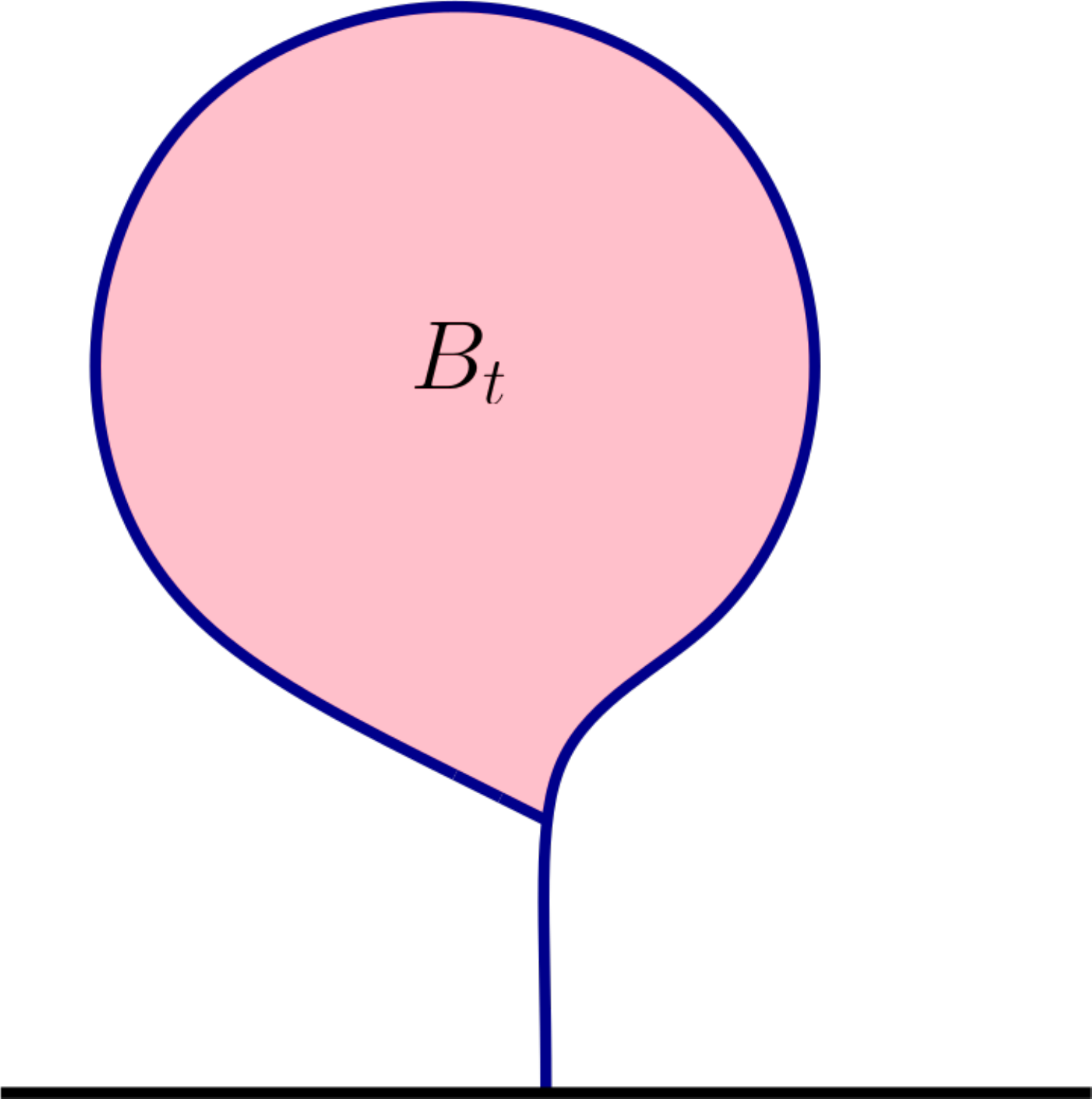}
\caption{Illustration of locally growing hulls forming a bubble $B_t$.}
\end{figure}

\begin{thm}
\label{thm: decomposition added stuff to hull intro}
Let $\boldsymbol{K} = (K_t)_{t \geq 0}$ be a family of left-locally growing hulls. 
Then, for all $t \geq 0$, we have
\begin{align*}
K_t \cup \bR = \Big( \bigcup_{s < t} K_s \cup \bR \Big) \cup B_t \cup P_t,
\qquad \textnormal{where} 
\end{align*}
\begin{itemize}
\item 
$B_t = \big( \bigcap_{s < t} K_t \setminus K_s \big) \setminus ( \bdry K_t \cup \bR )$ is a \emph{bubble}, 
and it is either empty or open and simply connected;

\smallskip

\item 
$P_t = ( \bdry K_t \cap \bH ) \setminus \bigcup_{s < t} K_s$ is compact and connected. 
\end{itemize}
\end{thm}

Theorem~\ref{thm: decomposition added stuff to hull intro} is an immediate consequence of Theorem~\ref{thm: decomposition added stuff to hull}, which we prove in Section~\ref{subsec:thm14}. 

\smallskip
\subsection{Locally growing hulls generated by a function}

Intuitively, local growth implies that the associated hulls $\boldsymbol{K}$ 
grow from a compact connected null-set\footnote{In the setting of Theorem~\ref{thm: decomposition added stuff to hull intro}, if $P_t \neq \emptyset$, then this compact connected null-set is $P_t$.} with respect to harmonic measure.
In most cases, such a compact connected null-set is a single point. 
Hence usually, but not always (see Example~\ref{subsec: logarithmic spiral}), it is then possible to associate to a locally growing Loewner chain a specific function $\eta \colon [0, \infty) \to \overline{\bH}$. 
If such a function exists, we call it a ``generating function'' of the Loewner chain.  

\begin{restatable}{defn}{DefGenFct} 
\label{def: generating function}
A family $\boldsymbol{K}$ 
of locally growing hulls, and the associated Loewner chain, is \emph{generated by a function} $\eta \colon [0,\infty) \to \overline{\bH}$ if, for each $t \geq 0$, the set $\bH \setminus K_t$ is 
the unbounded connected component of $\bH \setminus \eta[0,t]$. 
\textnormal{(}In the literature, it is often assumed that $\eta$ is continuous, which we will not assume here.\textnormal{)}
\end{restatable}

The following characterization of the existence of generating functions seems to be new.
It is an immediate consequence of a slightly more detailed result, Theorem~\ref{thm: equivalence generated by a function}, which we prove in Section~\ref{subsec:thm16}. 

\begin{thm}
\label{thm: generated by a fct intro}
Let $\boldsymbol{K} = (K_t)_{t \geq 0}$ be a family of locally growing hulls. 
Let $W \colon [0, \infty) \to \bR$ be the associated c\`adl\`ag driving function. 
Then, the following are equivalent.
\begin{enumerate}[label=\textnormal{(\arabic*):}, ref=\textnormal{(\arabic*)}]
\item The hulls $\boldsymbol{K}$ are generated by a function $\eta \colon [0, \infty) \to \overline{\bH}$.

\smallskip

\item For all $t \geq 0$, the set $\smash{P_t = ( \bdry K_t \cap \bH ) \setminus \bigcup_{s < t} K_s}$ consists of at most one point.

\smallskip

\item For all $t \geq 0$, the set $K_t \cup \bR$ is path-connected.
\end{enumerate}
In that case, we have $P_t = (\bdry K_t \cap \bH) \setminus \bigcup_{s < t} K_s \subset \{\eta(t)\}$.
Moreover, if $P_t \neq \emptyset$, then the limit
\begin{align}
\label{eq: canonical eta}
\eta(t) = \lim_{y \to 0+} g_t^{-1}(W(t-) + \ii y)
\end{align}
exists and is accessible from $\bH \setminus K_t$.
\end{thm}

If the conditions in Theorem~\ref{thm: generated by a fct intro} hold, 
it is possible to identify the hulls $\boldsymbol{K} = (K_t)_{t \geq 0}$ with the generating function $\eta \colon [0, \infty) \to \overline{\bH}$. 
As a consequence, Loewner's theorem can then be seen as a mapping (Loewner transform) sending driving functions $W \colon [0, \infty) \to \bR$ to their corresponding generating functions $\eta \colon [0, \infty) \to \overline{\bH}$. 
However, there is no recipe to see in general which driving functions $W$ give rise to generating functions $\eta$ --- as Loewner pondered in~\cite[Section 5]{Loewner:Untersuchungen_uber_schlichte_konforme_Abbildungen_des_Einheitskreises} in the continuous setup.

At this stage, we have not assumed any regularity for the generating function, other than satisfying Definition~\ref{def: generating function}. 
It turns out that a generating function needs to neither be left- nor right-continuous, as for instance the double-comb (Example~\ref{eg: double comb}) shows.
Nonetheless, the failure of right-continuity appears to be a highly specific phenomenon. 
In particular, it is tied to the limit~\eqref{eq: canonical eta} not existing for some time $t \geq 0$, while the hulls are still generated by a function. 
Conversely, as our next result shows, if the limit~\eqref{eq: canonical eta} exists everywhere, then the generating function has a right-continuous version.

\begin{thm}\label{thm:coro_of_2}
Let $\boldsymbol{K} = (K_t)_{t \geq 0}$ be a family of locally growing hulls. 
Let $W \colon [0, \infty) \to \bR$ be the associated c\`adl\`ag driving function. 
If the limit~\eqref{eq: canonical eta} exists for all $t \geq 0$, then the following hold. 
\begin{enumerate}[label=\textnormal{(\arabic*):}, ref=\textnormal{(\arabic*)}]
\item 
The hulls $\boldsymbol{K}$ are generated by $\eta \colon [0, \infty) \to \overline{\bH}$. 

\smallskip

\item 
For all $t \geq 0$, the function $\eta$ has a unique right limit at $t$: 
\begin{align*}
\eta(t+) := \lim_{s \to t+} \eta(s) = \lim_{y \to 0+} g_t^{-1}(W(t+) + \ii y) = \lim_{y \to 0+} g_t^{-1}(W(t) + \ii y) .
\end{align*}
\end{enumerate} 
Moreover, for all $t \geq 0$, we have $P_t = (\bdry K_t \cap \bH) \setminus \bigcup_{s < t} K_s \subset \{\eta(t)\}$, and $\eta(t)$ is accessible from $\bH \setminus K_t$.
\end{thm}

Theorem~\ref{thm:coro_of_2} is an immediate consequence of Theorems~\ref{thm: existence right-continuous generating function}~\&~\ref{thm: Minimal regularity of generating functions}, which we prove in Section~\ref{subsec: right-continuous generating functions}. 
Moreover, Theorem~\ref{thm: existence right-continuous generating function} gives a characterization for $\eta$ in terms of prime ends (see also Section~\ref{sec: preli}). 

It is worth noting that the assumptions of Theorem~\ref{thm:coro_of_2} do not imply that the generating function is {\em left-continuous}, as Examples~\ref{subsec: comb space}~\&~\ref{eg: Belyaev} show. 
When the generating function fails to be left-continuous, then the limit \eqref{eq: canonical eta} is still a canonical choice of the value $\eta(t)$. 
In Theorem~\ref{thm: existence right-continuous generating function}, 
we will show how this special point can be identified based on the geometry of the hulls $\boldsymbol{K} = (K_t)_{t \geq 0}$.

Usually in the literature, more regularity (in particular left-continuity) is assumed for $\eta$. 
In contrast, our result is rather general. 
For instance, it generalizes~\cite[Theorem~5.22]{Beliaev:Conformal_maps_and_geometry}, which proves the statement for locally connected hulls and a continuous driving function.
Interestingly, left-continuous generating functions are also automatically right-continuous and have some nice topological properties: 

\begin{thm} \label{thm:coro_of_3}
Let $\boldsymbol{K} = (K_t)_{t \geq 0}$ be a family of locally growing hulls. 
Let $W \colon [0, \infty) \to \bR$ be the associated c\`adl\`ag driving function. 
Suppose the hulls $\boldsymbol{K}$ are generated by a left-continuous function
$\eta \colon [0, \infty) \to \overline{\bH}$. 
Then, the following hold.
\begin{enumerate}[label=\textnormal{(\arabic*):}, ref=\textnormal{(\arabic*)}]
\item 
For each $t \geq 0$, we have $g_t^{-1}(W(t-) + \ii y) \to \eta(t)$ as $y \to 0+$.

\smallskip

\item 
The function $\eta$ has unique right limits. 

\smallskip

\item 
For each $t \geq 0$, the set $\eta[0, t] \cup \bR$ is path-connected and uniformly locally path-connected.
\end{enumerate} 
In particular, $\eta \colon [0, \infty) \to \overline{\bH}$ is c\`agl\`ad, i.e., it is left-continuous function with unique right limits\footnote{The acronym ``c\`agl\`ad'' is short for \emph{continue \`a gauche, limites \`a droite}.}.
\end{thm}

Theorem~\ref{thm:coro_of_3} is a collection of Proposition~\ref{prop: left-cont implies right-cont} and Theorem~\ref{thm: locally path-connected}, which we prove in Section~\ref{subsec: left-continuous generating function}.

The converse of this result is false in multiple ways. 
On the one hand, a right-continuous generating function does not need to be left-continuous (Example~\ref{eg: Belyaev}), nor do its corresponding hulls need to be locally connected (Example~\ref{subsec: comb space}). 
On the other hand, because the left-continuity of the generating function implies local connectedness of the corresponding hulls, one might think that locally connected hulls are generated by left-continuous functions. 
This is not true in general, as alternative left limits can be swallowed (Example~\ref{eg: Belyaev}). 
Nonetheless, if the hulls have an empty interior, local connectedness of the hulls and left-continuity of the generating function are in fact equivalent.

\begin{thm}\label{thm:coro_of_4}
Let $\boldsymbol{K} = (K_t)_{t \geq 0}$ be a family of locally growing hulls.
Assume that for all $t \geq 0$, the set $K_t$ has an empty interior and its boundary $\bdry K_t$ is locally connected. 
Then, the hulls $\boldsymbol{K}$ are generated by a c\`agl\`ad function, i.e.,~a left-continuous function with unique right limits.
\end{thm}

Theorem~\ref{thm:coro_of_4} is an immediate consequence of Proposition~\ref{prop: empty interior locally connected equiv cadlag generated}, which we prove in Section~\ref{subsec: left-continuous generating function}.

\smallskip
\subsection{Examples and counterexamples}
\label{sec: counter examples}
Let us illustrate some possible behaviors for the growing hulls.

\begin{figure}[h!]
\includegraphics[width=.5\textwidth]{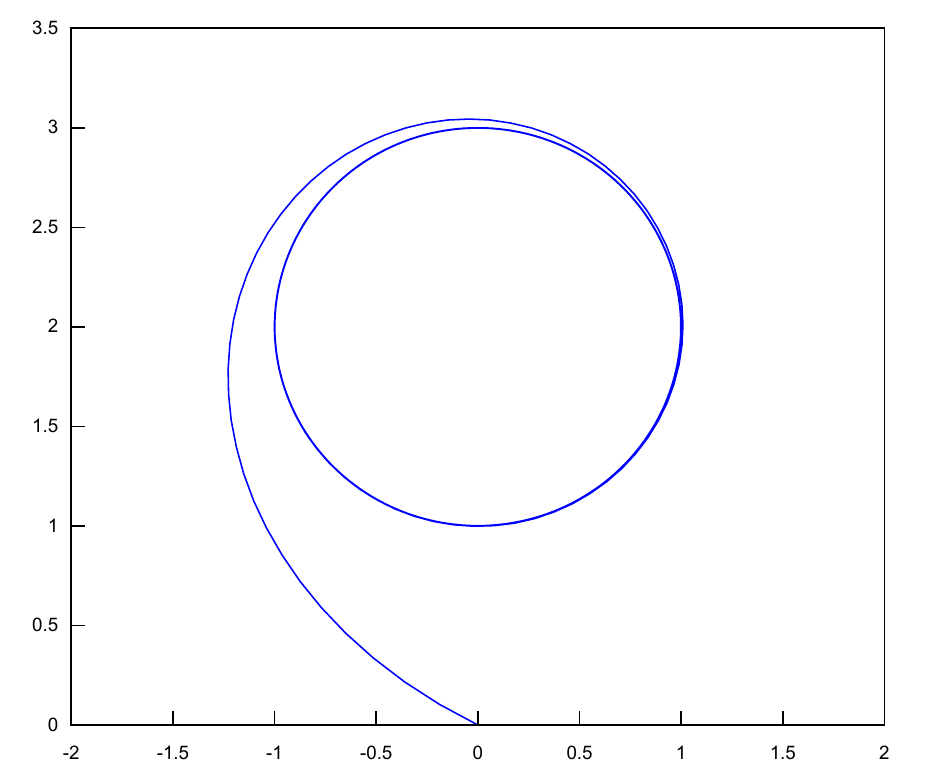}
\centering
\caption{\label{fig: spiral}
Illustration of spiraling hulls (Example~\ref{subsec: logarithmic spiral}).}
\end{figure}

\begin{example}[The logarithmic spiral: Continuous driver --- no generating function]\label{subsec: logarithmic spiral} 
See also Figure~\ref{fig: spiral}. 
In~\cite{Marshall-Rohde:The_loewner_differential_equation_and_slit_mappings}, 
Marshall~\&~Rohde constructed a logarithmic spiral 
spinning around a circle as an example of a Loewner chain which has a ($1/2$-H\"older) continuous driving function, but which is not generated by a c\`adl\`ag function. 
(See also~\cite{KNK:Exact_solutions_for_Loewner_evolutions, Lind:Sharp_condition_for_Loewner_equation_to_generate_slits} for related results.)
Consider the function
\begin{align*}
F\colon \bC \setminus \{0\} &\longrightarrow \bC
\\
z &\longmapsto \ii \Big((1 + |z| ) \frac{z}{|z|} + 2 \Big)
\end{align*}
and the function $
\gamma\colon [0, \infty) \setminus \{1\} \to \bC$ given by
\begin{align*}
\gamma(t) = 
\begin{cases}
F \big( (t-1) e^{\ii \log |t-1|} \big) , & t \leq 2 , \\ 
(t + 2)\ii , & t \geq 2 .
\end{cases}
\end{align*}
Both $\gamma \colon [0, 1) \to \bC$ and $\gamma: (1, 2] \to \bC$ are injective continuous functions wrapping infinitely often around $B(2\ii, 1)$, and their images do not intersect. 
We can define locally growing hulls $\boldsymbol{K} = (K_t)_{t \geq 0}$ by
\begin{align}
\label{eq: logarithmic spiral}
K_t := 
\begin{cases}
\overline{\gamma(0, t)} , & t < 1,
\vspace*{4pt} \\ 
\overline{\gamma(0, 1)} \cup \overline{B(2\ii, 1)} , & t = 1,
\vspace*{4pt} \\ 
\overline{\gamma(0, 1)} \cup \overline{B(2\ii, 1)} \cup \overline{\gamma(1, t)} , & t > 1 .
\end{cases}
\end{align}
The closed ball $\overline{B(2\ii, 1)}$ is the set of all \emph{grown} points\footnote{See Section~\ref{sec: Loewners Theorem} for a precise definition of grown points and other special points.}  
at time $t = 1$.
The results in~\cite[Section~5]{Marshall-Rohde:The_loewner_differential_equation_and_slit_mappings} imply that there 
exists a ($1/2$-H\"older) continuous function $W\colon [0, \infty) \to \bR$ 
which is the driving function of these hulls $\boldsymbol{K} = (K_t)_{t \geq 0}$. 
At the ``critical'' time $t=1$, the hulls fail to be locally connected and path-connected.
Nonetheless, this example can be adapted such that the spiral still closes and the mapping-in functions have (growing) radial limits that exist for all times (see Remark~\ref{rem:subtleties}):
\begin{align*}
\lim_{y \rightarrow 0+} g_t^{-1} ( W(t+) + \ii y ) .
\end{align*}
\end{example}

\begin{figure}[h!]
\includegraphics[width=.75\textwidth]{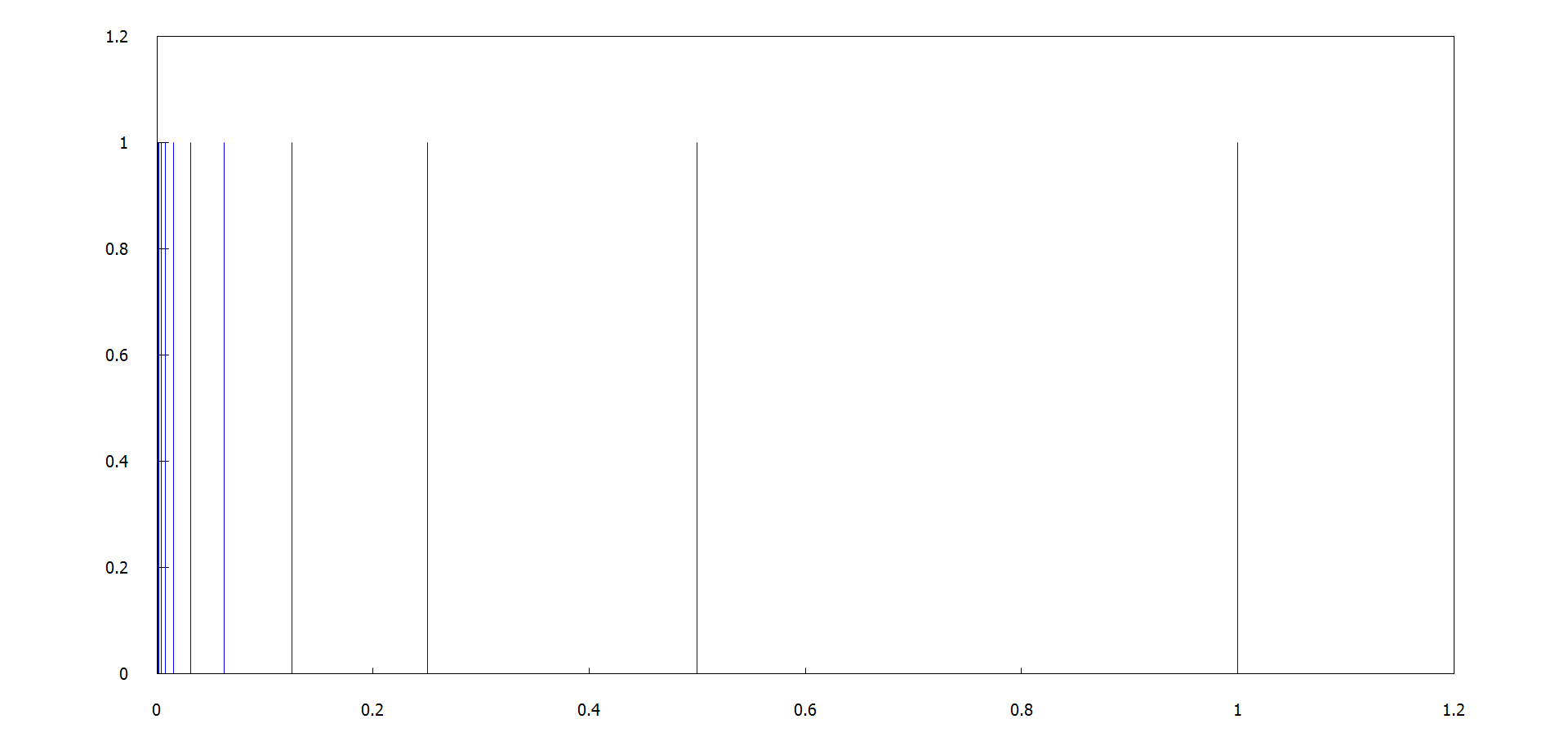}
\centering
\caption{\label{fig: comb}
Illustration of a path-connected but not locally connected comb (Example~\ref{subsec: comb space}).
}
\end{figure}

\begin{example}[The comb space: C\`adl\`ag driver --- non-locally connected frontier]\label{subsec: comb space}
See also Figure~\ref{fig: comb}. Consider 
\begin{align} 
\label{eq: comb space}
K := \{ \ii y \; | \; 0 \leq y \leq 1\} \cup \{ 2^{-n} + \ii y \; | \;  0 \leq y \leq 1, \, n \in \bZnn \} \; \subset \; \overline{\bH} .
\end{align}
The union $K \cup \bR$ of this comb with the real line is a path-connected set which is not locally connected.
It can be constructed via the graph of the function $\eta \colon [0, 4] \to \overline{\bH}$ 
(a function having no left limit at time $t=3$, which is however c\`adl\`ag elsewhere), 
\begin{align*}
\eta(t) :=
\begin{cases}
\ii t , & 0 \leq t < 1 , \\ 
1 + \ii (t-1) , & 1 \leq t < 2 , \\ 
2^{-n} + \ii \big( 2 + 2^n (t-3) \big) , 
& 3 - 2^{1-n} \leq t < 3 - 2^{-n} , \, n \in \bZpos , \\ 
\ii + \ii \big( t - 3 \big) , 
& 3 \leq t \leq 4 . 
\end{cases}
\end{align*}
The comb space~\eqref{eq: comb space} shows that, first,
for Loewner chains with c\`adl\`ag driving functions, local connectedness may fail, 
and second, 
not all Loewner chains with c\`adl\`ag driving functions are generated by c\`adl\`ag functions (see Theorem~\ref{thm: left continuity implies locally connectedness}). 
Indeed, it is not hard to check that the hulls
\begin{align*} 
K_t :=
\begin{cases}
\overline{\eta[0,t]} , & 0 \leq t < 3 , \\ 
K , & t = 3 ,  
\end{cases}
\end{align*}
are locally growing on $[0, 3] \ni t$, which shows that their driving function $W$ is actually c\`adl\`ag --- in particular it
has a unique left limit as $t \to 3-$, 
the image of the point $\ii$ under the conformal map $g_3 \colon \bH \setminus K \to \bH$.
\textnormal{(}This is in contrast to the fact that the function $\eta$ itself has no left limit as $t \to 3-$.\textnormal{)} 

The boundary $\bdry (\bH \setminus K_3)$ of the complement is not locally connected. 
In the terminology of Section~\ref{subsec: LC preli} and Section~\ref{sec: Loewners Theorem},  
the impression of the prime end of $\bH \setminus K_3$ containing the point $\ii \in \bH$ is the right side of the segment $\{ \ii y \; | \; 0 \leq y \leq 1\}$. 
This prime end may be both \emph{grown} and \emph{growing} at time $t=3$.
The only \emph{accessible} point from $\bH \setminus K_3$ in its impression is $\ii$.
\end{example}

\begin{figure}[h!]
\centering
\includegraphics[width=0.4\textwidth]{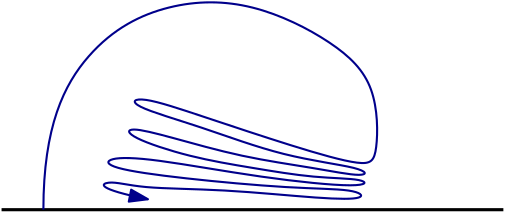}
\caption{\label{fig: Belyaev}
Illustration of Example~\ref{eg: Belyaev} (inspired by~\cite[Figure~5.10]{Beliaev:Conformal_maps_and_geometry}).
Here, the left-continuity fails for the generating function of the Loewner chain. 
}
\end{figure}

\begin{example}[Locally connected hulls --- discontinuous generating function]\label{eg: Belyaev}
Left-continuity of the generating function can also fail for Loewner chains with continuous driving functions. 
Figure~\ref{fig: Belyaev} 
shows an example with a right-continuous function generating
a Loewner chain, whose driving function is continuous, and for which the function itself is not left-continuous and not locally
connected but the boundaries of the associated domains $\bH \setminus K_t$ are still locally connected. 
(Compare with Proposition~\ref{prop: empty interior locally connected equiv cadlag generated}.)
\end{example}

Lastly, the right-continuity of the generating function can fail as well.  
A generating function can only fail to be right-continuous if there exists a prime end with more than one principal point (by Theorem~\ref{thm: Minimal regularity of generating functions}). 
These principal points are inaccessible (cf.~\eqref{eq: radial_limits_equiv}), so the hulls are not locally connected. 
Moreover, the generating function is not left-continuous either (by Proposition~\ref{prop: left-cont implies right-cont}).

\begin{example}[The double-comb space: Non-right(-and non-left-)continuous  generating function]\label{eg: double comb}
See also Figure~\ref{fig: double comb}. Consider 
\begin{align} \label{eq: double comb space}
\begin{split}
K := \ &\{ \ii y \; | \; 0 \leq y \leq 1\} \cup \{ \ii + x \; | \; 0 \leq x \leq 1\} 
\cup \{ 2^{-n} + \ii y \; | \;  0 \leq y \leq 3/4, \, n \in \bZnn \mathrm{ \ even} \} 
\\ 
&\cup \{ 2^{-n} + \ii y \; | \;  1/4 \leq y \leq 1, \, n \in \bZnn \mathrm{ \ odd} \}  
\; \subset \; \overline{\bH} .
\end{split}
\end{align}
The union $K \cup \bR$ of this double-comb with the real line is a path-connected set which is again not locally connected.
It can be constructed via the graph of the function $\eta \colon [0,3) \to \overline{\bH}$ 
(a function having no left limit at time $t=3$, which is however c\`adl\`ag elsewhere), 
\begin{align*}
\eta(t) :=
\begin{cases}
\ii t , & 0 \leq t < 1 , \\ 
\ii + (t-1) , & 1 \leq t < 2 , \\  
2^{-n} + 3 \ii \big( 1/2 + 2^{n - 2} (t-3) \big) , 
& 3 - 2^{1-n} \leq t < 3 - 2^{-n} , \, n \in \bZpos \mathrm{ \ even} , \\ 
2^{-n} + \ii \big( 7/4 + 3 \cdot 2^{n-2} (t-3) \big) , 
& 3 - 2^{1-n} \leq t < 3 - 2^{-n} , \, n \in \bZpos \mathrm{ \ odd} . 
\end{cases}
\end{align*}
\end{example}
Again, it is not hard to check that the hulls
\begin{align*} 
K_t :=
\begin{cases}
\overline{\eta[0,t]} , & 0 \leq t < 3 , \\ 
K , & t = 3 , 
\end{cases}
\end{align*}
are locally growing on $[0,3] \ni t$, which shows that their driving function $W$ is actually c\`adl\`ag --- in particular it
has a unique left limit as $t \to 3-$. 
This left limit is the image of the prime end associated with the segment $I := \{\ii y : 1/4 \leq y \leq 3/4 \}$. 
In fact, in the terminology of Section~\ref{subsec: LC preli} and Section~\ref{sec: Loewners Theorem}, 
all points in $I$ are \emph{principal} points of this (grown) prime end, and \emph{inaccessible}. 
This prime end may also be growing at time $t=3$. 
In that case, if it is growing at time $t=3$ and the hulls remain generated by a function, then this function grows out of $I$ by following the teeth of the comb. 
In particular, this function cannot be right-continuous at $t=3$. 

\begin{figure}[h!]
\centering
\includegraphics[width=.75\textwidth]{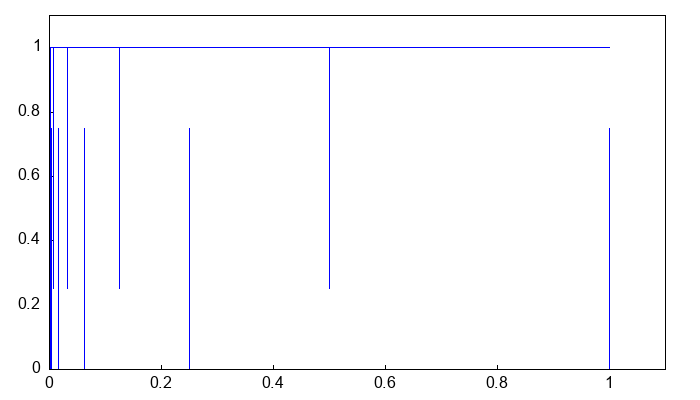}
\caption{\label{fig: double comb}
Illustration of a path-connected but not locally connected double-comb (Example~\ref{eg: double comb}). 
}
\end{figure}

\smallskip
{\bf Organization of this article.}

In Section~\ref{sec: preli}, we recall some relevant notions from complex analysis and Loewner theory, and gather terminology, also used above. 
We also derive several important properties of locally growing hulls (which mainly generalize the uniform case), highlighting the left- and right-continuity in the definitions.  

In Section~\ref{sec: Loewners Theorem}, we prove Loewner's theorem for (dis)continuous driving functions, 
i.e.,~the equivalence of locally growing hulls and their driving functions (Theorems~\ref{thm: loewner intro}~\&~\ref{thm: Locally growing hulls solve LE intro}). 
We also present a result characterizing the driving function in terms of Equation~\eqref{eq: shrink to points} (see Proposition~\ref{prop: first construction driving measure}).

In Section~\ref{sec: boundary behavior of locally groing hulls}, we study what is added to our hulls at a given time $t$. 
In Theorem~\ref{thm: decomposition added stuff to hull} (Section~\ref{subsec:thm14}), we find that this comprises a ``bubble'' $B_t$ of all swallowed points and a (potentially empty) compact and connected set $P_t$ of boundary points.
Theorem~\ref{thm: decomposition added stuff to hull intro} then follows. 
In Proposition~\ref{prop: A} (Section~\ref{subsec:critical_times}), we show a technical result that can often conclude that a given behavior for the hulls occurs at a critical time, i.e.,~the infimum of all times witnessing that behavior. 
We apply it to investigate path-connectedness of the hulls (see Proposition~\ref{prop: not path-connected at stopping time}), which will be needed in Section~\ref{sec: gen by fct}.

In the final Section~\ref{sec: gen by fct}, we consider topological properties of Loewner hulls that are generated by a function. 
Firstly, we show that a generating function $\eta$ exists exactly when the union of the hulls and the real line is path-connected (Theorem~\ref{thm: equivalence generated by a function}). 
Theorem~\ref{thm: generated by a fct intro} follows from this. 
Using the context of prime ends, we also give a characterization for the generating function as the radial limit~\eqref{eq: canonical eta} 
in Section~\ref{subsec: right-continuous generating functions} (see Theorem~\ref{thm: existence right-continuous generating function}). 
As a result, we show that in this case $\eta$ has unique right limits (Theorem~\ref{thm: Minimal regularity of generating functions}). 
These results imply Theorem~\ref{thm:coro_of_2}.
Moreover, in Section~\ref{subsec: left-continuous generating function} 
we prove that hulls generated by a left-continuous function are path-connected and uniformly locally (path-)connected (Theorems~\ref{thm: left continuity implies locally connectedness}~\&~\ref{thm: locally path-connected}). 
In particular, such a left-continuous generating function automatically has unique right limits (see Proposition~\ref{prop: left-cont implies right-cont}).  
These results imply Theorem~\ref{thm:coro_of_3}. 
(See~\cite{Chen-Rohde:SLE_driven_by_symmetric_stable_processes, Peltola-Schreuder:Loewner_traces_driven_by_Levy_processes} for related results.)
In the final Proposition~\ref{prop: empty interior locally connected equiv cadlag generated}, we find that if the hulls have empty interior, then the reverse holds as well: 
Locally connected hulls are generated by a left-continuous function.

\smallskip
{\bf Acknowledgements.}

We would like to thank Lukas Schoug and Yizheng Yuan for inspiring discussions and encouragement. 

This material is part of a project that has received funding from the  European Research Council (ERC) under the European Union's Horizon 2020 research and innovation programme (101042460): 
ERC Starting grant ``Interplay of structures in conformal and universal random geometry'' (ISCoURaGe) 
and from the Academy of Finland grant number 340461 ``Conformal invariance in planar random geometry.''

Both authors are supported by the Academy of Finland Centre of Excellence Programme grant number 346315 ``Finnish centre of excellence in Randomness and STructures (FiRST)'' 
E.P.~is also supported by the Deutsche Forschungsgemeinschaft (DFG, German Research Foundation) project number 390534769 ``Matter and Light for Quantum Computing (ML4Q).

Additionally, A.S. was funded by the Cantab Capital Institute for the Mathematics of Information 
and by the Deutsche Forschungsgemeinschaft (DFG,
German Research Foundation) under Germany's Excellence Strategy - EXC-2047/1 - 390685813.

\bigskip{}
\section{Hulls in the upper half-plane}
\label{sec: preli}
The purpose of this section is to collect notation, terminology, and basic facts. 
In Section~\ref{subsec: boundary behaviour of a conformal map}, we recall some relevant notions from complex analysis. 
Then in Section~\ref{subsec: LC preli}, we collect basic notions on Loewner chains. 
In particular, we discuss local growth, which is equivalent to the existence of a driving function.

\smallskip
\subsection{Background on complex analysis}
\label{subsec: boundary behaviour of a conformal map}

Here, we briefly discuss boundary behavior of conformal maps $\confmap \colon \bH \to \domain$
onto a simply connected domain $\domain \subsetneq \smash{\hat{\bC}}$. 
For extensive literature on this rather delicate subject, see~\cite[Chapter~2]{Pommerenke:Boundary_behaviour_of_conformal_maps}
and~\cite[Chapter~2]{Beliaev:Conformal_maps_and_geometry}. 
Let us first recall a few basic notions:
\begin{itemize}[leftmargin=*]
\item
A \emph{crosscut} in $\domain$ is an open Jordan arc $S \subset \domain$ which touches the boundary
at its endpoints $a,b \in \bdry \domain$ (which may coincide): $\overline{S} = S \cup \{ a,b \} \subset \overline{\domain}$. 

\smallskip

\item
A \emph{null-chain} $(S_n)_{n \in \bZnn}$ is a sequence of nested crosscuts such that for all $n$, we have 
$S_n \cap S_{n+1} = \emptyset$,
the crosscut $S_n$ separates $S_0$ and $S_{n+1}$, and $\diam(S_n) \to 0$ as $n \to \infty$.

\smallskip

\item 
Two null-chains $(S_n)_{n \in \bZnn}$ and $(S_n')_{n \in \bZnn}$ are 
equivalent if and only if
for each sufficiently large $m$, the crosscut $S_m$ (resp.~$S_m'$) separates all but finitely many $S_n'$ from $S_{m-1}$ (resp.~$S_n$ from $S_{m-1}'$).

\smallskip

\item
A \emph{prime end} $\xi$ of $\domain$ is an equivalence class of null-chains.

\smallskip

\item 
A \emph{principal point} of a prime end $\xi$ is a point $z \in \bdry D$ where 
$\xi$ can be represented by a null-chain $(C_n)_{n \in \bZnn}$ 
with $C_n \subset B(z, \varepsilon)$ for all $\varepsilon > 0$ and $n > n_0(\varepsilon)$ sufficiently large.

\smallskip

\item The \emph{impression} of a prime end $\xi$ of $\domain$ is defined as 
\begin{align*}
I(\xi) := \bigcap_{n \in \bZnn} \overline{\mathrm{int}_{\mathrm{in}} (S_n)} ,
\end{align*}
where $\mathrm{int}_{\mathrm{in}} (S_n)$ is 
the interior of the connected component of $\domain \setminus S_n$ not containing $S_0$.
Note that $I(\xi)$ is a non-empty compact connected set, whence it is either a single point or a continuum. If $I(\xi)$ is a single point, then it is a boundary point of $\domain$ and we say that the prime end $\xi$ is \emph{degenerate}. 

\smallskip

\item A set $A \subset \bC$ is (uniformly) \emph{locally connected} if for every $\varepsilon > 0$ there exists $\delta = \delta_\varepsilon > 0$ such that, for any pair of points $z,w \in A$ such that $|z-w| < \delta$, there exists a closed connected set $S$ such that $z,w \in S \subset A$ and $\diam(S) < \varepsilon$.
By~\cite[Lemma~6.7]{Sagan:Space_filling_curves}, a sufficient condition for this 
is that $A$ is compact, connected, and locally connected at every point $z \in A$, that is, 
for every $z \in A$ and $\varepsilon > 0$, there exists a radius $r_{z,\varepsilon} > 0$ such that for every $w \in A \cap B(z,r_{z,\varepsilon})$, there exists a closed connected set $S$ such that $z,w \in S \subset A \cap B(z,\varepsilon)$.

\smallskip

\item
A set $A \subset \bC$ is said to be (uniformly) \emph{locally path-connected} 
if for every $\varepsilon > 0$ there exists $\delta = \delta_\varepsilon  > 0$ such that, for any pair of points $z,w \in A$ such that $|z-w| < \delta$, there exists a continuous path $\gamma$ connecting $z$ and $w$ in $A \cap B(z,\varepsilon) \cap B(w,\varepsilon)$. 
\cite[Theorem~6.7.2]{Sagan:Space_filling_curves} implies that if $A$ is compact, connected, and locally connected,
then $A$ is locally path-connected\footnote{According to~\cite{Sagan:Space_filling_curves}, this was proven by Hahn;
see also the Mazurkiewicz-Moore-Menger theorem~\cite[page~254]{Kuratowski:Topology}.}.

\smallskip

\item Any connected and locally path-connected set is path-connected.
\end{itemize}

Carath\'eodory's theorem (see~\cite[Chapter~2]{Pommerenke:Boundary_behaviour_of_conformal_maps})
implies that a conformal map $\confmap \colon \bH \to \domain$ 
extends to a homeomorphism $\overline{\bH} \to \overline{\domain}$ if and only if $\bdry \domain$ is a Jordan curve. 
Also, $\confmap$ has a continuous extension to $\overline{\bH}$
if and only if $\bdry \domain$ is locally connected, which
is also equivalent to $\bdry \domain$ being a continuous curve, 
but perhaps not an injection (in which case $\confmap$ has no inverse on $\bdry \domain$).
In any case, the conformal map $\confmap$ always induces a one-to-one correspondence between 
the boundary points of $\bH$ (also including $\infty \in \bdry \bH \subset \smash{\hat{\bC}}$) 
and the prime ends $\xi$ of $\domain$
(cf.~\cite[Theorem~2.15]{Pommerenke:Boundary_behaviour_of_conformal_maps}). 
We write $\xi = \smash{\invbreve{\confmap}}(x) \in \smash{\invbreve{\bdry}} \domain$ for the prime end $\xi$ corresponding to the boundary point $x \in \bdry \bH$, and $\smash{\invbreve{\bdry}} \domain = \smash{\invbreve{\confmap}}(\bdry \bH)$ for the boundary of $\domain$ comprising its prime ends. 
In particular, for any null-chain $\smash{(S_n)_{n \in \bZnn}}$ representing the prime end $\xi$ in $\domain$, its inverse image $\smash{(\confmap^{-1}(S_n))_{n \in \bZnn}}$ is a null-chain in $\bH$ that shrinks to $x = \smash{\invbreve{\confmap}}\smash{{}^{-1}}(\xi)$:
\begin{align*}
\{x\} = \bigcap_{n \in \bZnn} \overline{\mathrm{int}_{\mathrm{in}} (\confmap^{-1}(S_n))} .
\end{align*}

We say that a prime end $\xi$ is \emph{accessible} if,
for any interior point $w \in \domain$, there exists a Jordan arc $J$ in $\overline{\domain}$ starting at $w$ which lies entirely in $\domain$ except at its endpoint in $I(\xi) \cap \bdry \domain$. 
In this case, we say that $J$ \emph{accesses} the prime end $\xi$, and the endpoint of $J$ is an \emph{accessible point}.  
By~\cite[Proposition~2.14]{Pommerenke:Boundary_behaviour_of_conformal_maps}, $\confmap^{-1} (J)$ is then a curve in $\overline{\bH}$ which lies entirely in $\bH$ except at its endpoint in $\bdry \bH$.
Furthermore, if $J_1$ and $J_2$ are two Jordan arcs accessing two distinct prime ends of $\domain$,
then the curves $\confmap^{-1} (J_1)$ and $\confmap^{-1} (J_2)$ also have distinct endpoints in $\bdry \bH$.
(Here, it is crucial that the image domain of $\confmap^{-1}$ is nice, e.g., $\bH$.)

For any boundary point $x \in \bdry \bH$, the limit
\begin{align} \label{eq: unrestricted limit}
\confmap(x) := \lim_{z \to x} \confmap(z) \, \in \, \bdry \domain \qquad \textnormal{along} \quad z \in \bH ,
\end{align}
is called the \emph{unrestricted} limit of $\confmap$ at $x$. 
By~\cite[Corollary~2.17 and Exercise~2.5.5]{Pommerenke:Boundary_behaviour_of_conformal_maps},
if the limit~\eqref{eq: unrestricted limit} exists, then the prime end $\smash{\xi = \invbreve{\confmap}(x)}$ is degenerate and accessible, and we have $I(\xi) = \{ \confmap(x) \}$.

Conversely, if $J \colon [0,1) \to \domain$ is a Jordan arc accessing a prime end $\xi$ of $\domain$, then the limit of $\confmap$ exists along the curve $L := \confmap^{-1} \circ J \colon [0,1) \to \bH$ by~\cite[Corollary~2.17 and Exercise~2.5.5]{Pommerenke:Boundary_behaviour_of_conformal_maps}:
\begin{align*} 
J(1) = \lim_{s \to 1-} \confmap(L(s)) \, \in \, \bdry \domain \qquad \textnormal{along} \quad L[0,1) \subset \bH ,
\end{align*}
which is also equivalent to the existence of a \emph{radial limit} of $\confmap$ at $\xi$~\cite[Corollary~2.17(i)]{Pommerenke:Boundary_behaviour_of_conformal_maps}. 
However, this does not guarantee the existence of the unrestricted limit~\eqref{eq: unrestricted limit}. In fact, the following are equivalent:

\begin{align}\label{eq: radial_limits_equiv}
\begin{split}
\textnormal{The map $\confmap \colon \bH \to \domain$ has a radial limit $a \in \bdry \domain$ at} & 
\textnormal{ $x \in \bR$, e.g., along $\smash{\underset{y \to 0+}{\lim} \, \confmap(x + \ii y) = a}$.}
\\[.5em]
\Longleftrightarrow & 
\\[.5em]
\textnormal{The prime end $\xi$ associated with $x \in \bR$ has a} & 
\textnormal{ unique principal point, which is $a$.}
\\[.5em]
\Longleftrightarrow & 
\\[.5em]
\textnormal{The prime end $\xi$ associated with $x \in \bR$ has} & 
\textnormal{ an accessible point (which is $a$).}
\end{split}
\end{align}

\smallskip
\subsection{Locally growing hulls and their mapping-out functions}
\label{subsec: LC preli}

We call a closed subset $K \subset \overline{\bH}$ a \emph{hull} if $K$ is bounded for the Euclidean metric and $\bH \setminus K$ is simply connected.
We write $\bdry K \subset \overline{\bH}$ and $\mathrm{int} (K) \subset \overline{\bH}$ 
respectively for the boundary and interior of the hull in the relative topology,
and $\bdry_{\mathrm{in}} K := \bdry K \cap \bH \subset \bH$. 
For each $\varepsilon > 0$ and a subset $A \subset \overline{\bH}$, we denote by
\begin{align*}
A^{\varepsilon} := \bigcup_{z \in A} \overline{B(z, \varepsilon)} \cap \overline{\bH}
\end{align*}
the \emph{$\varepsilon$-thickening} of $A$. 
The Riemann mapping theorem implies that for each hull $K$, 
there exists a unique conformal bijection 
\begin{align*} 
g_K \colon \bH \setminus K \to \bH ,
\qquad \qquad
g_K(z) = z + \sum_{n = 1}^{\infty} a_n(K) \, z^{-n} , \qquad |z| \to \infty ,
\end{align*}
with real coefficients $a_n(K)$~\cite[Lemma~4.1]{Kemppainen:SLE_book}. We call $g_K$ the \emph{mapping-out function} of $K$ (normalized at $\infty$). 
The first coefficient $\mathrm{hcap}(K) := a_1(K) \geq 0$ in the expansion~\eqref{eq: mof Laurent exp} is 
always non-negative, and we call it the \emph{half-plane capacity} 
of the hull $K$. 
Intuitively, the half-plane capacity describes the size of $K$ as seen from $\infty$,
and it is an increasing function in the sense that 
$\mathrm{hcap}(K) \leq \mathrm{hcap}(K')$ for $K \subset K'$.

\begin{defn}
\label{def: growth}
Let $\boldsymbol{K} = (K_t)_{t \geq 0}$ be a family of hulls.
We say that $\boldsymbol{K}$ is
\begin{itemize}
\item \emph{growing} if $K_s \subset K_t$ for all $s \leq t;$ 

\smallskip 

\item \emph{strictly growing} if $\boldsymbol{K}$ is growing, $K_0 \subset \bR$, 
and $K_s \cap \bH \subsetneq K_t \cap \bH$ holds for all $0 \leq s < t;$

\smallskip 

\item \emph{(strictly) left-continuously growing} if $\boldsymbol{K}$ is (strictly) growing and,
for every $t \geq 0$ and $\varepsilon > 0$, there exists $\delta = \delta(\varepsilon, t) \in (0, t)$ 
such that $\bdry(\bH \setminus K_{t}) \subset K_{t-\delta}^\varepsilon \cup \bR^\varepsilon;$

\smallskip 

\item \emph{(strictly) right-continuously growing} if $\boldsymbol{K}$ is (strictly) growing and,
for every $t \geq 0$ and $\varepsilon > 0$, there exists $\delta = \delta(\varepsilon, t) > 0$ 
such that $K_{t + \delta} \subset K_t^\varepsilon \cup \bR^\varepsilon;$

\smallskip 

\item \emph{(strictly) continuously growing} if $\boldsymbol{K}$ is both (strictly) left- and right-continuously growing.
\end{itemize}
\end{defn}

In particular, a family $\boldsymbol{K}$ of growing hulls is continuously growing if and only if their half-plane capacities are continuously increasing.
The following properties were proven in~\cite[Theorem~4.47~\&~Corollary~4.48]{Schreuder:PhD}:

\begin{lem}
\label{lem: hcap and growing}
Let $\boldsymbol{K}$ be a growing family of hulls. Define $\phi(t) := \mathrm{hcap}(K_t)$ for $t \geq 0$.
Then, 
\begin{itemize}
\item $\phi$ is increasing;

\smallskip 

\item $\phi$ is strictly increasing if and only if $\boldsymbol{K}$ is strictly growing;

\smallskip 

\item $\phi$ is left-continuous if and only if $\boldsymbol{K}$ is left-continuously growing;

\smallskip 

\item $\phi$ is right-continuous if and only if $\boldsymbol{K}$ is right-continuously growing.
\end{itemize}
Moreover, the following hold.
\begin{enumerate}[label=\textnormal{(\arabic*):}, ref=\textnormal{(\arabic*)}]
\item If $\boldsymbol{K}$ is left-continuously growing at time $t$, then $\overline{K_t \cap \bH}$ is the smallest hull containing $\bigcup_{s < t} K_s$. 

\smallskip

\item If $\boldsymbol{K}$ is right-continuously growing at time $t$, then $\bigcap_{s > t} K_s \subset K_t \cup \bR$.
\end{enumerate}
\end{lem}

Recalling Definition~\ref{def: local growth}, let us observe that local growth implies continuous growth. 

\begin{lem}
\label{lemma: local growth => continuous growth}
Let $\boldsymbol{K}$ be a family of hulls.
\begin{enumerate}[label=\textnormal{(\arabic*):}, ref=\textnormal{(\arabic*)}]
\item If $\boldsymbol{K}$ is left-locally growing, then it is left-continuously growing. 

\smallskip

\item If $\boldsymbol{K}$ is right-locally growing, then it is right-continuously growing.
\end{enumerate} 
\end{lem}

\begin{proof}
Fix $t \geq 0$ and $\varepsilon > 0$.
We prove the claim for right-continuous growth; the left-continuous case is analogous.
By the right-local growth, there exists $\delta > 0$ and a crosscut $S_{\delta}^{\mathrm{out}} \subset \bH \setminus K_t$ 
with $\mathrm{diam}(S_{\delta}^{\mathrm{out}}) < \varepsilon$ separating $K_{t + \delta} \setminus K_t$ from $\infty$ in $\bH \setminus K_t$.
Hence, there exists a point $z \in \bdry (\bH \setminus K_t) \subset K_t \cup \bR$ with $S_{\delta}^{\mathrm{out}} \subset \overline{B(z, \varepsilon)}$.
This implies that $K_{t+\delta} \subset K_t^{\varepsilon} \cup \bR^{\varepsilon}$, as required by Definition~\ref{def: growth}. 
\end{proof}

\begin{rem}
\label{remark: Left local null chains}
Regarding the definition of left-local growth, there is an important subtlety. 
Because $S_{\delta}^{\mathrm{in}}$ is a crosscut in $\bH \setminus K_{t-\delta}$ separating $K_t \setminus K_{t - \delta}$ from $\infty$, it is also a crosscut in $\bH \setminus K_t$; see Figure~\ref{fig: left-local growth}. 
\end{rem}

\begin{rem} 
\label{remark: uniform local growth}
In the literature, 
e.g.,~\cite[Theorem~2.6]{LSW:Brownian_intersection_exponents1} and~\cite[Chapter~4]{Kemppainen:SLE_book}, 
one usually considers~\eqref{eq: LE} with continuous driving functions, in 
which case the local growth property reads as follows:  
for every $\varepsilon > 0$ and $T > 0$ there exists $\delta = \delta(\varepsilon, T) > 0$ such that for each $t \in [0, T]$, 
there exists a crosscut $S_{\delta} \subset \bH \setminus K_t$
with $\diam(S_{\delta}) < \varepsilon$ separating $K_{t + \delta} \setminus K_t$ from $\infty$ in $\bH \setminus K_t$. 
In particular, $\delta$ is uniform over~$t$.
However, such uniformity fails for discontinuous driving functions, for instance when $K_t = \gamma[0,t]$ for a continuous curve $\gamma$ that crosses itself. 
In this case, the conditions involving $S_{\delta}^{\mathrm{out}}$ and $S_{\delta}^{\mathrm{in}}$ still hold.
\end{rem}

In the rest of this section, we gather crucial properties of locally growing hulls.

\begin{lem} \label{lem: some markov property}
Let $\boldsymbol{K}$ be a family of hulls and let $g_t : \bH \setminus K_t \rightarrow \bH$ denote their mapping-out functions.
Fix $t \geq 0$. Then, the sets
\begin{align}\label{eq:Ktilde}
\big( \tilde{K}^t_{s} \big)_{s \geq 0}
:= & \; \big( \overline{g_t(K_{t + s} \setminus K_t)} \big)_{s \geq 0} , \qquad s \geq 0 ,
\end{align}
are hulls whose mapping-out functions are given by $\smash{\tilde{g}^t_{s}} = g_{t+s} \circ g_t^{-1}$.
\end{lem}

\begin{proof}
This is a direct computation, using the uniqueness of the expansion~\eqref{eq: mof Laurent exp}. 
Note that the boundedness of $\tilde{K}^t_{s}$ is a direct consequence of~\cite[Lemma~4.5]{Kemppainen:SLE_book}
and because $K_t$ and $K_{t+s}$ are bounded.
\end{proof}

The gist of our refined Definition~\ref{def: local growth} of local growth is the following property. 
\clearpage

\begin{lem}
\label{lemma: ca receding balls}
Let $\boldsymbol{K}$ be a family of hulls. 
\begin{enumerate}[label=\textnormal{(\arabic*):}, ref=\textnormal{(\arabic*)}]
\item\label{item: ca receding balls left}  
If $\boldsymbol{K}$ is left-locally growing at time $t$, then for every $\varepsilon > 0$, there exists $\delta = \delta(\varepsilon, t) > 0$ with
$\mathrm{diam}(\tilde{K}^{t-s}_s) \le \varepsilon$ for all $s \in [0, \delta]$.

\smallskip

\item\label{item: ca receding balls right}  
If $\boldsymbol{K}$ is right-locally growing at time $t$, then for every $\varepsilon > 0$, there exists $\delta = \delta(\varepsilon, t) > 0$ with
$\mathrm{diam}(\tilde{K}^{t}_s) \le \varepsilon$ for all $s \in [0, \delta]$.

\end{enumerate} 
\end{lem}

\begin{proof}
Fix $t \geq 0$ and $\varepsilon > 0$. Because $K_{t}$ is compact, there exists $x_0 \in \bR$ and $R = R_t > 0$ such that $K_{t} \subset B(x_0, R)$. 
Choose $\epsilon' \in (0, 1)$ small enough such that 
\begin{align*}
\frac{6 \pi R}{\sqrt{\log ( 1 / \epsilon' ) }} < \varepsilon.
\end{align*}
We prove the claim for left-local growth; the right-local case is analogous (and slightly easier).
By the left local growth, there exists $\delta > 0$ and a crosscut $S_{\delta}^{\mathrm{in}} \subset \bH \setminus K_{t-\delta}$ separating $K_{t} \setminus K_{t - \delta}$ 
from $\infty$ in $\bH \setminus K_{t-\delta}$ with $\mathrm{diam}(S_{\delta}^{\mathrm{in}}) < \epsilon'$. 
Hence, there exists $z_0 \in \bC$ such that 
\begin{align}
\label{eq: in crosscut contained in a small ball}
S_{\delta}^{\mathrm{in}} \subset \overline{B(z_0, \epsilon')}.
\end{align}  
Since $K_{t-\delta} \subset K_{t-s} \subset K_{t} \subset B(x_0, R)$ for all $s \in [0, \delta]$ (due to the growth in time)~\cite[Lemma~4.5]{Kemppainen:SLE_book} 
implies that $g_{t-s}((\bH \setminus K_{t-s}) \cap B(x_0, 2R)) \subset B(x_0, 3R)$ for all $s \in [0, \delta]$. 
Define for $r > 0$ and $s \in [0, \delta]$ the set $C_{r, s} := \{ |z - z_0| = r \} \cap (\bH \setminus K_{t-s}) \cap B(x_0, 2R)$.
Then, by Wolff's lemma (\cite[Lemma 4.6]{Kemppainen:SLE_book})
\begin{align}
\label{eq: helper crosscut diameter bound in}
\inf_{\epsilon' < r < \sqrt{\epsilon'}} \mathrm{length} \big(g_{t-s}(C_{r, s})\big) \leq \frac{6 \pi R}{\sqrt{\log ( 1 / \epsilon' ) }} < \varepsilon , \qquad s \in [0, \delta] .
\end{align}

Furthermore $S_{\delta}^{\mathrm{in}}$ is a crosscut both in $\bH \setminus K_{t-\delta}$ and in $\bH \setminus K_t$ (recallling Remark~\ref{remark: Left local null chains}). 
Hence, $S_{\delta}^{\mathrm{in}}$ is a crosscut in $\bH \setminus K_{t-s}$ for all $s \in [0, \delta]$, by the growth of the hulls $\boldsymbol{K}$. 
Therefore,~(\ref{eq: in crosscut contained in a small ball},~\ref{eq: helper crosscut diameter bound in}) imply that  
\begin{align*}
\mathrm{diam}\big(\tilde{K}^{t-s}_s \big) 
\leq \diam \big( g_{t-s} (S_{\delta}^{\mathrm{in}})\big) 
\leq \inf_{\epsilon' < r < \sqrt{\epsilon'}} \mathrm{length}\big(g_{t-s}(C(r, s))\big) 
< \varepsilon , \qquad s \in [0, \delta].
\end{align*}
This proves Item~\ref{item: ca receding balls left}. 
We leave the very similar details of the proof of Item~\ref{item: ca receding balls right} to interested readers.
\end{proof}

Let us record the following consequence of the proof of Lemma~\ref{lemma: ca receding balls}.

\begin{cor}
\label{cor: conformal distortion}
Let $\boldsymbol{K} = (K_t)_{t \geq 0}$ be a family of locally growing hulls.
Fix $t \geq 0$. Let $x_0 \in \bR$ and $R > 0$ be such that $K_t \subset B(x_0, R)$.
Then, for all $\varepsilon \in (0, 1)$, the following hold.
\begin{enumerate}[label=\textnormal{(\arabic*):}, ref=\textnormal{(\arabic*)}]
\item\label{item: diameter control on right crosscut} 
If there exists a crosscut $S_{\delta}^{\mathrm{out}}$ separating $K_{t + \delta} \setminus K_t$ from $\infty$ in $\bH \setminus K_t$ with $\diam(S_{\delta}^{\mathrm{out}}) < \varepsilon$, then 
\begin{align*}
\mathrm{diam}\big(\tilde{K}^{t}_s \big)   
\leq \diam \big( g_{t} (S_{\delta}^{\mathrm{out}})\big)
\leq r_0(\varepsilon) := \frac{6 \pi R}{\sqrt{\log ( 1 / \varepsilon ) }} , \qquad s \in [0, \delta] .
\end{align*}

\item 
If there exists a crosscut $S_{\delta}^{\mathrm{in}}$ separating $K_{t} \setminus K_{t - \delta}$ from $\infty$ in $\bH \setminus K_{t - \delta}$ with $\diam(S_{\delta}^{\mathrm{in}}) < \varepsilon$, 
then 
\begin{align*}
\mathrm{diam} \big(\tilde{K}^{t-s}_s \big) 
\leq \diam ( g_{t-s} (S_{\delta}^{\mathrm{in}}))
\leq r_0(\varepsilon) , \qquad s \in [0, \delta] ,
\end{align*}
and each $S_{\delta}^{\mathrm{in}}$ is a crosscut in $\bH \setminus K_{t-s}$.
\end{enumerate}
\end{cor}

Lemma~\ref{lemma: ca receding balls} also implies that the mapping-out functions of locally growing hulls are continuous in time. 

\begin{cor}
\label{cor: gt cont in time}
Let $\boldsymbol{K} = (K_t)_{t \geq 0}$ be a family of locally growing hulls, 
and $g_t\colon \bH \setminus K_t \to \bH$ their mapping-out functions. 
For every $z \in \bH$, the map $t \mapsto g_t(z)$ is continuous on $[0,\tau(z)) \ni t$. 
\end{cor}

\begin{proof}
The right-continuity readily follows from~\cite[Lemma~4.5]{Kemppainen:SLE_book} and Lemma~\ref{lemma: ca receding balls}\ref{item: ca receding balls right}:
\begin{align*}
| g_{t+s}(z) - g_t(z) | 
= | \tilde{g}^{t}_{s}(g_t(z)) - g_t(z) | 
\leq \sup_{w \in \bH \setminus \tilde{K}^{t}_{s}} | \tilde{g}^{t}_{s}(w) - w | \leq 5 \, \diam (\tilde{K}^{t}_{s}) \; \overset{s \to 0+}{\longrightarrow} \; 0 .
\end{align*}
The left-continuity can be shown similarly 
(using Lemma~\ref{lemma: ca receding balls}\ref{item: ca receding balls left}). 
\end{proof}


\bigskip{}
\section{Loewner's theorem}
\label{sec: Loewners Theorem}
In this section, we prove the following two variants of Loewner's classical result.

\ThmLoewner*

\begin{rem}
\label{rem: minimal regularity}
Because locally growing hulls are continuously growing, these hulls geometrically have unique right and left limits \textnormal{(}see Lemma~\ref{lem: hcap and growing}\textnormal{)}.
Consequently, a driving function, if it exists, inherits this property. 
Hence, the minimal regularity of a driving function is being c\`adl\`ag. 
\end{rem}

The locally growing hulls can be obtained from their driving function, because their mapping-out functions solve the Loewner equation. 
We formulate this using the right derivative at time $t$, denoted $\pderr{t}$. 

\ThmLoewnerEquation*

Observe that for each fixed $z \in \overline{\bH}$, 
the blow-up time $\tau(z)$ of~\eqref{eq: LE} is the first time when the given point $z$ satisfies one of the following mutually exclusive properties:~it~is 
\begin{itemize}
\item either \emph{swallowed} by the growing hulls at time $\tau(z)$, i.e., we have
\begin{align*}
z \in \mathrm{int} ( K_{\tau(z)} ) \setminus \bigcup_{s < \tau(z)} K_s ,
\end{align*}
in which case we necessarily have
$\smash{\underset{t \to \tau(z)-}{\liminf}} \; |g_t(z) - W(t)| = 0$; 

\medskip

\item or \emph{hit} by the growing hulls at time $\tau(z)$, i.e., we have 
\begin{align*}
z \in ( \bdry K_{\tau(z)} ) \setminus \underset{s < \tau(z)}{\bigcup} K_s 
\qquad \textnormal{and} \qquad 
\liminf_{t \to \tau(z)-} \; |g_t(z) - W(t)| = 0 ;
\end{align*}

\medskip

\item or a \emph{branch point} at time $\tau(z)$, i.e., we have 
\begin{align*}
z \in ( \bdry K_{\tau(z)} ) \setminus \bigcup_{s < \tau(z)} K_s
\qquad \textnormal{but} \qquad 
\liminf_{t \to \tau(z)-} \; |g_t(z) - W(t)| > 0 ,
\end{align*}
in which case $W$ has a jump at time $\tau(z)$ and $g_{\tau(z)}(z) = W(\tau(z)+)$.
\end{itemize}
Note that swallowed points are never accessible from $\bH \setminus K_t$, while hit and branch points can be accessible or inaccessible from $\bH \setminus K_t$.

Conversely, one can retrieve the driving function from the hulls. 
By Definition~\ref{def: local growth} of local growth and by~\cite[Theorem~2.15]{Pommerenke:Boundary_behaviour_of_conformal_maps}, 
at each fixed time $t$, any countable subsequences 
$\smash{(S_{\delta_n}^{\mathrm{in}})_{n \in \bZnn}}$ 
and $\smash{(S_{\delta_n}^{\mathrm{out}})_{n \in \bZnn}}$ 
of these null-chains with $\delta_n = \delta_n(t) \to 0+$ as $n \to \infty$
represent two unique prime ends in $\bH \setminus K_t$:
\begin{align*}
\smash{\invbreve{f}_t}(W(t-)) = \grown_t \in \smash{\invbreve{\bdry}} (\bH \setminus K_t) 
\qquad \textnormal{and} \qquad
\smash{\invbreve{f}_t}(W(t)) = \growing_t  \in \smash{\invbreve{\bdry}} (\bH \setminus K_t) .
\end{align*}
These prime ends correspond to $W(t-)$ and $W(t)$, as will be shown in
the proof of Proposition~\ref{prop: constructing null-chains}. 

We call $\growing_t$ the \emph{growing end} for the Loewner chain at time $t$, 
and $\grown_t$ the \emph{grown end} at time $t$. 
Also, by the term \emph{growing point} at time $t$ we refer to points in the impression $I(\growing_t)$, and by the term \emph{grown point} at time $t$ we refer to points 
that are swallowed at time $t$ or belong to the impression $I(\grown_t)$. 
Note that the hulls might be generated by a 
self-crossing or self-touching curve $\gamma$, in which case a grown point $z$ might also be a double-point of the curve, i.e., $z = \gamma(s) = \gamma(t) \in K_s \cap K_t$ for some $s < t$.

In particular, the shrinking crosscuts
$S_{\delta}^{\mathrm{out}}$ and $S_{\delta}^{\mathrm{in}}$ correspond respectively 
to the right and left limits of the driving function $W$ at time $t$, which might be distinct (for $W$ is only assumed to be c\`adl\`ag):
\begin{align*} 
\{W(t)\} = \bigcap_{\delta > 0} \overline{g_t( K_{t+\delta} \setminus K_t )} \; \subset \; \bR
\qquad \textnormal{and} \qquad
\{W(t-)\} = \bigcap_{0 < \delta < t} \overline{g_{t-\delta}( K_t \setminus K_{t-\delta} )}  \; \subset \; \bR ,
\end{align*} 
where these equations hold by Proposition~\ref{prop: first construction driving measure}.
This leads to the following identification. 

\begin{prop}
\label{prop: characterise grown/growing end}
Consider a Loewner chain driven by a c\`adl\`ag function $W\colon [0, \infty) \to \bR$,
and let $(K_t)_{t \geq 0}$ be the associated locally growing hulls.
Let $t \geq 0$ be fixed.
\begin{enumerate}[label=\textnormal{(\arabic*):}, ref=\textnormal{(\arabic*)}]
\item The grown end at time $t$ is the unique prime end of $\bH \setminus K_t$ associated with $W(t-)$.

\smallskip

\item The growing end at time $t$ is the unique prime end of $\bH \setminus K_t$ associated with $W(t)$.
\end{enumerate}
\end{prop}

\begin{proof}
This holds by the proof of Loewner's theorem (Theorem~\ref{thm: loewner intro}) --- 
particularly, Corollary~\ref{cor: B(W(t), eps)} 
and the construction of the crosscuts $S_{\delta}^{\mathrm{in}} \subset \bH \setminus K_{t-\delta}$ 
and $S_{\delta}^{\mathrm{out}} \subset \bH \setminus K_{t}$ in the proof of Proposition~\ref{prop: constructing null-chains}. 
\end{proof}

Very importantly, Proposition~\ref{prop: characterise grown/growing end} allows us to immediately characterize the effect of the (dis)continuity of the driving function on the geometry and topology of the hulls in the following manner. 

\begin{cor}
\label{cor: characterise continuity of the driving function}
For a Loewner chain driven by a c\`adl\`ag function $W$, the following are equivalent. 
\begin{enumerate}[label=\textnormal{(\arabic*):}, ref=\textnormal{(\arabic*)}]
\item $W$ is continuous at time $t$.

\smallskip

\item The grown and growing end at time $t$ coincide.
\end{enumerate}
\end{cor}

\smallskip
\subsection{Proof of Loewner's theorem: Solving the Loewner equation}
\label{subsec: driving fct to locally growing hulls}

In this section, we prove Theorem~\ref{thm: Locally growing hulls solve LE intro}: 
Given a c\`adl\`ag function $W\colon [0, \infty) \to \bR$,  
there exists a unique solution to the Loewner equation~\eqref{eq: LE}, giving rise to the mapping-out functions $(g_t)_{t \geq 0}$
of a family $\boldsymbol{K} = (K_t)_{t \geq 0}$ of locally growing hulls.  
We establish this via several steps (Proposition~\ref{prop: existence and uniqueness LE}--Proposition~\ref{prop: first construction driving measure}).
For $z \in \bH$, set
\begin{align*}
\tau(z) := \; & \sup \Big\{ s \geq 0 \; | \; \inf_{u \in [0, s]} | g_u(z) - W(u) | > 0 \Big\} \; \in \; [0,\infty] ,
\\
\sigma(z) := \; & \sup \big\{ s \geq 0 \; | \; g_s(z) \in \bH \big\} \; \in \; [0,\infty].
\end{align*}

\begin{prop}
\label{prop: existence and uniqueness LE}
Let $W\colon [0, \infty) \to \bR$ be a c\`adl\`ag function. 
For all $z \in \bH$, there exists a unique absolutely continuous solution $t \mapsto g_t(z)$ on $[0, \tau(z))$ of the Loewner equation~\eqref{eq: LE} with driving function~$W$. 
\end{prop}

\begin{proof}
We first argue that $\tau(z) = \sigma(z)$. 
If $t < \sigma(z)$, then by the continuity of $t \mapsto \im(g_t(z))$, we have
\begin{align*}
0 < \inf_{u \in [0, t]} \im(g_u(z)) \leq \inf_{u \in [0, t]} | g_u(z) - W(u) | .
\end{align*} 
In particular, this implies that $t < \tau(z)$. Because $t < \sigma(z)$ was arbitrary, this proves $\tau(z) \leq \sigma(z)$. 
Conversely, by the Loewner equation~\eqref{eq: LE}, for all $z \in \bH$ and all $t$ sufficiently small, we have  
\begin{align}\label{eq:Im_g_t}
\im(g_t(z)) = \im(z) \exp \bigg( \! - \int_0^t \frac{2 \, \ud s}{|g_s(z) - W(s-)|^2} \bigg) \; > \; 0.
\end{align}
If $t < \tau(z)$, i.e.,~$\underset{u \in [0, t]}{\inf} \, | g_u(z) - W(s) | > 0$, then $\im(g_s(z)) > 0$ for all $s \in [0, t]$. 
Thus, $\tau(z) \leq \sigma(z)$.

Secondly, define $\Omega := [0, \infty) \times \bH$ and $f\colon \Omega \to \bC$ by $f(t, z) := \frac{2}{z - W(t)}$. 
Fix $z \in \bH$. 
From general ODE theory~\cite[Chapter~I.5.,~Theorems~5.1--5.3]{Hale:Ordinary_differential_equations} 
(see also~\cite[Chapter~6]{Pommerenke:Univalent_functions}), 
the existence of a unique absolutely continuous solution $t \mapsto g_t(z)$ of the Loewner equation~\eqref{eq: LE}, defined up to the blow-up time $\tau(z) = \sigma(z)$, 
follows\footnote{The Carath\'eodory conditions~\ref{item: hale 1}--\ref{item: hale 3} guarantee the existence of
an absolutely continuous solution $t \mapsto g_t(z)$ by~\cite[Theorem~5.1--5.2]{Hale:Ordinary_differential_equations}, and the additional local Lipschitz property~\ref{item: hale 4} yields uniqueness by~\cite[Theorem~5.3]{Hale:Ordinary_differential_equations}.} 
from checking the Carath\'eodory conditions~\ref{item: hale 1}--\ref{item: hale 3} 
and Lipschitz property~\ref{item: hale 4}:
\begin{enumerate}[label=\textnormal{(\arabic*):}, ref=\textnormal{(C\arabic*)}]
\item\label{item: hale 1}
for each fixed $z$, the map $t \mapsto f(t, z)$ is measurable on $\{t \in \bR \; | \; (t, z) \in \Omega\}$;

\smallskip

\item\label{item: hale 2}
for each fixed $t$, the map $z \mapsto f(t, z)$ is continuous on $\{t \in \bR \; | \; (t, z) \in \Omega\}$;

\smallskip

\item\label{item: hale 3} for each compact $U \subset \Omega$, there exists an integrable function $t \mapsto m_U(t)$ such that 
\begin{align*}
|f(t, z)| \leq m_U(t) \qquad \textnormal{for all } (t, z) \in U ;
\end{align*}

\smallskip

\item\label{item: hale 4} the map $z \mapsto f(t, z)$ is locally Lipschitz with measurable Lipschitz function, 
i.e., for each compact set $U \subset \Omega$, there exists an integrable function 
$t \mapsto k_U(t)$ such that 
\begin{align*}
|f(t, z) - f(t, w)| \leq k_U(t) |z - w|  \qquad \textnormal{for all } (t, z) ,(t, w) \in U . 
\end{align*}
\end{enumerate}
Property~\ref{item: hale 1} holds because $t \mapsto f(t, z)$ is c\`adl\`ag. 
To verify Property~\ref{item: hale 3}, we consider a compact $U \subset \Omega$.
Set $\delta := \min \{\im(w) \;|\; (t,w) \in U \textnormal{ for some } t \geq 0 \} > 0$. 
Then, we see that~\ref{item: hale 3} holds with $m_U(t) := 2/\delta$.
Similarly, Property~\ref{item: hale 4} holds with $k_U(t) := 2/\delta^2$.
This concludes the proof.
\end{proof}

\begin{cor}
\label{cor: g continuous} 
Write $H_t := \{ z \in \bH \;|\; \tau(z) > t\}$ for each $t \geq 0$, and
$H := \big\{ (t, z) \in  [0, \infty) \times \bH \;|\; z \in H_t \big\}$. 
The Loewner chain $g\colon H \to \bH$ from Proposition~\ref{prop: existence and uniqueness LE} is jointly continuous in~$t$~and~$z$.
\end{cor}

\begin{proof}
This follows directly from Proposition~\ref{prop: existence and uniqueness LE} and~\cite[Theorem~5.3]{Hale:Ordinary_differential_equations}.
\end{proof}

We can similarly deal with the \emph{backward Loewner equation}~\eqref{eq: BLE}:

\begin{prop}
\label{prop: BLE}
Let $W\colon [0, \infty) \to \bR$ be a c\`adl\`ag function. Let $t > 0$ be fixed. 
For all $z \in \bH$, 
\begin{align}
\label{eq: BLE}
\begin{split}
\pderr{s} h_s(w) &= \frac{-2}{h_s(z) - W(t-s)},
\\
h_0(z) &= z ,
\end{split}\tag{BLE}
\end{align}
has a unique absolutely continuous solution $s \mapsto h_s(z)$, which exists for all $s \in [0, t]$.
This \emph{backward Loewner chain} gives rise to a jointly continuous function $h\colon [0, t] \times \bH \to \bH$ in $s$ and $z$.
\end{prop}

\begin{proof}
The proof is the same as the proof of Proposition~\ref{prop: existence and uniqueness LE} and Corollary~\ref{cor: g continuous}. 
This solution exists for all time, since $s \mapsto \im(h_s(z))$ is increasing:~\eqref{eq: BLE} shows that $\frac{\im(h_s(z))}{\im(z)} = \exp \big( \int_0^t \frac{2 \, \ud r}{|h_r(z) - W((t-r)-)|^2} \big)$.
\end{proof}

\begin{prop}
\label{prop: solutions to LE are conformal bijections}
Let $W\colon [0, \infty) \to \bR$ be a c\`adl\`ag function.  
Let $(g_t)_{t \geq 0}$ be the unique solution to the Loewner equation~\eqref{eq: LE} with driving function $W$
\textnormal{(}from Proposition~\ref{prop: existence and uniqueness LE}\textnormal{)}. 
Write $H_t := \{ z \in \bH \;|\; \tau(z) > t\}$ for $t \geq 0$. 
Then, the map $g_t\colon H_t \to \bH$ is a conformal bijection.
\end{prop}

\begin{proof}
We follow the usual proof strategy for Loewner's theorem (e.g.,~\cite[Proposition~4.1]{Kemppainen:SLE_book}). 

{\bf Step 1}: We show that the maps $g_t \colon H_t \to g_t(H_t)$ are conformal for $t \ge 0$.
For $z, w \in H_t$, define $\Delta_t(z, w) := g_t(z) - g_t(w)$.
The Loewner equation~\eqref{eq: LE} gives
\begin{align*}
\pderr{t} \Delta_t(z, w) 
= \; & \frac{- 2 \, \Delta_t(z, w)}{(g_t(z) - W(t)) \, (g_t(w) - W(t))} .
\end{align*}
We will argue that ($\star$) this differential equation has a unique solution which is continuous in $t$:
\begin{align} \label{eq:Delta_def}
\Delta_t(z, w) = \; & (z - w) \, \exp \bigg( \! - \int_0^t \frac{2 \, \ud r}{(g_r(z) - W(r-)) \, (g_r(w) - W(r-))} \bigg) .
\end{align}
It then follows that $g_t\colon H_t \to g_t(H_t)$ is holomorphic and injective, hence conformal.

To verify ($\star$), write $\beta_r(z,w) := - \frac{2}{ (g_r(z) - W(r-)) \, (g_r(w) - W(r-)) }$. 
The function~\eqref{eq:Delta_def} satisfies
\begin{align}\label{eq: right derivative f}		
\Delta_{t+s}(z, w) - \Delta_t(z, w)
= \Delta_t(z, w)
\bigg( \exp \bigg( \int_t^{t+s} \beta_r(z,w) \, \ud r \bigg) - 1 \bigg) , \qquad s > 0.
\end{align}
Since the integrand $r \mapsto \beta_r(z,w)$ has unique right limits, 
for every $\varepsilon > 0$ there exists $\delta_{\varepsilon} > 0$ such that $\underset{\varepsilon \to 0}{\lim} \, \delta_{\varepsilon} \to 0$ 
and $| \beta_{t+s}(z,w) - \beta_t(z,w) | \leq \varepsilon$ for all $s \in [0, \delta_{\varepsilon}]$. 
This shows that
\begin{align*}
s \big( \beta_t(z,w) - \varepsilon \big) \leq	\int_{t}^{t+s} \beta_r(z,w) \, \ud r \leq s \big( \beta_t(z,w) + \varepsilon \big) , \qquad s \in [0, \delta_{\varepsilon}] .
\end{align*}
Taking $s \to 0$ and using the Taylor expansion for the exponential function, we see that
\begin{align*}
\beta_t(z,w) - \varepsilon
\; \leq \; \lim_{s \to 0} \bigg( \frac{\exp \big( \int_{t}^{t+s} \beta_r(z,w) \ \ud r \big) - 1}{s} \bigg)
\; \leq \; \beta_t(z,w) + \varepsilon
\end{align*}
Because $\delta_{\varepsilon} \to 0$ as $\varepsilon \to 0$, we see from~\eqref{eq: right derivative f} that $\pderr{t} \Delta_t(z, w) = \beta_t(z,w) \Delta_t(z, w)$, as claimed.
Note also that $t \mapsto \Delta_t(z, w)$ is right-continuous, being right-differentiable.
To see that it is left-continuous at $t$, note that since $t \mapsto \beta_t(z,w)$ has unique left-limits, 
for every $\varepsilon > 0$ there exists $\delta_{\varepsilon} > 0$ such that $\underset{\varepsilon \to 0}{\lim} \, \delta_{\varepsilon} \to 0$ 
and $| \beta_{t-}(z,w) - \beta_{t-s}(z,w) | < 1$ for all $s \in (0, \delta]$, and thus, 
similarly as above, we have
\begin{align*}
s \big( \beta_{t-}(z,w) - \varepsilon \big) \leq \int_{t-s}^{t} \beta_r(z,w) \, \ud r \leq s \big( \beta_{t-}(z,w) + \varepsilon \big) , \qquad s \in [0, \delta_{\varepsilon}] .
\end{align*} 
which yields $\int_{t-s}^{t} \beta_r(z,w) \, \ud r \to 0$ as $s \to 0$, so 
\begin{align*}
0 \leq \lim_{s \to 0} | \Delta_{t}(z, w) - \Delta_{t-s}(z, w) | 
= \; & \Delta_{t}(z, w) \, \lim_{s \to 0} 
\bigg| 1 -  \exp \bigg( \! - \int_{t-s}^{t} \beta_r(z,w) \, \ud r \bigg) \bigg| 
= 0 .
\end{align*}
Lastly, the uniqueness of the solution~\eqref{eq:Delta_def} follows from a version of Gr\"onwall's lemma (Lemma~\ref{thm: gronwall}).

{\bf Step 2}: We show that $g_t(H_t) = \bH$ for all $t \geq 0$.
The inclusion $g_t(H_t) \subset \bH$ follows from Equation~\eqref{eq:Im_g_t}.
For the reverse inclusion, fix $w \in \bH$.
By Proposition~\ref{prop: solutions to LE are conformal bijections}, the backward Loewner equation~\eqref{eq: BLE} 
has a unique solution that exists for all $s \in [0, t]$.
Then, $s \mapsto h_{t-s}(z)$ solves the Loewner equation~\eqref{eq: LE} on $[0, t]$ with initial value $u := h_t(w) \in \bH$.
Hence, by the uniqueness of solutions to the Loewner equation~\eqref{eq: LE} by Proposition~\ref{prop: existence and uniqueness LE}, we have $w = h_0(w) = g_t(u)$.
This shows that $\bH \subset g_t(H_t)$ and finishes the proof.
\end{proof}

\begin{lem}[Gr\"onwall's lemma]
\label{thm: gronwall}
Fix $a \in \bR$ and let $I \in \{[a, \infty), \, [a, b), \, [a, b]\}$ be an interval, where $b > a$.
Let $f\colon I \to \bR$ be a continuous function and $\beta \colon I \to \bR$ a c\`adl\`ag function.  
\begin{enumerate}[label=\textnormal{(\arabic*):}, ref=\textnormal{(\arabic*)}]
\item \label{item:gronwall1}
If $\pderr{t} f(t) \leq \beta(t) f(t)$ for all $t \in I$,
then $f(t) \leq f(a) \, \exp \big( \int_{a}^{t} \beta(s) \ud s \big)$.

\smallskip

\item \label{item:gronwall2}
If $\pderr{t} f(t) \geq \beta(t) f(t)$ for all $t \in I$,
then $f(t) \geq f(a) \, \exp \big( \int_{a}^{t} \beta(s) \ud s \big)$.
\end{enumerate}
\end{lem}

\begin{proof}
Set $v(t) := \exp \big( \int_{a}^{t} \beta(s) \ud s \big)$ and $u(t) := \frac{f(t)}{v(t)}$.
As in the above proof, $v\colon I \to \bR$ is continuous and $\pderr{t} v(t) = \beta(t) v(t)$. 
If $\pderr{t} f(t) \leq \beta(t) f(t)$ for all $t$, then $\pderr{t} u(t) \leq 0$ and thus, 
$\smash{\frac{u(t) - u(a)}{t-a} \leq 0}$.
Rearranging this yields Item~\ref{item:gronwall1}: $f(t) = u(t) v(t) \leq u(a) v(t) = f(a) v(t)$.
The proof of Item~\ref{item:gronwall2} is analogous.
\end{proof}

\begin{cor}
\label{cor: gt'(z) continuous}
Let $W\colon [0, \infty) \to \bR$ be a c\`adl\`ag function.  
Let $(g_t)_{t \geq 0}$ be the unique solution to the Loewner equation~\eqref{eq: LE} with driving function $W$
\textnormal{(}from Proposition~\ref{prop: existence and uniqueness LE}\textnormal{)}. 
Then, we have
\begin{align*}
g_t'(z) = \exp \bigg(  \! - \int_0^t \frac{2 \, \ud r}{(g_r(z) - W(r-))^2} \bigg) , \qquad t \geq 0 , \; z \in H_t.
\end{align*}
In particular, $t \mapsto g_t'$ is differentiable with 
\begin{align*}
(\pder{t} g_t')(z) = \frac{-2 \, g_t'(z)}{(g_t(z) - W(t))^2} .
\end{align*}
Moreover, the function $g' \colon H \to \bC$ is jointly continuous on the set 
$H := \big\{ (t, z) \in  [0, \infty) \times \bH \;|\; z \in H_t \big\}$. 
\end{cor}

\begin{proof}
The asserted formulas follow from Equation~\eqref{eq:Delta_def} in the proof of Proposition~\ref{prop: solutions to LE are conformal bijections}. 
Also, we have
\begin{align*}
\; & | g_{t+\delta}'(z) - g_t'(w) |
\leq | g_{t+\delta}'(z) - g_t'(z) | + | g_t'(z) - g_t'(w) |
\\
= \; & 
\bigg| \exp \bigg(  \! - \int_0^t \frac{2 \, \ud r}{(g_r(z) - W(r-))^2} \bigg) \bigg| \;
\bigg| \exp \bigg(  \! - \int_t^{t+\delta} \frac{2 \, \ud r}{(g_r(z) - W(r-))^2} \bigg) - 1 \bigg| + | g_t'(z) - g_t'(w) |
\; \overset{w \to z}{\underset{\delta \to 0+}{\longrightarrow}} \; 0 , 
\end{align*}
and similarly,
\begin{align*}
\; & | g_{t-\delta}'(z) - g_t'(w) |
\leq | g_{t-\delta}'(z) - g_t'(z) | + | g_t'(z) - g_t'(w) |
\\
= \; & 
\bigg| \exp \bigg(  \! - \int_0^{t-\delta} \frac{2 \, \ud r}{(g_r(z) - W(r-))^2} \bigg) \bigg| \;
\bigg| 1 - \exp \bigg(  \! - \int_{t-\delta}^t \frac{2 \, \ud r}{(g_r(z) - W(r-))^2} \bigg) \bigg| 
+ | g_t'(z) - g_t'(w) |
\; \overset{w \to z}{\underset{\delta \to 0+}{\longrightarrow}} \; 0 , 
\end{align*}
which shows the joint continuity of $(t,z) \mapsto g_t'(z)$. 
\end{proof}

In summary, solutions to the Loewner equation~\eqref{eq: LE} exist, are unique, and define conformal bijections $g_t\colon H_t \to \bH$, where $(H_t)_{t \geq 0}$ are shrinking simply connected domains in $\bH$. 
To conclude with Theorem~\ref{thm: Locally growing hulls solve LE intro}, it remains to be shown that the sets $K_t := \overline{\bH \setminus H_t}$ define a family of locally growing hulls. 

\begin{lem}
\label{lemma: Kt bounded}
Let $W\colon [0, \infty) \to \bR$ be a c\`adl\`ag function. Then, for all $t \geq 0$,
\begin{align*}
K_t := \{ z \in \overline{\bH} \mid \tau(z) \leq t \}
\subset 
\Big\{z \in \overline{\bH} \;\Big|\; |\re(z)| \leq \sup_{r \in [0, t]} |W(r) - W(0)| \textnormal{ and }  \im(z) \leq 2 \sqrt{t} \Big\}.
\end{align*}
In particular, for each $t \geq 0$, the set $K_t$ is a hull. 
\end{lem}

We follow the same argument as in the proof of~\cite[Proposition~4.1]{Kemppainen:SLE_book}.
Because some of the identities derived in the proof will become handy later, we include the arguments here.

\begin{proof}
By applying a translation if necessary, we may assume that $W(0) = 0$. Fix $t \geq 0$.

{\bf Bounding the imaginary part.}
First of all, we show that $\im(w) \leq 2 \sqrt{t}$ for all $w \in K_t$ $(\star)$.
Consider $k_t(z) := (\im (g_t(z)))^2 + 4t$.
If $z \in \bH$ with $ \im(z) > 2 \sqrt{t}$, then
\begin{align} \label{eq: height bound hulls}
(\im (g_t(z)))^2 \geq (\im(z))^2 - 4t
\qquad \Longleftrightarrow \qquad
k_t(z) \geq k_0(z) .
\end{align}
Moreover, the map $t \mapsto k_t(z)$ is increasing:
\begin{align*}
\pderr{t} k_t(z) = 2 \, \im (g_t(z)) \, \pderr{t} \im (g_t(z)) + 4
= 4 \, \bigg(1 - \frac{ (\im (g_t(z)))^2 }{|g_t(z) - W(t)|^2}\bigg) \geq 0.
\end{align*}
Therefore, we can conclude from~\eqref{eq: height bound hulls} that
if $\im(z) > 2 \sqrt{t}$, then $(\im (g_t(z)))^2 \geq (\im (z))^2 - 4t > 0$, 
and hence that $\im (g_t(z))$ is strictly bounded away from $0$. 
In particular, this implies that $\tau(z) > t$, 
so we see that $\{ z \in \bH \;|\; \im(z) > 2 \sqrt{t} \} \subset H_t$, which proves $(\star)$.

{\bf Bounding the real part.}
Second of all, we prove that 
\begin{align}
\label{eq: width hulls}
| \re(w) | \leq M := \sup_{s \in [0, t]} |W(s)|  < \infty 
\qquad \textnormal{for all } w \in K_t.
\end{align}
Consider $z \in \bH$ with $|\mathrm {Re}(z)| > M$, and assume $\re(z) > M$ by symmetry. 
By the continuity of $s \mapsto g_s(z)$ (Proposition~\ref{prop: existence and uniqueness LE}), $\re(g_s(z)) \geq M$ for all $s \in [0, \sigma]$,
where $\sigma := \inf\{s \geq 0 \colon \re(g_s(z)) = M\} \wedge t$.
Thus, 
\begin{align*}
\pderr{s} \re (g_s(z))
= \frac{2 (\re(g_s(z)) - W(s) )}{|g_s(z) - W(s)|^2}
\geq 
\frac{2 (M - W(s) )}{|g_s(z) - W(s)|^2} \geq 0 \qquad \textnormal{for all } s \in [0, \sigma] ,
\end{align*}
so $s \mapsto \re(g_s(z))$ is increasing on $[0, \sigma]$.
Therefore, $\sigma = t$ and $\re(g_t(z)) \geq \re(z) > M$. 
In particular, in this case 
$\re(g_t(z)) - W(t) > M - W(t) \geq 0$, i.e.,~$\tau(z) > t$ and $z \notin K_t$.
The bound~\eqref{eq: width hulls} follows.

We have now shown that $K_t$ is bounded. 
Moreover, as $g_t\colon H_t \to \bH$ is a conformal bijection by Proposition~\ref{prop: solutions to LE are conformal bijections}, $g_t$ is also a homeomorphism. 
Thus, the set $\bH \setminus K_t = H_t = \{ z \in \bH \;|\; \tau(z) > t \}$ is homeomorphic to the upper half-plane, whence it is open and simply connected. 
This proves that $K_t$ is a hull. 
\end{proof}

Let us next check that the mappings $g_t \colon \bH \setminus K_t \to \bH$ are normalized  at $\infty$ as in~\eqref{eq: mof Laurent exp}. 

\begin{lem}
\label{lem: LE and solutions hydro-normalized}
Let $W\colon [0, \infty) \to \bR$ be a c\`adl\`ag function.  
Let $(g_t)_{t \geq 0}$ be the unique solution to the Loewner equation~\eqref{eq: LE} with driving function $W$
\textnormal{(}from Proposition~\ref{prop: existence and uniqueness LE}\textnormal{)}.
Then, we have
\begin{align*}
g_t \colon H_t \to \bH ,
\qquad \qquad
g_t(z) = z + \sum_{n = 1}^{\infty} a_n(t) \, z^{-n} , \qquad |z| \to \infty ,
\end{align*}
for all $t \geq 0$, where $a_n(t) \in \bR$. 
\end{lem}

\begin{proof}
Fix $t \geq 0$. By the Loewner equation~\eqref{eq: LE}, we have 
\begin{align}
\label{eq: helper normalisation at inf}
| g_t(z) - z |
= \bigg| \int_0^t \frac{2 \,  \ud s}{g_s(z) - W(s-)} \bigg|
\leq \int_0^t \frac{2 \, \ud s}{|g_s(z) - W(s-)|} , \qquad \textnormal{for all } z \in H_t. 
\end{align}
In the proof of Lemma~\ref{lemma: Kt bounded}, we have shown that for every $z \in \bH \setminus K_t$,
\begin{itemize}
\item 
if $\im(z) > 2 \sqrt{t}$, then $(\im(g_t(z)))^2 \geq (\im(z))^2 - 4t > 0$; and

\bigskip

\item 
if $| \re(z) | > M :=\smash{ \underset{0 \leq r \leq t}{\sup}} \, |W(r) - W(0)|$, then $( \re(g_t(z)) )^2 > ( \re(z))^2$; and 

\bigskip

\item $K_s \subset K_t \subset B (0, \sqrt{M^2 + 4t} )$  for all $s \in [0, t]$.
\end{itemize}
Thus~\cite[Lemma~4.5]{Kemppainen:SLE_book} implies that
\begin{align*}
|\re(g_s(z))| \leq |g_s(z) - z| + |z| \leq |z| + 5 \sqrt{M^2 + 4t} , \qquad z \in \bH \setminus K_t , \; s \in [0, t] ,
\end{align*}
and we can estimate, for $s \in [0, t]$ and $z \in \bH \setminus K_t$ with $|z|$ sufficiently large,
\begin{align*}
| g_s(z) - W(s) |^2
\geq \; &  |z|^2 - M^2 - 4t - 2  M \, |\re(g_s(z))| \\
\geq \; & |z|^2 - M^2 - 4t - 2  M |z| - 10M \sqrt{M^2 + 4t} ,
\end{align*}
Combining this with the estimate~\eqref{eq: helper normalisation at inf} yields $| g_t(z) - z | \to 0$ as $|z| \to \infty$. 
It follows from reflection symmetry that $a_n(t) \in \bR$ (see, e.g.,~\cite[Lemma~4.1]{Kemppainen:SLE_book} for a detailed argument).
\end{proof}

As a result, this also proves that the hulls $\boldsymbol{K} = (K_t)_{t \geq 0}$ are parametrized by their half-plane capacity.
This proves the direction~\ref{item: loewner driver}~$\Longrightarrow$~\ref{item: loewner hulls} in Theorem~\ref{thm: loewner intro}.

\begin{cor} 
Let $W\colon [0, \infty) \to \bR$ be a c\`adl\`ag function.  
The associated hulls $\boldsymbol{K} = (K_t)_{t \geq 0}$ 
\textnormal{(}from Lemma~\ref{lemma: Kt bounded}\textnormal{)}
are parametrized by capacity, i.e.,~$\mathrm{hcap}(K_t) = 2 t$ for $t \geq 0$.
\end{cor}

\begin{proof} 
Let $t \geq 0$. 
By Lemma~\ref{lem: LE and solutions hydro-normalized} and the definition of the half-plane capacity, the mapping-out function $g_t$ has the expansion~\eqref{eq: mof Laurent exp}, where $a_1(t) = \mathrm{hcap}(K_t)$.
Using the Loewner equation~\eqref{eq: LE}, we have
\begin{align*}
&\frac{\pderr{t} \mathrm{hcap}(K_t)}{z} + \sum_{n = 2}^{\infty} \frac{\pderr{t} a_n(t)}{z^n}
= \pderr{t} g_t(z)
=
\frac{2}{z} \bigg( 1 - \frac{W(t)}{z} + \frac{\mathrm{hcap}(K_t)}{z^2} + \sum_{n = 2}^{\infty} \frac{a_n(t)}{z^{n+1}} \bigg)^{-1}.
\end{align*}
This gives $\pderr{t} \mathrm{hcap}(K_t) = 2$, so by Gr\"onwall's lemma (Lemma~\ref{thm: gronwall}), $\mathrm{hcap}(K_t) = 2t$ for all $t \geq 0$.
\end{proof}

It remains to show that the hulls $(K_t)_{t \geq 0}$ are locally growing. 
First we establish some technical results. 

\begin{lem}
\label{lemma: right ball W}
Let $W\colon [0, \infty) \to \bR$ be a c\`adl\`ag function.  
Let $(g_t)_{t \geq 0}$ be the unique solution to the Loewner equation~\eqref{eq: LE} with driving function $W$
\textnormal{(}from Proposition~\ref{prop: existence and uniqueness LE}\textnormal{)},
and $\boldsymbol{K} = (K_t)_{t \geq 0}$ the associated hulls \textnormal{(}from Lemma~\ref{lemma: Kt bounded}\textnormal{)}.
Then, the following inclusions hold for the sets~\eqref{eq:Ktilde}\textnormal{:}
\begin{align} \label{eq: helper right ball}
\tilde{K}^t_{s}
\subset \; & B \Big( W(t), \sqrt{ \sup_{0 \leq r \leq s} |W(t+r) - W(t)|^2 +  4s} \Big), \qquad s, t \geq 0 , \\
\label{eq: helper left ball}
\tilde{K}^{t-s}_{s} \subset \; &
B \Big( W(t-s), \sqrt{ \sup_{0 < r \leq s} |W(t - r) - W(t - s)|^2 +  4s} \Big)
\qquad t > 0 ,\; s \in (0, t) . 
\end{align} 
\end{lem}

\begin{proof}
For~\eqref{eq: helper right ball}, fix $t \geq 0$.
By Lemma~\ref{lem: some markov property}, for each $s \geq 0$, $\tilde{K}^{t}_{s}$ is a hull whose mapping-out function is given by $\smash{\tilde{g}^t_{s}} = g_{t+s} \circ g_t^{-1}$.  
Hence, by Proposition~\ref{prop: existence and uniqueness LE} it is easy to see that the $(\tilde{g}^t_{s})_{s \geq 0}$ solve Loewner's equation \eqref{eq: LE} with driving function $\tilde{W}^{t}(s) := W(t + s)$. 
Therefore, by Lemma~\ref{lemma: Kt bounded}, we have 
\begin{align*}
\tilde{K}^{t}_{s} \subset \Big\{ z \in \overline{\bH} \;\Big|\; |\re(z)| \leq \sup_{r \in [0, s]} |W(t+r) - W(t)|, \  \im(z) \leq 2 \sqrt{s} \Big\} , \qquad s \geq 0.
\end{align*}
This shows~\eqref{eq: helper right ball}. 
For~\eqref{eq: helper left ball}, fix $t > 0$ and $s \in (0, t)$. 
By the above, we have
\begin{align*}
\tilde{K}^{t-s}_{s-r} 
&\subset
B \Big( W(t - s), \sqrt{ \sup_{0 \leq u \leq s - r} |W(t - s + u) - W(t - s)|^2 +  4(s - r)} \Big)
\\
&= 
B \Big( W(t - s), \sqrt{ \sup_{r \leq v \leq s} |W(t - v) - W(t - s)|^2 +  4(s - r)} \Big) , \qquad r \in (0, s) . 
\end{align*}
Note that $\tilde{K}^{t-s}_{s-r} = \overline{g_{t-s}(K_{t-r} \setminus K_{t-s})}$ are growing as $r$ decreases. 
On the one hand, by Lemmas~\ref{lem: hcap and growing}~\&~\ref{lem: some markov property}, 
$\tilde{K}^{t-s}_{s}$ is the smallest hull containing $\bigcup_{r \in (0, s)} \tilde{K}^{t-s}_{s-r}$.
On the other hand, we have 
\begin{align*}
&\overline{B \Big( W(t-s), \sup_{0 < r < s} \sqrt{ \sup_{r \leq v \leq s} |W(t - v) - W(t - s)|^2 +  4(s - r)} \Big) \cap \bH}
\\
&= \overline{B \Big( W(t-s), \sqrt{ \sup_{0 < v \leq s} |W(t - v) - W(t - s)|^2 +  4s} \Big) \cap \bH}
\end{align*}
is a hull containing $\bigcup_{r \in (0, s)} \tilde{K}^{t-s}_{s-r}$. 
This implies~\eqref{eq: helper left ball} and concludes the proof. 
\end{proof}

\begin{cor}
\label{cor: B(W(t), eps)}
Let $W\colon [0, \infty) \to \bR$ be a c\`adl\`ag function.  
Let $(g_t)_{t \geq 0}$ be the unique solution to the Loewner equation~\eqref{eq: LE} with driving function $W$
\textnormal{(}from Proposition~\ref{prop: existence and uniqueness LE}\textnormal{)},
and $\boldsymbol{K} = (K_t)_{t \geq 0}$ the associated hulls \textnormal{(}from Lemma~\ref{lemma: Kt bounded}\textnormal{)}.
Then, the following inclusions hold for $t \geq 0$\textnormal{:} 
\begin{enumerate}[label=\textnormal{(\arabic*):}, ref=\textnormal{(\arabic*)}]
\item\label{item: contained in ball near right limit} 
for all $\varepsilon > 0$, there exists $\delta > 0$ such that 
$\tilde{K}^t_{s} 
\subset B(W(t), \varepsilon)$ for all $s \in (0, \delta]$\textnormal{;} and 

\smallskip

\item\label{item: contained in ball near left limit} 
for all $\varepsilon > 0$, there exists $\delta > 0$ such that  
$\tilde{K}^{t - s}_{s}
\subset B(W(t-), \varepsilon)$ for all $s \in (0, \delta]$.
\end{enumerate}
\end{cor}

\begin{proof} 
This holds by Lemma~\ref{lemma: right ball W} and as $W\colon [0, \infty) \to \bR$ 
is right-continuous with unique left limits. 
\end{proof}

With these technical estimates, we can now show that the hulls $(K_t)_{t \geq 0}$ are locally growing. 

\begin{prop}
\label{prop: constructing null-chains}
Let $W\colon [0, \infty) \to \bR$ be a c\`adl\`ag function.  
The associated hulls $\boldsymbol{K} = (K_t)_{t \geq 0}$ 
\textnormal{(}from Lemma~\ref{lemma: Kt bounded}\textnormal{)} are locally growing in the sense of Definition~\ref{def: local growth}.
\end{prop}

\begin{proof}
Clearly, the hulls are growing: $K_s = \{ z \in \overline{\bH} \;|\; \tau(z) \leq s \} \subset \{ z \in \overline{\bH} \;|\; \tau(z) \leq t \} = K_t$ for $s < t$.

It remains to show that the hulls $\boldsymbol{K}$ are also locally growing. Fix $t > 0$. 
We will prove that $\boldsymbol{K}$ are left-locally growing in $t$;  the right-local growth can be proven very similarly (using Corollary~\ref{cor: B(W(t), eps)}\ref{item: contained in ball near right limit}). 

Let $\varepsilon > 0$ and let $R > 0$ be such that $K_t \subset B(0, R)$. 
Then, we can choose $\epsilon' \in (0, 1)$ such that 
\begin{align*}
\frac{4\pi ( 5R + 1 + |W(t-)| )}{\sqrt{\log(1/\epsilon')}} < \varepsilon.
\end{align*}
By Corollary~\ref{cor: B(W(t), eps)}\ref{item: contained in ball near left limit}, 
there exists $\delta > 0$ such that 
$\tilde{K}^{t - \delta}_{\delta} = \overline{g_{t - \delta}(K_{t} \setminus K_{t - \delta})} 
\subset B ( W(t-), \epsilon' )$. 
Moreover, by~\cite[Lemma~4.5]{Kemppainen:SLE_book}, because $K_{t- \delta} \subset K_t \subset B(0, R)$, 
we have
\begin{align*}
| g_{t-\delta}^{-1}(z) | 
\leq | g_{t-\delta}^{-1}(z) - z | + |z - W(t-)| + |W(t-)| \leq 5 R + 1 + |W(t-)| , \qquad z \in B ( W(t-), 1 ) \cap \bH.
\end{align*}
By Wolff's lemma (e.g.,~\cite[Lemma~4.6]{Kemppainen:SLE_book}), we obtain
\begin{align*} 
&\inf_{\epsilon' < r < \sqrt{\epsilon'}} \ 
\mathrm{length} \big( g_{t-\delta}^{-1} (\bH \cap \bdry B(W(t-), r) ) \big)
\leq \frac{2\pi ( 5R + 1 + |W(t-)| )}{\sqrt{\log(1/\epsilon')}} =\colon \rho(t, \sqrt{\epsilon'}). 
\end{align*}
In particular, there exists $r = r(\epsilon') \in (\epsilon', \sqrt{\epsilon'})$ such that 
\begin{align} 
\label{eq: helper construct null chain}
\mathrm{length} \big( g_{t-s}^{-1} (\bH \cap \bdry B(W(t-), r) ) \big) 
\leq 2 \rho(t, \sqrt{\epsilon'}) < \varepsilon .
\end{align}

\begin{figure}[ht]
\centering
\includegraphics[width=0.6\textwidth]{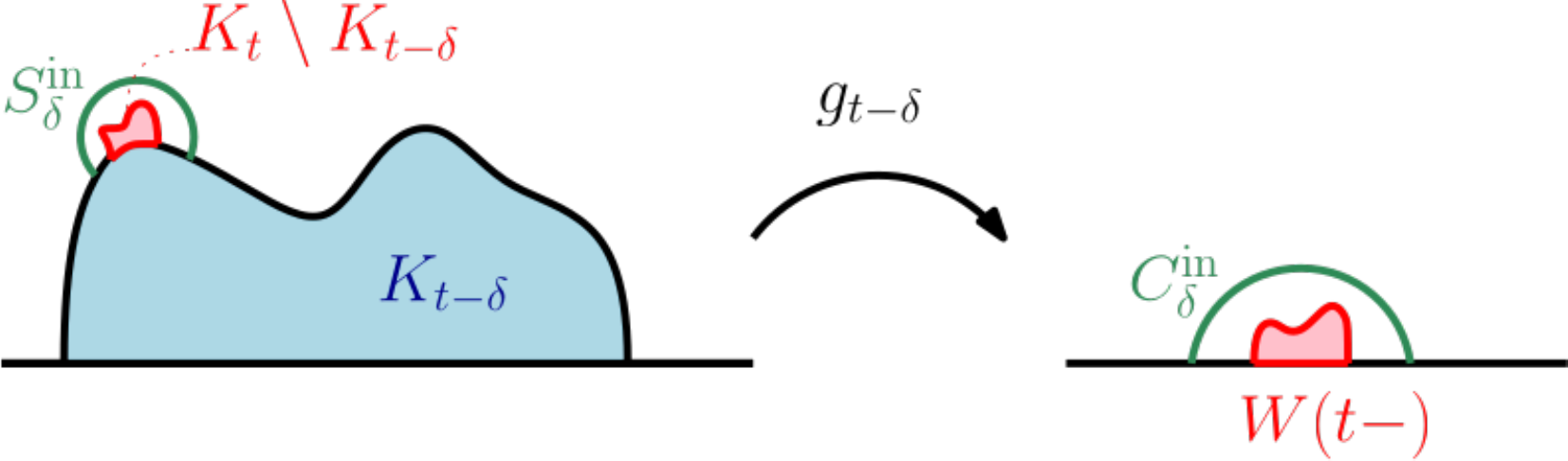}
\caption{Illustration for the proof of Proposition~\ref{prop: constructing null-chains}.}
\end{figure}

Furthermore, $C^{\mathrm{in}}_{\delta} := \bdry B(W(t-), r) \cap \overline{\bH}$ is a crosscut separating 
$g_{t - \delta}(K_{t} \setminus K_{t - \delta})$ from $\infty$ in $\bH$. 
Using this, the fact that $g_{t-\delta}^{-1}$ is a conformal bijection, 
and by the finite length estimate~\eqref{eq: helper construct null chain}, we see that
$$S^{\mathrm{in}}_{\delta} := \overline{ g_{t-\delta}^{-1} ( C^{\mathrm{in}}_{\delta} \cap \bH ) }$$ 
is a crosscut in $\bH \setminus K_{t - \delta}$ 
separating $K_{t} \setminus K_{t - \delta}$ from $\infty$ in $\bH \setminus K_{t-\delta}$ 
with $\diam (S^{\mathrm{in}}_{\delta}) < \varepsilon$.   
This proves that the hulls $\boldsymbol{K}$ are left-locally growing at $t$ (which was arbitrary).
\end{proof}

Collecting the above results concludes the proof of Theorem~\ref{thm: Locally growing hulls solve LE intro}.

Finally, let us point that we can see how the driving function can be obtained from the hulls $\boldsymbol{K} = (K_t)_{t \geq 0}$. 
This construction will be crucial in the reverse direction of the proof --- see Section~\ref{subsec: construct driving measure}. 

\begin{prop}
\label{prop: first construction driving measure}
Let $W\colon [0, \infty) \to \bR$ be a c\`adl\`ag function.  
Let $(g_t)_{t \geq 0}$ be the unique solution to the Loewner equation~\eqref{eq: LE} with driving function $W$
\textnormal{(}from Proposition~\ref{prop: existence and uniqueness LE}\textnormal{)},
and $\boldsymbol{K} = (K_t)_{t \geq 0}$ the associated hulls \textnormal{(}from Lemma~\ref{lemma: Kt bounded}\textnormal{)}.
Then, for all $t \geq 0$, we have 
\begin{align*}
\{ W(t) \} = \bigcap_{s > 0} \overline{g_t(K_{t+s} \setminus K_t)} 
= \bigcap_{s > 0} \tilde{K}^t_s 
\qquad \mathrm{and} \qquad
\{ W(t-) \} = \bigcap_{0 < s < t} \overline{g_{t-s}(K_{t} \setminus K_{t-s})}
= \bigcap_{0 < s < t} \tilde{K}^{t-s}_s .
\end{align*}
\end{prop}

\begin{proof}
Fix $t \geq 0$. $( \tilde{K}^{t}_{s} )_{s > 0} = ( \overline{g_t(K_{t+s} \setminus K_t)} )_{s > 0}$ is a decreasing sequence of non-empty compact sets as $ s \to 0+$.
By Cantor's intersection theorem (e.g.~\cite[Proposition~1.6(a)]{Dugundji:Topology}), 
and Corollary~\ref{cor: B(W(t), eps)}\ref{item: contained in ball near right limit}, 
\begin{align*}
\emptyset \neq \bigcap_{s > 0} \overline{g_t(K_{t+s} \setminus K_t)}  = \bigcap_{s > 0} \tilde{K}^{t}_{s} \subset \{W(t)\}.
\end{align*}
This shows that $\{W(t)\} = \bigcap_{s > 0} \tilde{K}^{t}_{s}$.
Likewise, we can conclude by Corollary~\ref{cor: B(W(t), eps)}\ref{item: contained in ball near left limit} that 
\begin{align}
\label{eq: identication left limit}
\bigcap_{0 < s < t} \overline{g_{t - s}(K_t \setminus K_{t - s})} = \bigcap_{0 < s < t} \tilde{K}^{t-s}_s \subset \{W(t-)\} .
\end{align}
Therefore, it remains to show that $W(t-) \in \tilde{K}^{t-s}_{s}$ for all $s \in (0, t)$.
Assume towards a contradiction that there exists $s \in (0, t)$ such that $W(t-) \notin \tilde{K}^{t-s}_{s}$. 
Then, the sequence 
\begin{align}
\label{eq: wrong distance}
S_n := \bdry B ( W(t-), 2^{-(n+1)} \varepsilon ) \cap \overline{\bH} , \qquad
\varepsilon := \dist \big(W(t-), \tilde{K}^{t-s}_{s} \big) > 0 ,
\end{align}
is a null-chain $(S_n)_{n \in \bZnn}$ in $\bH$ that has $W(t-)$ as its principal point. 
By~\cite[Theorem~2.15]{Pommerenke:Boundary_behaviour_of_conformal_maps}, 
there is a prime end $\xi$ of $\bH \setminus K_t$ that corresponds bijectively to the prime end of $\bH$ represented by $(S_n)_{n \in \bZnn}$. 
Moreover, the hulls $\boldsymbol{K}$ are left-locally growing by Proposition~\ref{prop: constructing null-chains}.
Let $(\delta_n)_{n \in \bZnn}$ be a sequence with limit zero. 
Then $(S_{\delta_n}^{\mathrm{in}})_{n \in \bZnn}$ is a null-chain in $\bH \setminus K_t$ (recall Remark~\ref{remark: Left local null chains}). 
Thus, $( g_t ( S_{\delta_n}^{\mathrm{in}} ) )_{n \in \bZnn}$ is a null-chain in $\bH$ by~\cite[Theorem~2.15]{Pommerenke:Boundary_behaviour_of_conformal_maps}. 
Furthermore, the crosscut $S_{\delta_n}^{\mathrm{in}}$ separates $K_t \setminus K_{t-\delta_n}$ from infinity in $\bH \setminus K_t$ for all $n \in \bZnn$. 
Now as $g_t\colon \bH \setminus K_t \to \bH$ is a conformal bijection, 
the crosscut $g_t ( S_{\delta_n}^{\mathrm{in}} )$ separates 
$g_t(K_t \setminus K_{t-\delta_n})$ from infinity in $\bH$ for all $n \in \bZnn$. 
Consequently, by~\eqref{eq: wrong distance}, the null-chains 
$(S_n)_{n \in \bZnn}$ and $( g_t ( S_{\delta_n}^{\mathrm{in}} ) )_{n \in \bZnn}$ 
are not equivalent as null-chains in $\bH$. 
This is a contradiction, as by Proposition~\ref{prop: characterise grown/growing end} (or proof of Proposition~\ref{prop: constructing null-chains}), 
$( S_{\delta_n}^{\mathrm{in}} )_{n \in \bZnn}$ is a null-chain in $\bH \setminus K_t$  corresponding to $W(t-)$. In conclusion, we have $W(t-) \in \tilde{K}^{t-s}_{s}$ for all $s \in (0, t)$,
which implies by~\eqref{eq: identication left limit} that
\begin{align*}
	\hspace*{122pt}
	\bigcap_{0 < s < t} \overline{g_{t - s}(K_t \setminus K_{t - s})} = \bigcap_{0 < s < t} \tilde{K}^{t-s}_{s} = \{W(t-)\}.
	 \hspace*{122pt}
 	\qedhere
\end{align*}
\end{proof}

\begin{figure}[ht]
\centering
\includegraphics[width=0.6\textwidth]{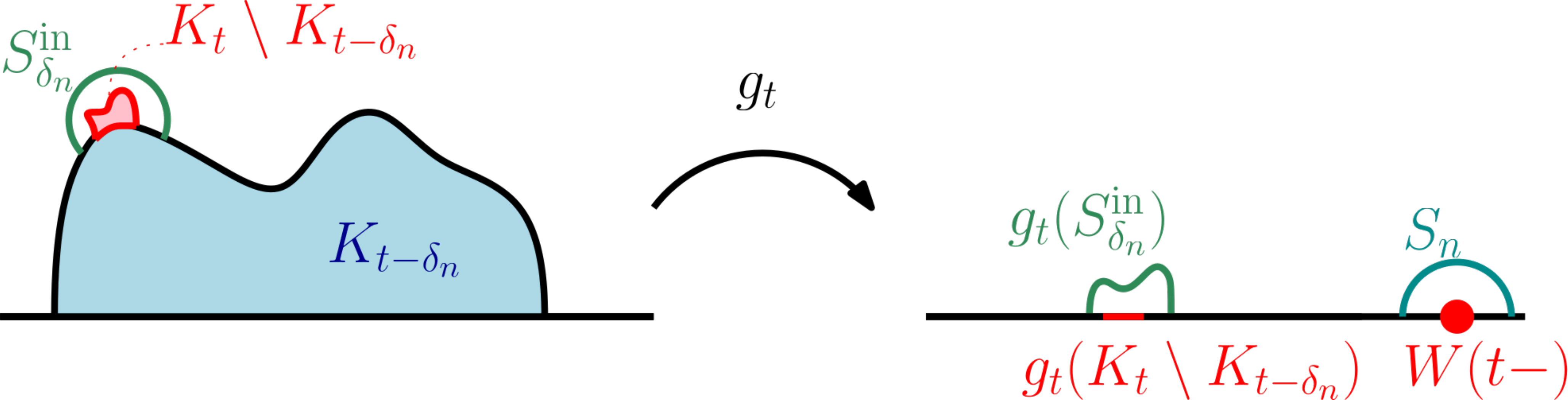}
\caption{Illustration for the proof of Proposition~\ref{prop: first construction driving measure}.}
\end{figure}

\smallskip
\subsection{Proof of Loewner's theorem: Constructing the driving function}
\label{subsec: construct driving measure}

In this section, we prove that every family $(K_t)_{t \geq 0}$ of locally growing hulls admits a 
c\`adl\`ag function $W\colon [0, \infty) \to \bR$ such that the mapping-out functions 
$g_t\colon \bH \setminus K_t \to \bH$ solve the Loewner equation~\eqref{eq: LE} with driving function $W$. 
This proves the direction~\ref{item: loewner hulls}~$\Longrightarrow$~\ref{item: loewner driver} in Theorem~\ref{thm: loewner intro}
and shows that the correspondence is a bijection.

As we have seen in Proposition~\ref{prop: first construction driving measure}, the key idea is to define $W(t)$ via
\begin{align}
\label{eq: candidate driving fct}
\{W(t)\} = \bigcap_{s > 0} \overline{g_{t} (K_{t + s} \setminus K_t)} = \bigcap_{s > 0} \tilde{K}^{t}_s.
\end{align}
Hence, it is crucial to show that the sets on the right-hand side do shrink to (at most) a single point:

\begin{lem}
\label{lemma: single real point}
Let $\boldsymbol{K} = (K_t)_{t \geq 0}$ be a family of strictly locally growing hulls.
For each $t \geq 0$, the intersection $\bigcap_{s > 0} \overline{g_{t} (K_{t + s} \setminus K_t)} = \bigcap_{s > 0} \tilde{K}^{t}_s$ is a single point on the real line.
\end{lem}

\begin{proof}
Note that by the strict growth of $\boldsymbol{K}$, the set
$\tilde{K}^{t}_{s} := \overline{g_t(K_{t + s} \setminus K_t)}$ is non-empty and compact for all $s > 0$. 
Therefore, by Cantor's intersection theorem (e.g.~\cite[Proposition~1.6(a)]{Dugundji:Topology}), we have
\begin{align*}
A_t := \bigcap_{s > 0} \tilde{K}^{t}_{s}
\neq \emptyset.
\end{align*}
Moreover, we have $\diam ( A_t ) = 0$ by Lemma~\ref{lemma: ca receding balls}. 
These facts imply that $A_t = \{z_0\}$ for some $z_0 \in \overline{\bH}$. 
It remains to show that $z_0 \in \bR$. To this end, note that for all $s > 0$, 
$\tilde{K}^{t}_{s} \cap \bR = \overline{g_{t} (K_{t + s} \setminus K_t)} \cap \bR \neq \emptyset$.
Therefore, by Cantor's intersection theorem again, we conclude that 
\begin{align*}
\emptyset \neq \bigcap_{s > 0} \tilde{K}^{t}_{s} \cap \bR \subset A_t = \{z_0\}
\qquad \Longrightarrow \qquad z_0 \in \bR .
\end{align*}
\end{proof}

It remains to prove that the function $W\colon [0, \infty) \to \bR$ defined via~\eqref{eq: candidate driving fct}
is c\`adl\`ag and indeed yields the driving function of $\boldsymbol{K}$ in the sense of Loewner's equation
(see Proposition~\ref{prop: W cadlag} and Theorem~\ref{thm: LE final}).

\begin{prop}\label{prop: W cadlag}
Let $\boldsymbol{K} = (K_t)_{t \geq 0}$ be a family of strictly locally growing hulls. Define \textnormal{(}cf.~Lemma~\ref{lemma: single real point}\textnormal{)}
\begin{align*}
\{W(t)\} := \bigcap_{s > 0} \overline{g_{t} (K_{t + s} \setminus K_t)} = \bigcap_{s > 0} \tilde{K}^{t}_s , \qquad t \geq 0 .
\end{align*}
Then, the function $W\colon [0, \infty) \to \bR$ is c\`adl\`ag. 
\end{prop}

\begin{proof}
Fix $t \geq 0$ and $\varepsilon > 0$.
By Lemma~\ref{lemma: ca receding balls}\ref{item: ca receding balls right}, there exists $\delta > 0$ such that 
$\mathrm{diam}(\tilde{K}^{t}_{s}) \leq \varepsilon$ for all $s \in [0, \delta]$.
By the definition of $W$, we obtain
$\tilde{K}^{t}_{s} \subset \tilde{K}^{t}_{\delta} \subset B(W(t), 2 \varepsilon)$ for all $s \in (0, \delta]$.
Hence, we have
\begin{align}
\label{eq: helper driving fct right-cont 2}
\tilde{g}^{t}_{s} \big( B(W(t), 4 \varepsilon) \cap ( \bH \setminus \tilde{K}^{t}_{s} ) \big) 
\subset B(W(t), 6 \varepsilon) , \qquad s \in [0, \delta] ,
\end{align}
by Lemma~\ref{lem: some markov property} and~\cite[Lemma~4.5]{Kemppainen:SLE_book}.
By the uniqueness of the mapping-out function, we see that
\begin{align*}
W(t + s) \in  \;& \tilde{K}^{t + s}_{\delta - s} =
\overline{g_{t + s} ( K_{t + \delta} \setminus K_{t + s} )} \subset \overline{B(W(t), 6 \varepsilon)} , 
\qquad s \in [0, \delta]. 
\end{align*}
As $\varepsilon > 0$ is arbitrary, this shows that $W$ is right-continuous at $t$. 
It remains to show the existence of left limits for $W$.
By Lemma~\ref{lemma: ca receding balls}\ref{item: ca receding balls left}, 
there exists $\delta > 0$ such that $\mathrm{diam}( \tilde{K}^{t-\delta}_{\delta} ) \leq \varepsilon$. 
In particular, because $\tilde{K}^{t-\delta}_{\delta} \cap \bR \neq \emptyset$, there exists $x_0 \in \bR$ such that 
$\tilde{K}^{t-\delta}_{\delta - s} \subset \tilde{K}^{t-\delta}_{\delta} \subset B(x_0, 2 \varepsilon)$  for all $s \in (0, \delta)$
--- and using~\cite[Lemma~4.5]{Kemppainen:SLE_book}, we have 
\begin{align}
\label{eq: helper driving fct left-cont 2}
\tilde{g}^{t-\delta}_{\delta - s} \big( B(x_0, 4 \varepsilon \big) \cap ( \bH \setminus \tilde{K}^{t-\delta}_{\delta - s} ) \big) 
\subset B(x_0, 6 \varepsilon) , 
\qquad s \in [0, \delta] . 
\end{align}
We thus obtain $W(t-s) \in \tilde{K}^{t-s}_{s} \subset \overline{B(x_0, 6 \varepsilon)}$, which implies that $W$ has a unique left limit at $t$. 
\end{proof}

Finally, we prove that $W\colon [0, \infty) \to \bR$ is indeed the driving function of $(K_t)_{t \geq 0}$:

\begin{thm} \label{thm: LE final}
Let $\boldsymbol{K} = (K_t)_{t \geq 0}$ be a family of right-locally growing hulls parametrized by capacity.   
The mapping-out functions $g_t\colon \bH \setminus K_t \to \bH$ solve the Loewner equation~\eqref{eq: LE} with driving function
\begin{align*}
\{W(t)\} := \bigcap_{s > 0} \overline{g_{t} (K_{t + s} \setminus K_t)} = \bigcap_{s > 0} \tilde{K}^{t}_s , \qquad t \geq 0 .
\end{align*}
\end{thm}

This is the same as usual proof, see, e.g.,~the end of the proof of~\cite[Theorem~4.1]{Kemppainen:SLE_book}.
It is interesting to observe, however, that only the right-local growth from Definition~\ref{def: local growth} is needed to verify the right-sided differential equation~\eqref{eq: LE}.
Assuming also the left-local growth from Definition~\ref{def: local growth},
we see that the map $t \mapsto g_t(z)$ is continuous by Corollary~\ref{cor: gt cont in time}, and $t \mapsto g_t(z)$ is then also absolutely continuous.

\begin{proof}  
Fix $t \geq 0$, $z \in \bH \setminus K_t$, and $R > 0$ with $K_t \subset B(0, R)$. 
Set $r_0(\varepsilon) = 6 \pi R / \sqrt{\log (1/\varepsilon)}$ as in Corollary~\ref{cor: conformal distortion}. 
Since $r_0(\varepsilon)$ tends to zero as $\varepsilon \to 0$, we can choose
$\varepsilon \in (0, 1)$ sufficiently small such that 
\begin{align}
\label{eq: helper differentiability}
\im(g_t(z)) \geq (2C_0 + 6) \,r_0(\varepsilon),
\end{align}
where $C_0 > 0$ is an absolute constant from~\cite[Lemma~4.7]{Kemppainen:SLE_book}. 
By the right local growth of $(K_t)_{t \geq 0}$ (Definition~\ref{def: local growth}), 
there exist $\delta > 0$ and a crosscut $S_{\delta}^{\mathrm{out}} \subset \bH \setminus K_t$ separating $K_{t + \delta} \setminus K_t$ from $\infty$ with $\mathrm{diam}(S_{\delta}^{\mathrm{out}}) < \varepsilon$
and $z \in \bH \setminus K_{t + \delta}$.
By Corollary~\ref{cor: conformal distortion}\ref{item: diameter control on right crosscut}, we have 
$\mathrm{diam}\big(\tilde{K}^t_{\delta}\big) \leq r_0(\varepsilon)$, and 
by the construction of $W(t)$ in~\eqref{eq: candidate driving fct}, we have
$\tilde{K}^t_{\delta} \subset \overline{B(W(t), 2 r_0(\varepsilon))}$.  
Applying~\cite[Lemma~4.7]{Kemppainen:SLE_book} and Lemma~\ref{lem: some markov property} to the mapping-in-function $\tilde{f}^t_{\delta} := (\tilde{g}^t_{\delta})^{-1} = g_t \circ g_{t + \delta}^{-1}$, we have
\begin{align}
\label{eq: differentiability g}
\bigg| \tilde{f}^t_{\delta}(w) - w + \frac{\mathrm{hcap} (\tilde{K}^t_{\delta})}{w - W(t)} \bigg| \leq C_0 \, \frac{2 r_0(\varepsilon) \, \mathrm{hcap} (\tilde{K}^t_{\delta})}{|w - W(t)|^2} ,
\end{align}
for all $w \in \bH$ satisfying $|w - W(t)| \geq C_0 \, 2 r_0(\varepsilon)$. 
For the right-hand side of~\eqref{eq: differentiability g}, 
by the parametrization by capacity and additivity of the half-plane capacity 
(see, e.g.,~\cite[Exercise 2.15]{Beliaev:Conformal_maps_and_geometry}), we have
\begin{align*}
2 (t+\delta) &= \mathrm{hcap}(K_{t + \delta}) 
= \mathrm{hcap}(K_t) + \mathrm{hcap} (\tilde{K}^t_{\delta})
= 2 t + \mathrm{hcap} (\tilde{K}^t_{\delta}),
\end{align*}
so $\mathrm{hcap} (\tilde{K}^t_{\delta}) = 2\delta$. 
Thus,~\eqref{eq: differentiability g} simplifies to 
\begin{align}
\label{eq: differentiability g 2}
\bigg| \tilde{f}^t_{\delta}(w) - w + \frac{2 \delta}{w - W(t)} \bigg| \leq C_0 \, \frac{4 r_0(\varepsilon) \, \delta}{|w - W(t)|^2} .
\end{align}
To finish, we take $w = g_{t + \delta}(z)$ --- but we need to check that
$|g_{t + \delta}(z) - W(t)| \geq C_0 \, 2 r_0(\varepsilon)$ holds.
Because $\tilde{K}^t_{\delta} \subset \overline{B(W(t), 2 r_0(\varepsilon))}$, 
by~\cite[Lemma~4.5]{Kemppainen:SLE_book} we have
$| \tilde{g}^t_{\delta}(v) - v | \leq 10 r_0(\varepsilon)$ for all $v \in \bH \setminus g_t(K_{t + \delta} \setminus K_t)$.
Therefore, we obtain 
$|g_{t+\delta}(z) - g_t(z) | =  | \tilde{g}^t_{\delta} (g_{t}(z)) - g_t(z) | \leq 10 r_0(\varepsilon)$.
By our initial choice~\eqref{eq: helper differentiability}, 
\begin{align*}
| g_{t+\delta}(z) - W(t) |
&\geq 
|g_t(z) - W(t) | - | g_{t+\delta}(z) - g_{t}(z) |
\geq \im(g_t(z)) - 10 r_0(\varepsilon) 
\geq 2 C_0 \, r_0(\varepsilon).
\end{align*}
Therefore, choosing $w = g_{t+\delta}(z)$ in (\ref{eq: differentiability g 2}) we obtain 
\begin{align}
\label{eq: helper hulls solve LE}
\bigg| \tilde{f}^t_{\delta}(g_{t+\delta}(z)) - g_{t+\delta}(z) + \frac{2 \delta}{g_{t+\delta}(z) - W(t)} \bigg| 
\leq C_0 \, \frac{4r_0(\varepsilon) \, \delta}{|g_{t+\delta}(z) - W(t)|^2}.
\end{align}
Moreover, $\tilde{f}^t_{\delta}(g_{t+\delta}(z)) = ( g_t \circ g_{t + \delta}^{-1} )(g_{t+\delta}(z))  = g_t(z)$. 
Thus, by dividing~\eqref{eq: helper hulls solve LE} by $\delta$ we find that
\begin{align*}
\bigg| \frac{g_{t+\delta}(z) - g_t(z)}{\delta} - \frac{2}{g_{t+\delta}(z) - W(t)} \bigg| 
\leq \frac{4 C_0 \, r_0(\varepsilon)}{|g_{t+\delta}(z) - W(t)|^2}
\; \overset{\varepsilon \to 0+}{\underset{\delta \to 0+}{\longrightarrow}} \; 0 ,
\end{align*}
implying~\eqref{eq: LE}: $\pderr{t}g_t(z) = \frac{2}{g_t(z) - W(t)}$. 
Clearly, the initial value $g_0(z)=z$ also holds. 
\end{proof}


\bigskip{}
\section{Boundary behavior of locally growing hulls}
\label{sec: boundary behavior of locally groing hulls}
In this section, we study the boundary growth of a family of locally growing hulls $\boldsymbol{K} = (K_t)_{t \geq 0}$.  
The main result is Theorem~\ref{thm: decomposition added stuff to hull}, which fully classifies all added points to the hull $K_t$ at some time $t$.  
In particular, the concept of swallowed points is well-defined for any family of locally growing hulls (Proposition~\ref{prop: bubbles}).
The main Theorem~\ref{thm: decomposition added stuff to hull intro} then follows immediately.
In Proposition~\ref{prop: A}, we present a technical result that can often conclude that a given behavior for the hulls occurs at a critical time, i.e.,~the infimum of all times witnessing that behavior.
We apply this result in Corollary~\ref{cor: Kt are grown} and Proposition~\ref{prop: not path-connected at stopping time}. 

\subsection{Structure of the hulls --- Proof of Theorem~\ref{thm: decomposition added stuff to hull intro} }
\label{subsec:thm14}

We refer to the notions of grown, growing, and principal points and ends from the beginning of Section~\ref{sec: Loewners Theorem}.

\begin{lem}
\label{lemma: grown points and principal growing points}
Let 
$\boldsymbol{K}$ 
be a family of locally growing hulls. 
Then, for all $t \geq 0$, the following holds.
\begin{enumerate}[label=\textnormal{(\arabic*):}, ref=\textnormal{(\arabic*)}]
\item\label{item: grown points} If $z \in \bigcap_{s < t} \overline{K_t \setminus K_s}$, then $z$ is a grown point at time $t$.

\smallskip

\item\label{item: grown principal points} If $z \in \bigcap_{s < t} (K_t \setminus K_s) \cap \bdry K_t$, then $z$ is a principal point of the grown end at time $t$.

\smallskip

\item\label{item: growing principal points} If $z \in \bigcap_{s > t} \overline{K_s \setminus K_t}$, then $z$ is a principal point of the growing end at time $t$.
\end{enumerate}
\end{lem}

\begin{proof}
Fix $t \geq 0$. 
Item~\ref{item: grown points} is immediate from the definition.
Consider $z$ as in Item~\ref{item: grown principal points}. 
Then, $z \in \bdry K_t$, so $B(z, \varepsilon) \cap (\bH \setminus K_t) \neq \emptyset$ for all $\varepsilon > 0$. 
Let $(S_n)_{n \in \bZnn}$ be a null-chain representing the grown end at time $t$.
As $z \in \bigcap_{s < t} K_t \setminus K_s$, it is disconnected from $\infty$ by $S_n$ for all $n \in \bZnn$. 
Thus, we find $\varepsilon_n \in (0, 2^{-n})$ and connected components $L_n$ of $\bdry B(z, \varepsilon_n) \cap (\bH \setminus K_t)$ such that $L_n$ is separated from $\infty$ by $S_n$. 
These arcs $(L_n)_{n \in \bZnn}$ can be used to define a null-chain (of the grown end at time $t$) in $\bH \setminus K_t$ whose principal point is $z$.
This shows Item~\ref{item: grown principal points}. 
Lastly, consider $z$ as in Item~\ref{item: growing principal points}. 
Then, $z \in \bdry K_t$ and thus, $B(z, \varepsilon) \cap (\bH \setminus K_t) \neq \emptyset$ for all $\varepsilon > 0$. 
Let $(S_n)_{n \in \bZnn}$ be a null-chain representing the growing end at time $t$.
As $z \in \bigcap_{s > t} \overline{K_s \setminus K_t}$, due to the strict growth of 
$\boldsymbol{K}$,  
the point $z$ is disconnected from $\infty$ by $S_n$ for all $n \in \bZnn$. 
Thus, we find an $\varepsilon_n \in (0, 2^{-n})$ and connected components $L_n$ of $\bdry B(z, \varepsilon_n) \cap (\bH \setminus K_t)$ such that $L_n$ is separated from $\infty$ by $S_n$. 
These arcs $(L_n)_{n \in \bZnn}$ can be used to define a null-chain (of the growing end at time $t$) in $\bH \setminus K_t$ whose principal point is $z$. 
This shows Item~\ref{item: growing principal points}.
\end{proof}

\begin{rem}
The assumptions in Lemma~\ref{lemma: grown points and principal growing points} can be weakened. 
Namely, the proof of Item~\ref{item: grown principal points} only requires left-local growth, whereas the proof of Item~\ref{item: growing principal points} only uses right-local growth. 
\end{rem}

As discussed in Example~\ref{subsec: comb space}, the point $\{\ii\} = \bigcap_{s > 0} \overline{K_{3+s} \setminus K_3}$ is indeed the only principal point of the prime end with impression $\ii [0,1]$ of the comb space $K_3$. 
However, $\bigcap_{s > 0} \overline{K_{3} \setminus K_{3-s}} = \ii [0,1]$ contains non-principal points.
This is a first hint at a symmetry breaking between grown and growing points, which will play an important role in Section~\ref{sec: gen by fct}.

\begin{lem}
\label{lemma: grown principal points Kt minus Ks}
Let 
$\boldsymbol{K}$ 
be a family of left-locally growing hulls.  
Suppose that $\bigcap_{s < t} (K_t \setminus K_s) \cap \bdry K_t \neq \emptyset$ for some $t > 0$. 
Then, $\overline{\bigcap_{s < t} K_t \setminus K_s} \cap \bdry K_t$ is the set of all principal points of the grown end at time $t$. 
\end{lem}

\begin{proof}
Item~\ref{item: grown principal points} of Lemma~\ref{lemma: grown points and principal growing points} shows that each
$z \in \bigcap_{s < t} (K_t \setminus K_s) \cap \bdry K_t$ is a principal point of the grown end at time $t$.
On the other hand, we will show that all such points lie in $\overline{\bigcap_{s < t} K_t \setminus K_s} \cap \bdry K_t$. 
Because the set of all principal points of the grown end at time~$t$ is compact by~\cite[Theorem~7.1]{Epstein:Prime_ends},
these two facts imply that $\overline{\bigcap_{s < t} K_t \setminus K_s} \cap \bdry K_t$ is the set of all principal points of the grown end at time $t$. 

To show that all principal points of the grown end at time $t$ lie in $\overline{\bigcap_{s < t} K_t \setminus K_s} \cap \bdry K_t$, conversely, consider 
$w \in \bdry K_t \setminus \overline{\bigcap_{s < t} K_t \setminus K_s}$. 
Set $\varepsilon := \dist(w, \bigcap_{s < t} K_t \setminus K_s) > 0$. 
Let $(S_n)_{n \in \bZnn}$ be a null-chain representing the grown end at time $t$. 
Then, there exists $N \in \bZnn$ such that for all $n \geq N$, $\diam(S_n) < \varepsilon/2$. 
Because $(S_n)_{n \geq N}$ is a null-chain equivalent to $\smash{(S_{\delta_n}^{\mathrm{in}})_{n \in \bZnn}}$, 
it separates $\bigcap_{s < t} K_t \setminus K_s$ from $\infty$ in $\bH \setminus K_t$. 
Hence, if $(S_n)_{n \geq N}$ disconnects $w$ from $\infty$ in $\bH \setminus K_t$, then $w$ is already disconnected from $\infty$ in $\bH \setminus K_t$. 
This contradicts the fact that $w \in \bdry K_t$. 
In particular, $w$ is not a principal point of the grown end at time $t$. 
This shows that all principal points of the grown end at time $t$ lie in $\overline{\bigcap_{s < t} K_t \setminus K_s} \cap \bdry K_t$. 
\end{proof}


\begin{prop}
\label{prop: bubbles}
Let 
$\boldsymbol{K}$ 
be a family of left-locally growing hulls. 
Then, for all $t > 0$, the set
\begin{align*}
B_t := \Big( \bigcap_{s < t} \overline{K_t \setminus K_s} \Big) \setminus ( \bdry K_t \cup \bR )
\end{align*}
is either empty, or open, path-connected, and simply connected. 

We call $B_t$ the \emph{bubble} at time $t$, and say that $z \in B_t$ is \emph{swallowed} at time $t$. 
\end{prop}

\begin{proof}
Fix $t > 0$. If $B_t = \emptyset$, then the claim is trivial, so we assume that $B_t \neq \emptyset$. 
By the left-local growth, for each $\varepsilon > 0$, we find crosscuts $S^{\mathrm{in}}_{\delta}$ 
separating $K_t \setminus K_{t- \delta}$ from $\infty$ in $\bH \setminus K_{t- \delta}$ and\footnote{See Remark~\ref{remark: Left local null chains}.}
 in $\bH \setminus K_{t}$,
such that $\diam(S^{\mathrm{in}}_{\delta}) < \varepsilon/2$. 
In particular, both endpoints of $S^{\mathrm{in}}_{\delta}$ lie in $\bdry K_t \cap \bdry K_{t- \delta}$. 
We will leave the dependence on $t$ and $\varepsilon$ for $\delta = \delta(\varepsilon, t)$ implicit and choose different values for $\varepsilon$ throughout the proof.

{\bf Step 1: $B_t$ is open.}
Let $z \in B_t$. 
Then, $z \in K_t$ and $\varepsilon := \dist(z, \bdry K_t \cup \bR) > 0$. 
As $K_t$ is compact, there exists $w \in \bdry K_t \cup \bR$ such that $\varepsilon = |z - w|$.
Because 
$\boldsymbol{K}$ 
are left-locally growing, they are also left-continuously growing by Lemma~\ref{lemma: local growth => continuous growth}. 
Hence, by definition, there exists $\delta \in (0, t)$ such that 
\begin{align}
\label{eq: helper prop bubbles 2}
\bdry K_t = \bdry (\bH \setminus K_t) \subset K_{t-\delta}^{\varepsilon/2} \cup \bR^{\varepsilon/2} 
\subset K_{s}^{\varepsilon/2} \cup \bR^{\varepsilon/2} \qquad \textnormal{for all } s \in (t - \delta, t) .
\end{align}
Thus, we have 
$|z - u| \geq |z - w| - |w - u| \geq \frac{1}{2} \varepsilon > 0$ for all 
$s \in (t - \delta, t)$ and $u \in K_s \cup \bR$, so
\begin{align}
\label{eq: helper prop bubbles A} 
B (z, \varepsilon/2 ) \subset \bH \setminus K_s \qquad \textnormal{for all }  s < t .
\end{align}
In particular, we have $z \in \bigcap_{s < t}  K_t \setminus K_s$. 
Moreover, by~\eqref{eq: helper prop bubbles A} and the choice of $\varepsilon$, 
there exists $\delta_0 > 0$ such that $S^{\mathrm{in}}_{\delta} \cap B(z, \varepsilon/2) = \emptyset$ for all $\delta < \delta_0$. 
Hence, because $S^{\mathrm{in}}_{\delta}$ separates $z \in \bigcap_{s < t}  K_t \setminus K_s$ from $\infty$ in $\bH \setminus K_{t}$ for all $\delta < \delta_0$, 
it also separates $B(z, \varepsilon/2)$ from $\infty$ in $\bH \setminus K_{t}$ for all $\delta < \delta_0$. 
By taking $\delta \rightarrow 0$, we obtain
\begin{align}
\label{eq: helper prop bubbles B}
B(z, \varepsilon/2) \subset K_t. 
\end{align}
In particular, the inclusions~(\ref{eq: helper prop bubbles A},~\ref{eq: helper prop bubbles B}) imply that 
\begin{align*}
B(z, \varepsilon/2) 
\subset ( K_t \setminus \bdry K_t ) \cap \bigcap_{s < t} ( \bH \setminus K_s ) 
= \Big( \bigcap_{s < t} K_t \setminus K_s \Big) \setminus (\bdry K_t \cup \bR) 
\subset B_t.
\end{align*}
This proves that $B_t$ is open, as $z \in B_t$ was arbitrary.

{\bf Step 2: $B_t$ is path-connected.}
Let $u, w \in B_t$ be arbitrary.  
Set $\varepsilon := \dist ( \{u, w\}, \bdry K_t \cup \bR ) > 0$.
Note that $S^{\mathrm{in}}_{\delta}$ separates $\bH \setminus K_{t - \delta}$ into two open connected sets; 
one of which is unbounded, the other, denoted $U^{\mathrm{in}}_\delta$, being bounded and containing $K_t \setminus K_{t- \delta} \ni u, w \in B_t$. 
In particular, $U^{\mathrm{in}}_\delta$ is path-connected\footnote{An open and connected subset of $\bC$ is path-connected.}. 
Let us assume that $(S^{\mathrm{in}}_{\delta_n})_{n \in \bZnn}$ has principal point $z_0 \in \bdry K_t$ and that 
$S^{\mathrm{in}}_{\delta_n} \subset B(z_0, 2^{-n})$ for all $n \in \bZnn$. 
Then, by construction 
$(U^{\mathrm{in}}_{\delta_n} \cup B(z_0, 2^{-n}) )_{n \in \bZnn}$, is a shrinking sequence of open, path-connected sets and 
\begin{align}
\label{eq: helper bubble as intersection}
u, w \in B_t \subset \bigcap_{n \in \bZnn} \big(U^{\mathrm{in}}_{\delta_n} \cup B(z_0, 2^{-n}) \big) \subset B_t \cup \{z_0\}.
\end{align}
This proves that $u$ and $w$ are path-connected in $B_t$, because we can always choose a path connecting $u$ and $w$ in 
$U^{\mathrm{in}}_{\delta_n} \cup B(z_0, 2^{-n})$ that does not pass through $z_0$. 
Hence, for sufficiently large $n \in \bZnn$ we find a path connecting $u$ and $w$ in 
$U^{\mathrm{in}}_{\delta_n} \cup B(z_0, 2^{-n})$ that does not pass through $B(z_0, 2^{-n})$ and stays in $B_t$. 
By construction, it is a path connecting $u$ and $w$ in $B_t$. This shows that $B_t$ is path-connected. 

\begin{figure}[ht]
\centering
\includegraphics[width=0.3\textwidth]{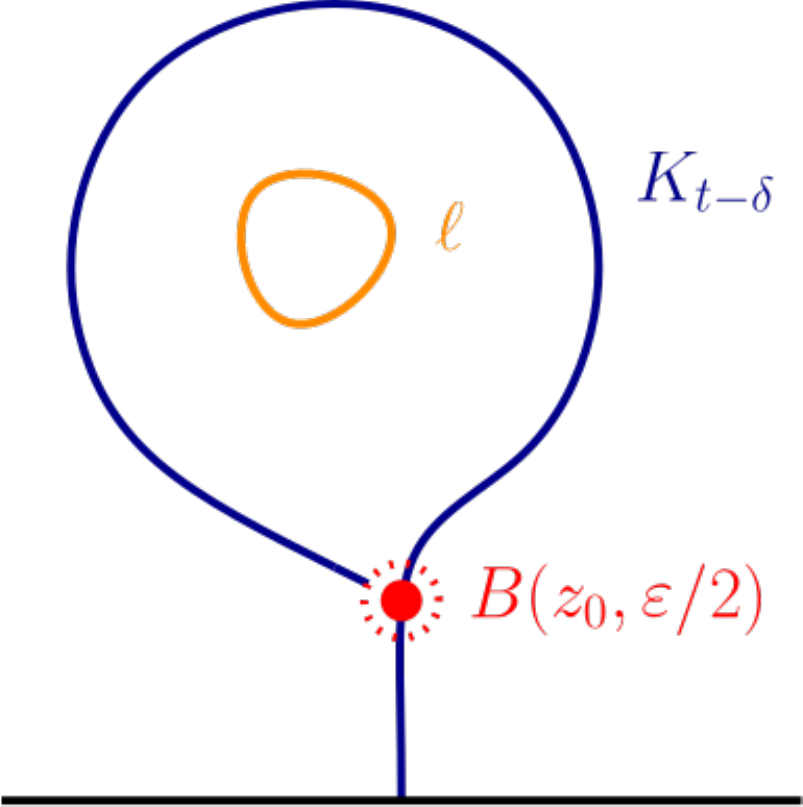}
\caption{\label{fig: closing bubble}Illustration of locally growing hulls with a closing bubble (Proposition~\ref{prop: bubbles}).}
\end{figure}

{\bf Step 3: $B_t$ is simply connected.}
Because $B_t$ is path-connected, it remains to show that every closed loop $\ell$ in $B_t$ is contractible in $B_t$. 
Hence, let $\ell$ be a closed loop in $B_t$.
As $B_t \cap (\bdry K_t \cup \bR) = \emptyset$, we have $\varepsilon := \dist (\ell, \bdry K_t \cup \bR ) > 0$. 
Note that by~\eqref{eq: helper bubble as intersection}, every point inside of $\ell$ lies in $B_t$. 
Moreover, there exist $z_0 \in \bdry K_t \cup \bR$ and $\delta > 0$ such that $S^{\mathrm{in}}_{\delta} \subset B(z_0, \varepsilon/2)$, so both imply that 
\begin{align*}
\ell \subset B_t \setminus B(z_0, \varepsilon/2) \subset \bH \setminus ( K_{t-\delta} \cup B(z_0, \varepsilon/2) )
\subset \bH \setminus K_{t-\delta} . 
\end{align*}
Hence, $\ell$ is a closed loop in $\bH \setminus K_{t- \delta}$. 
As $\bH \setminus K_{t- \delta}$ is simply connected, $\ell$ is contractible in $\bH \setminus K_{t- \delta}$. 
Therefore, since every point inside of $\ell$ lies in $B_t$, the loop $\ell$ is contractible in $B_t$. 
\end{proof}

In the course of the proof, we have in fact obtained a slightly stronger result:

\begin{cor}
\label{cor: bubbles}
Let 
$\boldsymbol{K}$ 
be a family of left-locally growing hulls. 
Then, for all $t > 0$, we have
\begin{align*}
B_t = \Big( \bigcap_{s < t} K_t \setminus K_s \Big) \setminus ( \bdry K_t \cup \bR )
\qquad \textnormal{and} \qquad 
B_t \subset \bigcap_{s < t} \overline{K_t \setminus K_s} \subset B_t \cup \bdry K_t \cup \bR .
\end{align*}
\end{cor}

\begin{proof}
Proposition~\ref{prop: bubbles} gives
\begin{align*}
\Big( \bigcap_{s < t} K_t \setminus K_s \Big) \setminus ( \bdry K_t \cup \bR ) 
\; \subset \; 
B_t := \Big( \bigcap_{s < t} \overline{K_t \setminus K_s} \Big) \setminus ( \bdry K_t \cup \bR )
\; \subset \; \bigcap_{s < t} \overline{K_t \setminus K_s} 
\; \subset \; B_t \cup \bdry K_t \cup \bR .
\end{align*}
Conversely, in Step 1 of the proof of Proposition~\ref{prop: bubbles}, we have seen that $B_t \subset \bigcap_{s < t}  K_t \setminus K_s$. 
\end{proof}

We can take this argument further to fully classify all points added to $\boldsymbol{K} = (K_t)_{t \geq 0}$ at time $t$. 

\begin{thm}
\label{thm: decomposition added stuff to hull}
Let 
$\boldsymbol{K}$ 
be a family of left-locally growing hulls. 
Then, for all $t \geq 0$, we have
\begin{align*}
K_t \cup \bR = \Big( \bigcup_{s < t} K_s \cup \bR \Big) \cup B_t \cup P_t,
\end{align*}
where 
\begin{itemize}[leftmargin=*]
\item 
$B_t$ is the bubble at time $t$, 
i.e.,~$B_t = \big( \bigcap_{s < t} K_t \setminus K_s \big) \setminus ( \bdry K_t \cup \bR )$ 
and $B_t$ is either empty, or open, path-connected, and simply connected;

\item 
$P_t = ( \bdry K_t \cap \bH ) \setminus \bigcup_{s < t} K_s$ is compact and connected. 
\end{itemize}
Moreover, if $P_t \neq \emptyset$, then $P_t$ is the set of all principal grown points at time $t$. 
\end{thm}

\begin{proof}
This is a consequence of Proposition~\ref{prop: bubbles}, Corollary~\ref{cor: bubbles}, 
and Proposition~\ref{prop: new principal grown points} proven below. 
\end{proof}

\begin{example}
There are several possibilities for what $B_t$ and $P_t$ could be. 
Here are some examples. 
\begin{itemize}
\item 
If 
the family $\boldsymbol{K}$ 
is generated by a simple curve $\eta$, then $B_t = \emptyset$ and $P_t = \{\eta(t)\}$. 
Moreover, the grown points at time $t$ are 
$\bigcap_{s < t} \overline{K_t \setminus K_s} = \{\eta(t)\} = \bigcap_{s < t} K_t \setminus K_s$.

\smallskip

\item 
If 
$\boldsymbol{K}$ 
is generated by a non-simple curve $\eta$, at a time of self-intersection (e.g.,~Figure~\ref{fig: closing bubble}), 
$P_t = \emptyset$ and $\emptyset \neq B_t = \bigcap_{s < t} K_t \setminus K_s$.
Here, $B_t \cup \{\eta(t)\} \subsetneq \overline{B_t} = \bigcap_{s < t} \overline{K_t \setminus K_s}$ are the grown points. 

\smallskip

\item  
In Example~\ref{subsec: logarithmic spiral} at time $t = 1$, when the spiral closes, we have $B_{t} = B(2\ii, 1)$ and $P_t = \bdry B(2\ii, 1)$.
The grown points at time $t$ are $\overline{B(2\ii, 1)} = \bigcap_{s < t} \overline{K_t \setminus K_s} = \bigcap_{s < t} K_t \setminus K_s$.

\smallskip

\item 
For the comb space from Example~\ref{subsec: comb space} at the critical time $t = 3$, we have in fact $B_t = \emptyset = P_t$,
because $\bigcap_{s < t} K_t \setminus K_s = \emptyset$. 
However, $\bigcap_{s < t} \overline{K_t \setminus K_s} = \ii [0, 1]$ are the grown points at time $t = 3$. 
\end{itemize}
\end{example}

\begin{figure}[ht]
\centering
\includegraphics[width=0.3\textwidth]{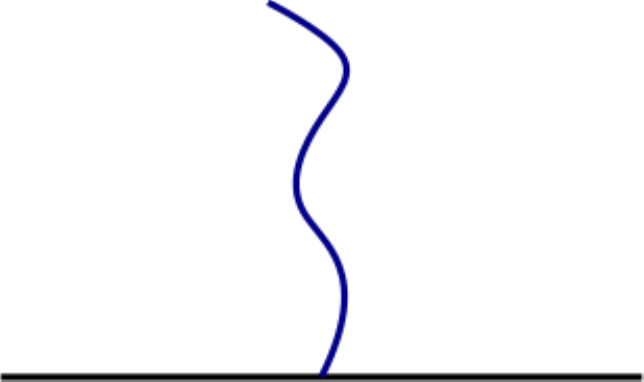}
\caption{Locally growing hulls generated by a simple curve.}
\end{figure}

\begin{rem}
By~\cite[Section~5]{LMR:Collisions_and_spirals_of_Loewner_traces}, every open, simply connected, and bounded set $A \subset \bH$ can be the bubble of a locally growing Loewner chain. 
In fact,~\cite{LMR:Collisions_and_spirals_of_Loewner_traces} establishes a stronger claim: 
For a compact, connected set $A \subset \bH$, there exists a sufficiently smooth spiral winding infinitely often around $A$ with limit set $\bdry A$. 
In the setting of Theorem~\ref{thm: decomposition added stuff to hull}, if we call $t_0 > 0$ the limiting time at which the spiral closes, 
then we have $B_{t_0} = \mathrm{int}(A)$ and $P_{t_0} = \bdry A$. 
\end{rem}

\begin{prop}
\label{prop: new principal grown points}
Let 
$\boldsymbol{K}$ 
be a family of left-locally growing hulls. 
If there exists $t \geq 0$ such that
\begin{align}\label{eq:assumption_emptyset}
\emptyset \neq P_t = ( \bdry K_t \cap \bH ) \setminus \bigcup_{s < t} K_s ,
\end{align} 
then $P_t$ are the principal points of the grown end at time $t$. 
In particular, $P_t$ is compact and connected. 
\end{prop}

\begin{proof}
By assumption, we have 
\begin{align*}
\emptyset \neq ( \bdry K_t \cap \bH ) \setminus \bigcup_{s < t} K_s 
\; \subset \; \bigcap_{s < t} (K_t \setminus K_s) \cap \bdry K_t. 
\end{align*}
Thus, by Lemma~\ref{lemma: grown principal points Kt minus Ks}, 
$\overline{\bigcap_{s < t} (K_t \setminus K_s)} \cap \bdry K_t$ is the set of all principal points of the grown end at time $t$. 
Therefore, it is sufficient to prove that 
\begin{align}
\label{eq: helper classification grown principal points}
\emptyset 
= \Big( \overline{\bigcap_{s < t} (K_t \setminus K_s)} \cap \bdry K_t \Big) 
\setminus P_t =: X_t. 
\end{align}
Towards a contradiction, if there were a point $z \in X_t$, 
then $z \in \bdry K_t \cap \big( \bigcup_{s < t} K_s \cup \bR \big)$ and 
$z$ was a principal grown point at time $t$. 
Thus, by the below Lemma~\ref{lemma: boundary closed outside of t}, 
this would imply that $\bdry K_t \subset \bigcup_{s < t} K_s \cup \bR$, 
which contradicts the assumption~\eqref{eq:assumption_emptyset}. 
This proves~\eqref{eq: helper classification grown principal points} and hence that $P_t$ is the set of principal points of the grown end at time $t$. 
Lastly, this set is compact and connected by~\cite[Theorem~7.1]{Epstein:Prime_ends}. 
\end{proof}

\begin{lem}
\label{lemma: boundary closed outside of t}
Let 
$\boldsymbol{K}$ 
be a family of left-locally growing hulls. 
Assume that there exist $t \geq 0$ and $z \in \bigcup_{s < t} K_s \cup \bR$ such that $z$ is a principal point of the grown end at time $t$. 
Then, $\bdry K_t \subset \bigcup_{s < t} K_s \cup \bR$. 
\end{lem}

\begin{figure}[ht]
\centering
\includegraphics[width=0.3\textwidth]{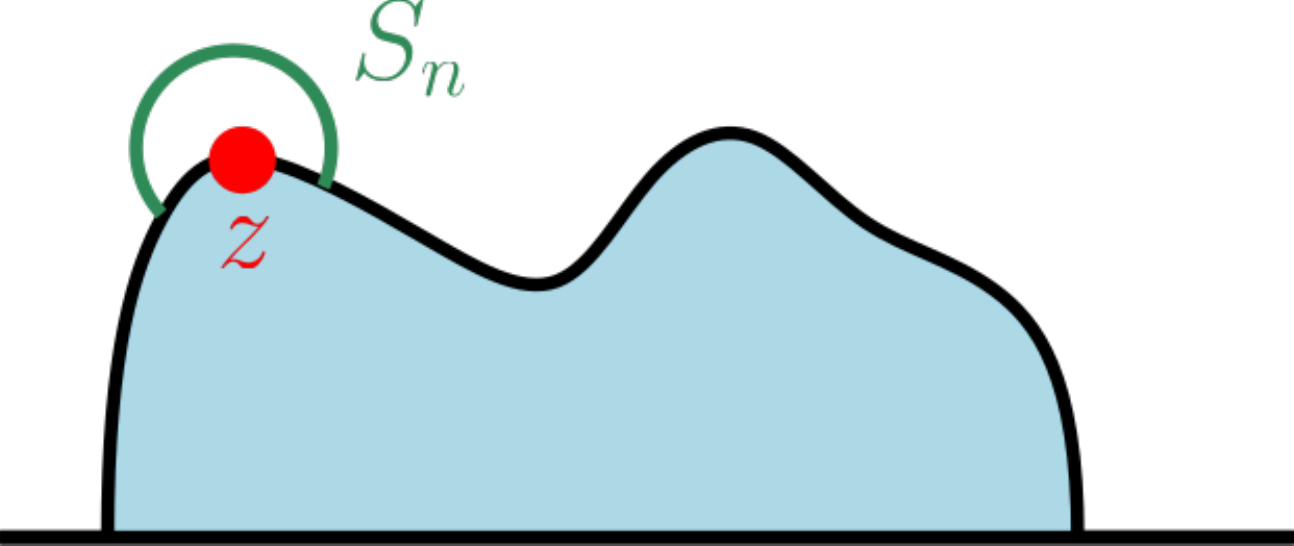}
\caption{Illustration for the proof of Lemma~\ref{lemma: boundary closed outside of t}: We have $z \in \bigcup_{s < t} K_s \cup \bR$.}
\end{figure}

\begin{proof}
Let $(S_n)_{n \in \bZnn}$ be a null-chain representing $z$.  
Then, for every $n \in \bZnn$, there exists $\varepsilon_n > 0$ such that $S_n \subset B(z, \varepsilon_n)$ and 
$\varepsilon_n \rightarrow 0$ as $n \rightarrow \infty$.  
Moreover, because $(S_n)_{n \in \bZnn}$ represents the grown end at time~$t$, 
for every $n \in \bZnn$, there exists $\delta_n > 0$ such that $S_n$ is a crosscut in $\bH \setminus K_{t - \delta_n}$ 
separating $K_t \setminus K_{t - \delta_n}$ from $\infty$, and $\delta_n \rightarrow 0$ as $n \rightarrow \infty$. 
In particular, this implies that for all $n \in \bZnn$, 
\begin{align*}
\bdry K_{t - \delta_n} \cup S_n \subset K_{t - \delta_n} \cup \bR \cup B(z, \varepsilon_n).
\end{align*}
Moreover, by assumption, we have
\begin{align*}
\bigcap_{n \in \bZnn} \big( K_{t - \delta_n} \cup \bR \cup B(z, \varepsilon_n) \big) 
\; \subset \; \bigcup_{s < t} K_s \cup \bR \cup \{z \} 
\; = \; \bigcup_{s < t} K_s \cup \bR.
\end{align*}
As $S_n$ separates $K_t \setminus K_{t - \delta_n}$ from $\infty$ in $\bH \setminus K_{t - \delta_n}$ for all $n \in \bZnn$, 
this yields $\bdry K_t \subset \bigcup_{s < t} K_s \cup \bR$.
\end{proof}

\smallskip
\subsection{Behavior at critical times}
\label{subsec:critical_times}

For right-continuously growing hulls, we can prove the following result.
It can be applied in order to study the first time that locally growing hulls display certain properties,
which becomes handy in analyzing the topology and geometry of the hulls ---  cf.~Proposition~\ref{prop: not path-connected at stopping time}.

\begin{prop}
\label{prop: A}
Let 
$\boldsymbol{K}$ 
be a family of right-continuously growing hulls. 
Let $\emptyset \neq L \subset \bH$ be arbitrary. 
Suppose that $\sigma_L := \inf\{ t \geq 0 \;|\; L \cap K_t \neq \emptyset \} < \infty$. 
Then, we have
\begin{align}
\label{eq: to show prop A}
\emptyset \neq L \cap K_{\sigma_L} \subset \big( K_{\sigma_L} \cap \bH \big) \setminus \bigcup_{s < \sigma_L} K_s. 
\end{align}
\end{prop}

\begin{proof}
By construction, we have 
\begin{align}
\label{eq: helper prop A}
L \cap \Big( \bigcup_{s < \sigma_L} K_s \cup \bR  \Big) = \emptyset,
\end{align}
and $\emptyset \neq L \cap (K_s \cup \bR) = L \cap K_s$ holds for all $s > \sigma_L$. 
Note that the set $\overline{L \cap K_s}$ is compact for all $s \geq 0$. 
Hence, by Cantor's intersection theorem, e.g.~\cite[Proposition~1.6(a)]{Dugundji:Topology},  we have
\begin{align*}
\emptyset 
\neq 
\bigcap_{s > \sigma_L} \overline{L \cap K_s} 
\; \subset \; 
\overline{L} \cap \big( \bigcap_{s > \sigma_L} (K_s \cup \bR ) \big).  
\end{align*} 
By the right-continuous growth 
(see Lemma~\ref{lem: hcap and growing}), this implies that
\begin{align*}
\emptyset 
\neq 
\overline{L} \cap \big( \bigcap_{s > \sigma_L} (K_s \cup \bR ) \big) 
= \overline{L} \cap (K_{\sigma_L} \cup \bR ) 
= 
\overline{L \cap (K_{\sigma_L} \cup \bR )},
\end{align*}
where the last equality holds because $K_{\sigma_L} \cup \bR = \overline{K_{\sigma_L} \cup \bR}$ is closed. 
In particular, this shows that $\emptyset \neq L \cap (K_{\sigma_L} \cup \bR )$. 
Together with~\eqref{eq: helper prop A} this implies~\eqref{eq: to show prop A}, as claimed. 
\end{proof}

The above result is very powerful, as the set $\emptyset \neq L \subset \bH$ can be completely arbitrary. 
One example is when we choose $L$ to be a singleton set --- see Corollary~\ref{cor: Kt are grown}. 
However, $L$ can also encode the first time when the locally growing hulls display a certain geometric or 
topological property, see Proposition~\ref{prop: not path-connected at stopping time}. 
Often it is possible to prove that this property already holds at this critical time. 

\smallskip 

\begin{cor}
\label{cor: Kt are grown}
Let 
$\boldsymbol{K}$ 
be a family of right-continuously growing hulls. 
Then, for all $t > 0$, we have
\begin{align*}
K_t \cap \bH = \bigcup_{s \in [0, t]} \Big( (K_s \cap \bH) \setminus \bigcup_{r < s} K_{r} \Big).
\end{align*}
\end{cor}

\begin{proof}
Let $t > 0$. 
Firstly, by the growth of $\boldsymbol{K}$, we have for all $s \in [0, t]$, 
\begin{align*}
(K_s \cap \bH ) \setminus \bigcup_{r < s} K_{r}  \subset K_t \cap \bH
\qquad \Longrightarrow \qquad 
\bigcup_{s \in [0, t]} \Big((K_s \cap \bH) \setminus \bigcup_{r < s} K_{r} \Big) \subset K_t \cap \bH . 
\end{align*}
Secondly, let $z \in K_t \cap \bH$. 
Applying Proposition~\ref{prop: A} to $L := \{z\} \subset K_t \cap \bH$, we find $\sigma_z \in [0, t]$ such that 
\begin{align*}
\emptyset \neq \{z\} \subset ( K_{\sigma_z} \cap \bH ) \setminus \bigcup_{s < \sigma_z} K_s. 
\end{align*}  
This concludes the proof. 
\end{proof}

\begin{prop}
\label{prop: not path-connected at stopping time}
Let 
$\boldsymbol{K}$ 
be a family of right-continuously growing hulls. 
Suppose that 
\begin{align*}
\rho := \inf \{ t \geq 0 \;|\; K_t \cup \bR \textnormal{ is not path connected} \} \, < \, \infty .
\end{align*} 
Then, $K_{\rho} \cup \bR$ is not path-connected. 
\end{prop}

\begin{proof}
Define 
\begin{align*}
L := \bigcup_{t \geq 0} \{ z \in K_t \;|\; z \textnormal{ is not path connected to zero in } K_t \cup \bR \} \subset \bH. 
\end{align*}
Then, $\sigma_L := \inf\{ t \geq 0 \;|\; L \cap K_t \neq \emptyset \} < \infty$ holds by assumption. 
We will show that $\rho = \sigma_L$.  
In that case, we have $\emptyset \neq L \cap K_{\sigma_L} = L \cap K_{\rho}$ by Proposition~\ref{prop: A}. 
Observe that by the choice of $\rho$, if $\rho = \sigma_L$, then 
\begin{align*}
\emptyset \neq L \cap K_{\rho} 
= \{ z \in K_{\rho} \;|\; z \textnormal{ is not path connected to zero in } K_{\rho} \cup \bR \}.
\end{align*}
Therefore, it is sufficient to prove that $\rho = \sigma_L$.
\begin{itemize}
\item Let $t > \rho$. 
Then, by the choice of $\rho$, there exists $s \in [\rho, t]$ such that $K_s \cup \bR$ is not path-connected. 
In particular, we then have $L \cap K_s \neq \emptyset$. 
Hence, $\sigma_L \leq s \leq t$, so 
\begin{align*}
\sigma_L \leq \inf_{t > \rho} t = \rho. 
\end{align*}
\item Let $t > \sigma_L$. 
Because $\boldsymbol{K}$ is growing, this implies that $\emptyset \neq L \cap K_t$. 
Hence, there exist $z \in K_t$ and $s \geq 0$ such that $z$ is not path-connected to zero in $K_s \cup \bR$. 
Set $r := \min(s, t)$. 
In particular, $z$ is not path-connected to zero in $K_r \cup \bR$, because otherwise $z$ would also be path-connected to zero in $K_r \cup \bR \subset K_s \cup \bR$. 
Therefore, we have $\rho \leq r \leq t$, so 
\begin{align*}
\rho \leq \inf_{t > \sigma_L} t = \sigma_L. 
\end{align*}
\end{itemize}
This concludes the proof. 
\end{proof}

\bigskip{}
\section{Hulls generated by a function}
\label{sec: gen by fct}
In this section, we analyze the geometry and topology of locally growing Loewner hulls. 
We are mainly interested in scenarios where the hulls can be identified with a function $\eta \colon [0, \infty) \to \overline{\bH}$:

\DefGenFct*

By Example~\ref{subsec: logarithmic spiral}, not every family of locally growing hulls is generated by a function. 
Furthermore, as shown in Examples~\ref{subsec: comb space} and~\ref{eg: double comb}, a generating function need not necessarily be continuous.

This section contains the main results of our work: 
\begin{itemize}[leftmargin=*]
\item In Theorem~\ref{thm: equivalence generated by a function} in Section~\ref{subsec:thm16}, 
we show that $\eta \colon [0, \infty) \to \overline{\bH}$ exists exactly when the union of the hulls and the real line is path-connected. 
Our first main Theorem~\ref{thm: generated by a fct intro} follows immediately from this.

\smallskip

\item In Theorem~\ref{thm: existence right-continuous generating function} in Section~\ref{subsec: right-continuous generating functions}, we show that if 
\begin{align}\label{eq:radl}
\eta(t) = \lim_{y \to 0+} g_t^{-1}(W(t-) + \ii y)
\end{align}
exists as a radial limit for all time, then the Loewner chain is generated by $\eta \colon [0, \infty) \to \overline{\bH}$. 
We also give a characterization for $\eta$ as a radial limit in terms of prime ends.
By Example~5.3, this generating function $\eta$ need not to be left-continuous. 
Interestingly,~\eqref{eq:radl} does imply that $\eta$ has a right-continuous version --- as we show in in Theorem~\ref{thm: Minimal regularity of generating functions}.
These results imply our second main Theorem~\ref{thm:coro_of_2}.

\smallskip

\item In Theorems~\ref{thm: locally path-connected}~\&~\ref{thm: left continuity implies locally connectedness} in Section~\ref{subsec: left-continuous generating function}, we prove that hulls generated by a left-continuous function are
(globally) path-connected and uniformly locally (path-)connected, 
and such a left-continuous generating function automatically has unique right limits (see Proposition~\ref{prop: left-cont implies right-cont}). 

\smallskip

\item Lastly, in Proposition~\ref{prop: empty interior locally connected equiv cadlag generated} we show that if the hulls have empty interior, then the reverse holds as well: 
Locally connected hulls are generated by a left-continuous function. 
\end{itemize}

\smallskip
\subsection{Generating functions of Loewner chains --- proof of Theorem~\ref{thm: generated by a fct intro}}
\label{subsec:thm16}

Assuming local growth, we obtain a precise characterization of the existence of generating functions (Theorem~\ref{thm: equivalence generated by a function}). To this end, we gather two auxiliary results in 
Lemma~\ref{lemma: first entry point of a path means unique principal point} and Proposition~\ref{prop: path-connectedness propagates forward}. We begin by investigating the set
\begin{align*}
P_t := (\bdry K_t \cap \bH) \setminus \bigcup_{s < t} K_s  .
\end{align*}

\begin{lem}
\label{lemma: first entry point of a path means unique principal point}
Let $\boldsymbol{K} = (K_t)_{t \geq 0}$ be a family of locally growing hulls. 
Let $t \geq 0$ be fixed.
Assume that there exists $z \in P_t$ 
and a path $\pi\colon [0, 1] \to K_t \cup \bR$ connecting $\pi(0) = 0$ to $\pi(1) = z$ such that $\pi[0,1) \subset \bigcup_{s < t} K_s \cup \bR$. 
Then, $z$ is the unique principal point of the grown end at time $t$. 
\end{lem}

\begin{proof}
By Proposition~\ref{prop: new principal grown points}, $z$ is a principal point of grown end at time $t$. 
Thus, it remains to prove that $z$ is unique. 
Let $(S_n)_{n \in \bZnn}$ be a null-chain representing $z$, and thus the grown end at time $t$.
Then, for every $n \in \bZnn$, there exists $\delta_n > 0$ such that 
$S_n$ is a crosscut in $\bH \setminus K_{t - \delta_n}$ separating $K_t \setminus K_{t- \delta_n}$ from $\infty$, 
and $\delta_n \to 0$ as $n \to \infty$. 
By assumption, we have $z \notin K_{t - \delta_n} \cup \bR$ which is closed, so that
\begin{align*}
\varepsilon_n := \dist (z, K_{t - \delta_n} \cup \bR ) > 0 \qquad \textnormal{for all } n \in \bZnn. 
\end{align*}
As $\pi$ is continuous for all $n \in \bZnn$, there exists $\rho_n \in (0, 1)$ such that $\pi(1 - \rho_n, 1] \subset B(z, \varepsilon_n)$. 
Therefore, $S_n$ separates $\pi(1 - \rho_n, 1] \cap (K_t \setminus (K_{t-\delta_n} \cup \bR))$ from $\infty$ in $\bH \setminus K_{t - \delta_n}$. 
In particular, every null-chain $(C_n)_{n \in \bZnn}$ that is equivalent to $(S_n)_{n \in \bZnn}$ has the same property. 
Thus, $(C_n)_{n \in \bZnn} \sim (S_n)_{n \in \bZnn}$ implies that $(C_n)_{n \in \bZnn}$ has principal point $z = \pi(1) = \lim_{s \to 1-} \pi(s)$. This yields the uniqueness of $z$. 
\end{proof}

\begin{prop}
\label{prop: path-connectedness propagates forward}
Let $\boldsymbol{K} = (K_t)_{t \geq 0}$ be a family of left-locally growing hulls. 
Let $\rho > 0$ such that $K_t \cup \bR$ is path-connected for all $t < \rho$. 
If $P_{\rho} \subset \{w \}$ for some $w \in \bH \cap K_{\rho}$, then $K_{\rho} \cup \bR$ is path-connected.
\end{prop}

\begin{proof}
By assumption, $\bigcup_{s < \rho} K_s \cup \bR$ is path-connected. 
Theorem~\ref{thm: decomposition added stuff to hull} shows that 
\begin{align}
\label{eq: helper prop: path-connectedness propagates forward}
K_{\rho} \cup \bR = \Big( \bigcup_{s < \rho} K_s \cup \bR \Big) \cup B_{\rho} \cup \{w\},
\end{align}
where $B_{\rho} = \big( \bigcap_{s < \rho} K_{\rho} \setminus K_s \big) \setminus ( \bdry K_{\rho} \cup \bR )$ is either empty, or open, path-connected, and simply connected.  
Moreover, $K_{\rho} \cup \bR$ is closed and connected. 
This implies by~\eqref{eq: helper prop: path-connectedness propagates forward} that 
\begin{align}
\label{eq: helper prop: path-connectedness propagates forward 2}
\bdry B_{\rho} \subset \Big( \bigcup_{s < \rho} K_s \cup \bR \Big) \cup \{w\}.
\end{align}
We first show that $\big( \bigcup_{s < \rho} K_s \cup \bR \big) \cup B_{\rho}$ is path-connected. 
If $B_{\rho} = \emptyset$, this holds by assumption. 
Otherwise, as $B_{\rho}$ is open and simply connected, invoking 
the Riemann mapping theorem, Fatou's Theorem, e.g.~\cite[Theorem II.5.3]{Garnett:Bounded_analytic_functions}, 
and~\cite[Exercise~2.5.5]{Pommerenke:Boundary_behaviour_of_conformal_maps}, we see that 
for almost every $z \in \bdry B_{\rho}$, the set $B_{\rho}$ is path-connected to $z$ in $\overline{B}_{\rho}$. 
By~\eqref{eq: helper prop: path-connectedness propagates forward 2} and by assumption, 
this proves that $\big( \bigcup_{s < \rho} K_s \cup \bR \big) \cup B_{\rho}$ is path-connected.
To finish, we show that $K_{\rho} \cup \bR$ is path-connected. 
Because $K_{\rho}$ is compact, there is $M > 0$ such that $K_{\rho} \subset B(0, M)$. 
Hence, $(K_{\rho} \cup \bR) \cap \overline{B(0, M)}$ is compact and connected,
and $(K_{\rho} \cup \bR) \cap \overline{B(0, M)} \setminus \{w\}$ is path-connected. 
Lemma~\ref{lemma: path-connected outside of a point is point} now implies that $K_{\rho} \cup \bR$ is path-connected. 
\end{proof}

\begin{lem}
\label{lemma: path-connected outside of a point is point}
Let $X \subset \bC$ be compact and connected. 
Assume there exists a point $x_0 \in X$ such that $X \setminus \{x_0\}$ is path-connected. 
Then, the whole set $X$ is path-connected.
\end{lem}

Lemma~\ref{lemma: path-connected outside of a point is point} follows from classical topology concerning cut-points.
A topological space $Y$ is said to be \emph{irreducible between points $a,b \in Y$} if it is connected and these two points cannot be joined 
by any closed connected subset different from the whole space (see~\cite[Page~190]{Kuratowski:Topology}).
By~\cite[Theorem~1, page~192]{Kuratowski:Topology}, 
every continuum joining two points $a$ and $b$ contains an irreducible continuum between them. 
Consider a metric continuum $C$ such that, 
with the exception of two points $a ,b \in C$, the set $C \setminus \{c\}$ is disconnected for every $c \in C$. 
Then, by~\cite[Theorem~1, page~179]{Kuratowski:Topology}, $C$ is an arc between $a$ and $b$.

\begin{proof}[Proof of Lemma~\ref{lemma: path-connected outside of a point is point}]
Let $z \in X \setminus \{x_0\}$. 
It is sufficient to show that $z$ is path-connected to $x_0$ in $X$. 
By~\cite[Theorem~1, page~192]{Kuratowski:Topology}, 
there exists an irreducible continuum $C \subset X$ between $z$ and $x_0$. 
We show that $C$ is an arc. 
By~\cite[Theorem~1, page~179]{Kuratowski:Topology}, 
this is the case if for all $c \in C \setminus \{z, x_0\}$, 
the set $C \setminus \{c\}$ is disconnected. 
By assumption, $X \setminus \{x_0\}$ is path-connected, and so is $C \setminus \{x_0\}$.  
If $C \setminus \{c\}$ would have just one connected component, then $C$ would not be irreducible. 
This finishes the proof.
\end{proof}

We are now ready to present the characterization of the existence of generating functions.

\begin{thm}
\label{thm: equivalence generated by a function}
Let $\boldsymbol{K} = (K_t)_{t \geq 0}$ be a family of locally growing hulls. 
Let $W \colon [0, \infty) \to \bR$ be the associated c\`adl\`ag driving function. 
Then, the following are equivalent.
\begin{enumerate}[label=\textnormal{(\arabic*):}, ref=\textnormal{(\arabic*)}]
\item\label{item: generated by a fct} 
The hulls $\boldsymbol{K}$ are generated by a function $\eta \colon [0, \infty) \to \overline{\bH}$.

\smallskip

\item\label{item: bounded boundary points} 
For all $t \geq 0$, the set $\smash{P_t = ( \bdry K_t \cap \bH ) \setminus \bigcup_{s < t} K_s}$ consists of at most one point.

\smallskip

\item\label{item: path-connected} 
For all $t \geq 0$, the set $K_t \cup \bR$ is path-connected.
\end{enumerate}
In that case, we have $P_t \subset \{\eta(t)\}$.
Moreover, if $P_t \neq \emptyset$, then the limit
\begin{align*}
\eta(t) = \lim_{y \to 0+} g_t^{-1}(W(t-) + \ii y)
\end{align*}
exists. 
It is the unique principal point of the grown end at time~$t$ and an accessible grown point at time~$t$.
\end{thm}

\begin{proof}
{\bf \ref{item: generated by a fct}~$\Rightarrow$~\ref{item: bounded boundary points}.} 
If 
$\boldsymbol{K}$ is generated by $\eta\colon [0, \infty) \to \overline{\bH}$, 
then we have $\eta[0, s] \subset K_s \cup \bR$ for all $s \geq 0$, and
\begin{align*}
P_t = (\bdry K_t \cap \bH) \setminus \bigcup_{s < t} K_s 
\; \subset \; ( \eta[0, t] \cap \bH ) \setminus \bigcup_{s < t} K_s 
\; \subset \; \eta[0, t] \setminus \bigcup_{s < t} \eta[0, s] \
\; \subset \; \{\eta(t)\}
\qquad \textnormal{for all } t \geq 0 .
\end{align*}

{\bf \ref{item: bounded boundary points}~$\Rightarrow$~\ref{item: generated by a fct}.}
If $P_t$ consists of at most one point, then by Corollary~\ref{cor: Kt are grown} and the growth of 
$\boldsymbol{K}$,
\begin{align}
\label{eq: helper proof gen by fct}
\begin{split}
\bdry K_t \cap \bH 
\; = \;\; & (\overline{\bH \setminus K_t}) \cap K_t \cap \bH 
\; = \; (\overline{\bH \setminus K_t}) \cap \bigg( \bigcup_{s \in [0, t]} \Big(K_s \setminus \bigcup_{r < s} K_r \Big) \bigg) \cap \bH
\\
\; \subset \; & \bigcup_{s \in [0, t]} \bigg(  (\overline{\bH \setminus K_s}) \cap \Big( K_s \setminus \bigcup_{r < s} K_r \Big) \cap \bH \bigg)
\; = \; \bigcup_{s \in [0, t]} P_s , \qquad  t \geq 0 .
\end{split}
\end{align}
Moreover, by assumption, for all $s \in [0, t]$, the set $P_s$ consists of at most one point. 
Hence, if $P_s \neq \emptyset$, then we may define $\{\eta(s)\} := P_s$; otherwise, we 
may choose an arbitrary $z \in \bdry K_s$ and set $\eta(s) := z$.  
Therefore, by~\eqref{eq: helper proof gen by fct} we have 
\begin{align*}
\bdry K_t \cap \bH 
\; \subset \; \bigcup_{s \in [0, t]} P_s
\; \subset \; \eta[0, t] 
\; \subset \; \bigcup_{s \in [0, t]} \bdry K_s \subset K_t \cup \bR
\qquad \textnormal{for all } t \geq 0 .
\end{align*}
Thus indeed, $\bH \setminus K_t$ is the unbounded connected component of $\bH \setminus \eta[0, t]$ for all $t \geq 0$.

{\bf \ref{item: bounded boundary points}~$\Rightarrow$~\ref{item: path-connected}.}
Set $\rho := \inf \{ t \geq 0 \;|\; K_t \cup \bR \textnormal{ is not path-connected} \}$. 
Assume towards a contradiction that $\rho < \infty$. 
Then by Proposition~\ref{prop: not path-connected at stopping time} and Lemma~\ref{lemma: local growth => continuous growth}, $K_{\rho} \cup \bR$ is not path-connected. 
However, by assumption~\ref{item: bounded boundary points} and Proposition~\ref{prop: path-connectedness propagates forward}, $K_{\rho} \cup \bR$ is path-connected. 
This contradiction proves that $\rho = \infty$. 

{\bf \ref{item: path-connected}~$\Rightarrow$~\ref{item: bounded boundary points}.} 
Fix $t \geq 0$. If $P_t = \emptyset$, then we are done. 
Thus, suppose that there exists $w \in P_t$.  
Let $\pi\colon [0, 1] \to K_t \cup \bR$ be a simple path connecting $\pi(0) = 0$ to $\pi(1) = w$. 
Set $\sigma := \inf \{ s \in [0, 1] \;|\; \pi(s) \in P_t \}$. 
Because $P_t$ is closed by Proposition~\ref{prop: new principal grown points}, we have 
$\pi(\sigma) \in P_t$. 
Then by Lemma~\ref{lemma: first entry point of a path means unique principal point} applied to $z = \pi(\sigma)$ and by Proposition~\ref{prop: new principal grown points}, we conclude that   
$P_t = \{ \pi(\sigma)\}$, as desired. 
This also proves that $w = \pi(\sigma)$. 

Lastly, consider $t \geq 0$ such that $\emptyset \neq P_t = \{\eta(t)\}$. 
Then, $\eta(t)$ is the unique principal grown point at time $t$ by Proposition~\ref{prop: new principal grown points}. 
From the equivalences~\eqref{eq: radial_limits_equiv}, we see that 
$\eta(t)$ being the unique principal grown point at time $t$ implies that $\eta(t)$ is an accessible point and together with Proposition~\ref{prop: characterise grown/growing end}, we find
\begin{align*}
\eta(t) = \lim_{y \to 0+} g_t^{-1}(W(t-) + \ii y).
\end{align*}
This concludes the proof. 
\end{proof}

\begin{rem}
The proof of Theorem~\ref{thm: equivalence generated by a function} only requires the hulls $(K_t)_{t \geq 0}$ to be left-locally growing and right-continuously growing. 
It is unclear whether these two properties imply right-local growth. 
\end{rem}

\smallskip
\subsection{Right-continuous generating functions --- Theorem~\ref{thm:coro_of_2}}
\label{subsec: right-continuous generating functions}

By Theorem~\ref{thm: equivalence generated by a function}, a generating function is not unique in general.
Indeed, if a family $\boldsymbol{K} = (K_t)_{\geq 0}$ of locally growing hulls generated by a function $\eta\colon [0, \infty) \to \overline{\bH}$ satisfies 
$P_t = ( \bdry K_t \cap \bH ) \setminus \bigcup_{s < t} K_s = \emptyset$ for some $t \geq 0$, then $\eta(t)$ can be any point in $\bdry K_t$ (as we saw in the proof of Theorem~\ref{thm: equivalence generated by a function}). 
Nonetheless, if all grown ends have unique principal points, then there is a canonical choice of generating function. 

\begin{thm}
\label{thm: existence right-continuous generating function}
Let $\boldsymbol{K} = (K_t)_{t \geq 0}$ be a family of locally growing hulls. 
Let $W \colon [0, \infty) \to \bR$ be the associated c\`adl\`ag driving function. 
Assume that the following radial limits exist:
\begin{align}
\label{eq: expressing grown principal point}
\eta(t) := \lim_{y \to 0+} g_t^{-1} ( W(t-) + \ii y ) \qquad \textnormal{for all $t \geq 0$.}
\end{align}
Then, the hulls $\boldsymbol{K}$ are generated by $\eta \colon [0, \infty) \to \overline{\bH}$. 
Moreover, for all $t \geq 0$, 
\begin{itemize}
\item $P_t = ( \bdry K_t \cap \bH ) \setminus \bigcup_{s < t} K_s \subset \{\eta(t)\}$; 

\smallskip

\item $\eta(t)$ is the unique principal point of the grown end at time $t$; 

\smallskip

\item $\eta(t)$ is an accessible grown point at time $t$.
\end{itemize}
\end{thm}

\begin{proof}
By the assumption~\eqref{eq: expressing grown principal point}, 
from the equivalences~\eqref{eq: radial_limits_equiv} and Proposition~\ref{prop: characterise grown/growing end}, 
we see that the grown end at time $t$ has radial limit $\eta(t)$ for all $t \geq 0$. 
In particular, $\eta(t)$ is accessible in $\bH \setminus K_t$ and the unique principal grown point at time $t$.
Therefore, by Proposition~\ref{prop: new principal grown points} 
the set $P_t = (\bdry K_t \cap \bH) \setminus \bigcup_{s < t} K_s$ consists of at most one point. 
Thus, $\boldsymbol{K}$ is generated by $\eta\colon [0, \infty) \to \bR$ by Theorem~\ref{thm: equivalence generated by a function}. 
\end{proof}

\begin{rem} \label{rem:subtleties} \
\begin{itemize}[leftmargin=*]
\item 
This result is very subtle. For instance, in contrast to Theorem~\ref{thm: existence right-continuous generating function}, the existence of the radial limit 
\begin{align}\label{eq:right_radial_limit}
\lim_{y \to 0+} g_t^{-1}(W(t+) + \ii y)
\end{align}
for all $t$ does not imply that the corresponding hulls are generated by a function.
Indeed, it is possible to grow the spiral from Example~\ref{subsec: logarithmic spiral} and jump to the real line at the exact time $t_0 = 1$ when the spiral closes. 
Then, the limit~\eqref{eq:right_radial_limit} at $t=t_0$ exists, since the real line is locally connected, 
but the corresponding hulls are not generated by a function, because that would require $\{\eta(t_0)\} = \bdry B(2\ii, 1)$. 

\smallskip

\item
By the double-comb in Example~\ref{eg: double comb}, the assumption~\eqref{eq: expressing grown principal point} 
in Theorem~\ref{thm: existence right-continuous generating function} is stronger than the property of being generated by a function. 
Namely, the double-comb is generated by a function, but no radial limit exists for the grown end at the critical time $t = 3$. 

\smallskip

\item 
The assumption~\eqref{eq: expressing grown principal point} does not imply that $\bH \setminus K_t$ is locally connected. 
We have seen a counterexample in the comb space in Example~\ref{subsec: comb space}. 
These hulls are generated by a function that satisfies~\eqref{eq: expressing grown principal point} \textnormal{(}because every prime end has an accessible point and by~\eqref{eq: radial_limits_equiv}\textnormal{);} 
yet, $\bH \setminus K_3$ is not locally connected. 
Analytically, this means that the radial limit~\eqref{eq: expressing grown principal point} exists for all $t$, but the convergence is not uniform in $t$. 
\end{itemize}
\end{rem}

However, if the radial limit~\eqref{eq: expressing grown principal point} exists, then all prime ends of $\bH \setminus K_t$ admit unique principal points. 

\begin{prop}
\label{prop: If generated by a function, all prime ends are accessible}
Let $(K_t)_{t \geq 0}$ be a family of locally growing hulls and let $W\colon [0, \infty) \to \bR$ be its driving function. 
Assume that the radial limits~\eqref{eq: expressing grown principal point} exist for all time.
Then, for all $t \geq 0$, every prime end of $\bH \setminus K_t$ has a unique principal point.
\end{prop}

\begin{proof}
Fix $t \geq 0$. 
Suppose that there exists a prime end of $\bH \setminus K_t$ with two principal points $z$ and $w$. 
Let $C = (C_n)_{n \in \bZnn}$ and $S = (S_n)_{n \in \bZnn}$ be null-chains representing $z$ and $w$ in $\bH \setminus K_t$, respectively. 
Set 
\begin{align*}
\sigma := \inf\{ s \geq 0 \;|\; z, w \in K_s \cup \bR \} \leq t.
\end{align*}
Then by Lemmas~\ref{lemma: local growth => continuous growth}~\&~\ref{lem: hcap and growing}, 
we have $z, w \in \bigcap_{\varepsilon > 0} K_{\sigma + \varepsilon} \cup \bR = K_\sigma \cup \bR $.
Moreover, $z, w \in \bdry K_t$, so $z, w \in \bdry K_t \cap (K_\sigma \cup \bR  ) = \bdry K_t \cap \bdry K_\sigma$.
Hence, by Lemma~\ref{lemma: prime end extension lemma} proven below, for every $s \in [\sigma, t]$ we can construct null-chains $C^s = (C^s_n)_{n \in \bZnn}$ and $S^s = (S^s_n)_{n \in \bZnn}$ out of $C$ and $S$ such that 
\begin{itemize}
\item 
$C^s$ and $S^s$ are null-chains in $\bH \setminus K_s$;

\smallskip

\item 
$C^s$ has principal point $z$; and 

\smallskip

\item 
$S^s$ has principal point $w$.
\end{itemize}
Define $\tau := \inf \big\{ s \in [\sigma, t] \;|\; C^s \sim S^s\textnormal{ in } \bH \setminus K_s \big\} \in [\sigma, t]$.

\textbf{Step 1:} 
We show that $C^\tau \sim S^\tau$ are equivalent null-chains.

Firstly, assume that 
$C^\tau$ does not represent the growing end at time $\tau$. 
Let $\Pi_{\tau}$ be the set of principal growing points at time $\tau$. 
Then, $\Pi_{\tau}$ is compact by~\cite[Theorem~7.1]{Epstein:Prime_ends}. 
Because $z \notin \Pi_{\tau}$, this implies that $\varepsilon := \dist(z, \Pi_{\tau}) > 0$. 
Let $(S_{\delta_n}^{\textnormal{out}})_{n \in \bZnn}$ represent the growing end at time $\tau$. 
Because $\diam (S_{\delta_n}^{\textnormal{out}}) \to 0$ as $n \to \infty$, there exists $N_1 \in \bZnn$ such that  
\begin{align} 
\label{eq: helper prop unique principal point -1}
\dist (\Pi_{\tau}, S_{\delta_n}^{\textnormal{out}}) < \varepsilon/4 \qquad \textnormal{for all } n \geq N_1 ,
\end{align} 
and similarly, there exists $N_2 \in \bZnn$ such that 
\begin{align}
\label{eq: helper prop unique principal point -2}
C_n^{\tau} \subset B(z, \varepsilon/4) \qquad \textnormal{for all } n \geq N_2.
\end{align}
Moreover, by the right local growth at time $\tau$, there exist $\rho_1 > 0$ 
and a crosscut $S_{\rho_1}^{\textnormal{out}} \subset \bH \setminus K_{\tau}$ with $\diam(S_{\rho_1}^{\textnormal{out}}) < \varepsilon/4$
such that $S_{\rho_1}^{\textnormal{out}}$ separates $K_{\tau + \rho_1} \setminus K_{\tau}$ from $\infty$ in $\bH \setminus K_{\tau}$. 
Thus, by~(\ref{eq: helper prop unique principal point -1},~\ref{eq: helper prop unique principal point -2}) 
the null-chain $(C_n^{\tau})_{n \geq N}$, with $N = \max(N_1, N_2)$, is a null-chain in $\bH \setminus K_{\tau + \rho_1}$. 
By the choice of $\tau$, we have 
\begin{align*}
(C_n^{\tau})_{n \geq N} 
\; \sim \; (C_n^{\tau + \rho_1})_{n \geq N} 
\; \sim \; (S_n^{\tau + \rho_1})_{n \geq N}  
\; \sim \; (S_n^{\tau})_{n \geq N}.
\end{align*}
By symmetry, if we assume that 
$S^\tau$ does not represent the growing end at time $\tau$, then $S^\tau \sim C^\tau$. 
Lastly, if both 
$S^\tau$ and $C^\tau$ represent the growing end at time $\tau$, then they are trivially equivalent. 

\textbf{Step 2:} We show that both 
$S^\tau$ and $C^\tau$ represent the grown end at time $\tau$. 

Assume towards a contradiction that 
$C^\tau$ does not represent the grown end at time $\tau$.
Let $(S_{\delta_n}^{\textnormal{in}})_{n \in \bZnn}$ represent the grown end at time $\tau$. 
Then, $\eta(\tau)$ is the unique principal point of $(S_{\delta_n}^{\textnormal{in}})_{n \in \bZnn}$ by assumption and by Theorem~\ref{thm: existence right-continuous generating function}.  
In particular,  we have $\varepsilon := |z - \eta(\tau)| > 0$. 
Because $\diam (S_{\delta_n}^{\textnormal{in}}) \to 0$ as $n \to \infty$, there exists $N_1 \in \bZnn$ such that 
\begin{align}
\label{eq: helper prop unique principal point}
S_{\delta_n}^{\textnormal{in}} \subset B ( \eta(\tau), \varepsilon/4 ) \qquad \textnormal{for all } n \geq N_1 ,
\end{align} 
and similarly, there exists $N_2 \in \bZnn$ such that 
\begin{align}
\label{eq: helper prop unique principal point 2}
C_n^{\tau} \subset B(z, \varepsilon/4) \qquad \textnormal{for all } n \geq N_2.
\end{align}
Moreover, by the left local growth at time $\tau$, there exists a $\rho_1 \in (0, \tau)$ 
and a crosscut $S_{\tau - \rho_1}^{\textnormal{in}} \subset \bH \setminus K_{\tau - \rho_1}$ with $\diam(S_{\tau - \rho_1}^{\textnormal{in}}) < \varepsilon/4$
such that $S_{\tau - \rho_1}^{\textnormal{in}}$ separates $K_{\tau} \setminus K_{\tau - \rho_1}$ from $\infty$ in $\bH \setminus K_{\tau - \rho_1}$. 
Therefore, by~(\ref{eq: helper prop unique principal point},~\ref{eq: helper prop unique principal point 2}), 
the null-chain $(C_n^{\tau})_{n \geq N}$ is a null-chain in $\bH \setminus K_{\tau - \rho_1}$. 
Consequently, by Step 1, 
\begin{align*}
(C_n^{\tau - \rho_1})_{n \geq N} 
\; \sim \; 
(C_n^{\tau})_{n \geq N}  
\; \sim \; 
(S_n^{\tau})_{n \geq N} 
\; \sim \; 
(S_n^{\tau - \rho_1})_{n \geq N}, 
\end{align*}
which contradicts the definition of $\tau$. 
Therefore, we deduce that 
$C^\tau$ indeed represents the grown end at time $\tau$. 
By a symmetric argument, 
$S^\tau$ represents the grown end at time $\tau$ as well. 

\textbf{Step 3:} Finally, we prove that $z = w$. 

By Step 2, the assumption, and Theorem~\ref{thm: existence right-continuous generating function}, 
both 
$C^\tau$ and $S^\tau$ have $\eta(\tau)$ as their unique principal point. Therefore, $z = \eta(\tau) = w$. 
This concludes the proof. 
\end{proof}

\begin{lem}
\label{lemma: prime end extension lemma}
Let $K \subset L$ be two half-plane hulls. 
Let $(C_n)_{n \in \bZnn}$ be a null-chain in $\bH \setminus L$ with principal point $z \in \bdry K \cap \bdry L$. 
Then, we can extend $(C_n)_{n \in \bZnn}$ into a null-chain $(C_n')_{n \in \bZnn}$ in $\bH \setminus K$ such that
\begin{enumerate}[label=\textnormal{(\arabic*):}, ref=\textnormal{(\arabic*)}]
\item $(C_n')_{n \in \bZnn}$ has $z$ as its principal point; and

\item for every $\varepsilon > 0$ and $n \in \bZnn$, we have $C_n \subset \overline{B(z, \varepsilon)}$ 
if and only if $C_n' \subset \overline{B(z, \varepsilon)}$.
\end{enumerate} 
\end{lem}

\begin{proof}
Note that $\bH \setminus L \subset \bH \setminus K$. If both endpoints of $C_n$ lie in $\bdry K \cap \bdry L$, 
then $C_{n}$ is already a crosscut in $\bH \setminus K$. 
In this case, we set $C_{n}' := C_{n}$. 
Otherwise, we may parameterize $C_n \colon [0,1] \to \bC$, 
and assume that the right endpoint lies not in $\bdry K \cap \bdry L$, i.e., $C_n(1) \in \bdry L \setminus \bdry K$ and $C_{n}(0, 1] \cap \bdry K = \emptyset$. 
However, since $z \in \bdry K$ by assumption, we also have $\bdry B(z, |z - C_{n}(1)|) \cap \bdry K \neq \emptyset$. 
Thus, we can extend $C_{n}$ along the circle $\bdry B(z, |z - C_{n}(1)|)$ from $C_{n}(1)$ to the first time when this circle 
intersects $\bdry K$. 
If necessary, we apply the same construction to $C_{n}(0)$ --- to obtain a crosscut in $\bH \setminus K$, which we will call $C_{n}'$. 
Doing this for all $n$ yields a sequence $(C_n')_{n \in \bZnn}$ of crosscuts in $\bH \setminus K$, as desired.  
It follows by construction that $C_n \subset \overline{B(z, \varepsilon)}$ if and only if $C_n' \subset \overline{B(z, \varepsilon)}$. 
Hence, $(C_n')_{n \in \bZnn}$ also has $z$ as its principal point. 
\end{proof}

\begin{figure}[ht]
\centering
\includegraphics[width=0.25\textwidth]{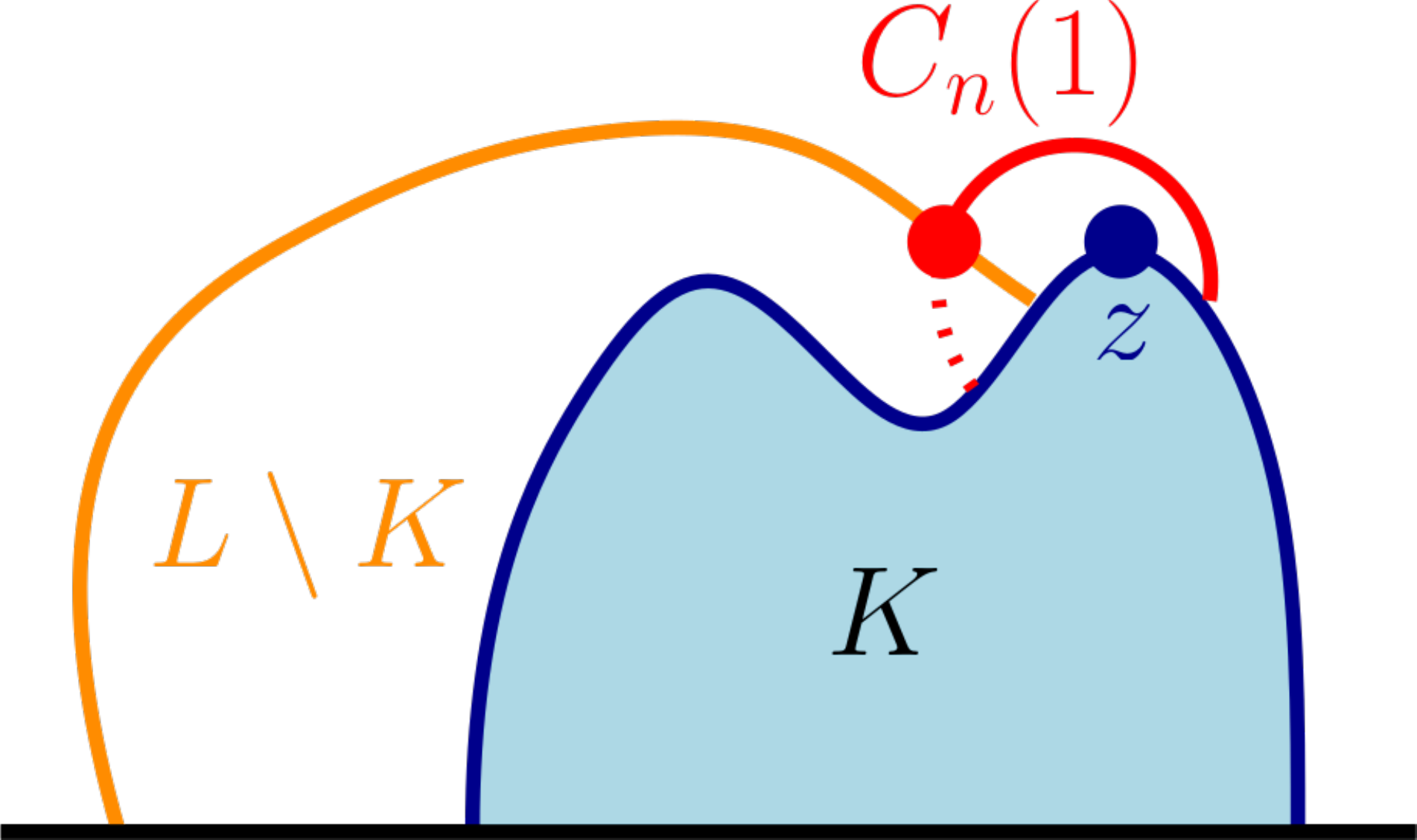}
\caption{Illustration for the proof of Lemma~\ref{lemma: prime end extension lemma}.}
\end{figure}

The following result generalizes~\cite[Theorem~5.22]{Beliaev:Conformal_maps_and_geometry}, which proves that a locally connected Loewner chain with a continuous driving function is generated by a right-continuous function.

\begin{thm}
\label{thm: Minimal regularity of generating functions}
Let $\boldsymbol{K} = (K_t)_{t \geq 0}$ be a family of locally growing hulls. 
Let $W \colon [0, \infty) \to \bR$ be the associated c\`adl\`ag driving function. 
Assume that the radial limits~\eqref{eq: expressing grown principal point} exist for all time.
Then, the limit
\begin{align*}
\eta(t+) = \lim_{s \to t+} \eta(s) = \lim_{y \to 0+} g_t^{-1}(W(t) + \ii y)  \qquad \textnormal{exists for all $t \geq 0$.}
\end{align*}
It is the unique right limit of $\eta$ at $t$, which is the unique principal point of the growing end at time $t$.
\end{thm}

\begin{proof} 
Fix $t \geq 0$.
Let $(s_n)_{n \in \bZnn}$ be a decreasing sequence in $(t, t + 1)$ with limit $t$. 
Then, $(\eta(s_n))_{n \in \bZnn}$ is a sequence in the compact set $K_{t + 1}$. 
Thus, there exists a convergent subsequence $(\eta(s_{n_k}))_{k \in \bZnn}$ with limit $z \in K_{t + 1}$. 
By Lemma~\ref{lemma: grown points and principal growing points}, $z$ is a principal point of the growing end at time $t$. 
Moreover, by Proposition~\ref{prop: If generated by a function, all prime ends are accessible} 
the principal point of the growing end at time $t$ is unique. 
Therefore, $\eta$ has a right limit at time $t$, which is the unique principal point of the growing end at time $t$. 
By Proposition~\ref{prop: characterise grown/growing end},
we can write this principal point of the growing end at time $t$ as the radial limit 
\begin{align*}
	\hspace*{180pt}
	\lim_{y \to 0+} g_t^{-1}(W(t) + \ii y). 
	\hspace*{227pt}
	\qedhere
\end{align*}
\end{proof}

\begin{rem}
The converse of Theorem~\ref{thm: Minimal regularity of generating functions} is not true in general. 
A generating function can have right limits for all times and yet the radial limit~\eqref{eq: expressing grown principal point} need not exist.
For instance, in the double-comb from Example~\ref{eg: double comb} at the critical time $t = 3$ the radial limit~\eqref{eq: expressing grown principal point} does not exist. 
Nonetheless, by choosing a convenient jump at time $t = 3$, we can extend the growth of the hulls $\boldsymbol{K}$ 
somewhere away from the set $\ii [0, 1]$ is such a way that the sought limit $\underset{s \to 3-}{\lim} \, \eta(s)$ does exist. 
\end{rem}

\begin{rem}
By Examples~\ref{subsec: comb space}~and~\ref{eg: Belyaev}, a generating function of a Loewner chain need not be 
left-continuous in general. 
Nonetheless, by the compactness of the corresponding hulls, 
such a generating function must have sub-sequential left limits for all times $t > 0$.
All of these left limits are grown points.
\end{rem}

\begin{cor}
\label{cor: left limts are grown points} 
Let $\boldsymbol{K} = (K_t)_{t \geq 0}$ be a family of locally growing hulls generated by $\eta\colon [0, \infty) \to \overline{\bH}$. 
Fix~$t > 0$, and let $(s_n)_{n \in \bZnn}$ be a strictly increasing sequence with limit $t$ such that 
\begin{align*}
z := \lim_{n \to \infty} \eta(s_n) \qquad \textnormal{exists.}
\end{align*}
Then, $z$ is a grown point at time $t$.
\end{cor}

\begin{proof}
This is immediate by Theorem~\ref{thm: equivalence generated by a function} and Lemma~\ref{lemma: grown points and principal growing points}. 
\end{proof}

By Example~\ref{subsec: comb space}, it is possible that a generating function has uncountably many left limits at a fixed time $t$.
By Example~\ref{eg: Belyaev}, this can also happen while the Loewner chain is locally connected and has a continuous driving function. 
Hence, whether or not the driving function is continuous at time $t$ does not determine whether or not a generating function, if it exists, is left-continuous at time $t$. 
In fact, the left-continuity of the generating function is an additional regularity property; see Section~\ref{subsec: left-continuous generating function}.

Nonetheless, one would intuitively expect that the continuity of the driving function affects the continuity of the corresponding generating function and vice versa. 
This turns out to be true, if we regard ``continuity'' to mean ``$\eta(t)$ equals the unique right limit at time $t$.''
This is possible: 
By Theorem~\ref{thm: existence right-continuous generating function}, the value $\eta(t)$ is canonically defined, even if $\eta$ is not left-continuous at 
$t$\footnote{For instance, in the comb space in Example~\ref{subsec: comb space}, $\eta(3+) = \ii$ and it is the only accessible grown point at time $t = 3$.}.

\begin{prop}
\label{prop: characterise continuity of the driving function, gen by a function}
Let $\boldsymbol{K} = (K_t)_{t \geq 0}$ be a family of locally growing hulls. 
Let $W \colon [0, \infty) \to \bR$ be the associated c\`adl\`ag driving function. 
Assume that the radial limits~\eqref{eq: expressing grown principal point} exist for all time.
Then, for all $t \geq 0$, the following are equivalent.
\begin{enumerate}[label=\textnormal{(\arabic*):}, ref=\textnormal{(\arabic*)}]
\item The driving function $W$ is continuous at time $t$.

\smallskip

\item The function $\eta$ does not self-cross at time $t$ and $\eta(t) = \eta(t+) := \displaystyle \lim_{s \to t+} \eta(s)$.
\end{enumerate}
\end{prop}

\begin{proof}
Fix $t \geq 0$. 
By Corollary~\ref{cor: characterise continuity of the driving function}, the driving function $W$ is continuous at $t$ if and only if the grown and growing end at time $t$ coincide. 
By Theorem~\ref{thm: existence right-continuous generating function}, the grown end at time $t$ is characterized by its unique accessible (grown) point, which is $\eta(t)$. 
Likewise, by Theorem~\ref{thm: Minimal regularity of generating functions} the unique accessible growing point at time $t$ is $\eta(t+)$. 
Because prime ends distinguish from which side the hulls are being approximated, the grown and growing ends are different if and only if 
either $\eta$ self-crosses at time $t$, i.e.~one is the right-sided prime end and the other is the left-sided prime end, or $\eta(t) \neq \eta(t+)$. 
\end{proof}

\smallskip
\subsection{Left-continuous generating functions and local connectedness --- Theorems~\ref{thm:coro_of_3}~\&~\ref{thm:coro_of_4}}
\label{subsec: left-continuous generating function}

A generating function being left-continuous implies that all right limits exist for all $t \geq 0$. 
Consequently, left-continuity of the generating function is a stronger property than its right-continuity.

\begin{prop}
\label{prop: left-cont implies right-cont}
Let $\boldsymbol{K} = (K_t)_{t \geq 0}$ be a family of locally growing hulls generated by a left-continuous function $\eta\colon [0, \infty) \to \overline{\bH}$. 
Let $W \colon [0, \infty) \to \bR$ be the associated c\`adl\`ag driving function. 
Then, the radial limits~\eqref{eq: expressing grown principal point} exist:
\begin{align}
\label{eq: left limits for left-continuous}
\eta(t) = \lim_{y \to 0+} g_t^{-1} ( W(t-) + \ii y ) \qquad \textnormal{for all $t \geq 0$.}
\end{align}
Moreover, $\eta\colon [0, \infty) \to \overline{\bH}$ has unique right limits for all $t \geq 0$:
\begin{align}
\label{eq: right limits for left-continuous}
\eta(t+) = \lim_{s \to t+} \eta(s) = \lim_{y \to 0+} g_t^{-1}(W(t) + \ii y) \qquad \textnormal{for all $t \geq 0$.}
\end{align}
Thus, $\eta \colon [0, \infty) \to \overline{\bH}$ is c\`agl\`ad, i.e., it is left-continuous function with unique right limits.
\end{prop}

\begin{proof}
By Proposition~\ref{prop: characterise grown/growing end} and by the equivalences~\eqref{eq: radial_limits_equiv}, 
in order to prove~\eqref{eq: left limits for left-continuous}, it is sufficient to show that for all $t > 0$, 
the grown end at time $t$ has $\eta(t)$ as its unique principal point. Fix $t > 0$. 
By the left-continuity of $\eta$, for every $\varepsilon > 0$ there exists $\delta \in (0, t)$ such that 
$\eta[t - \delta, t] \subset B(\eta(z), \varepsilon)$.
By the proof of Theorem~\ref{thm: equivalence generated by a function}, this implies that
\begin{align}
\label{eq: helper right limits for left-continuous}
\bdry K_t \cap ( K_t \setminus K_{t-\delta} ) \subset \eta[t - \delta, t] \subset B(z, \varepsilon). 
\end{align}
By the left-local growth, for every $\varepsilon > 0$ there exists  
$\delta > 0$ and a crosscut $S_{\delta}^{\textnormal{in}} \subset \bH \setminus K_{t-\delta}$  
with $\diam(S_{\delta}^{\textnormal{in}}) < \varepsilon$ 
such that $S_{\delta}^{\textnormal{in}}$ separates $K_{t} \setminus K_{t-\delta}$ from $\infty$ in $\bH \setminus K_{t-\delta}$ and in $\bH \setminus K_{t}$. 
Since $K_{t-\delta} \subsetneq K_t$ (by the strict growth of $\boldsymbol{K}$), 
as $\varepsilon > 0$ is arbitrary,~\eqref{eq: helper right limits for left-continuous} implies that every null-chain representing the grown end at time $t$ has principal point $\eta(t)$.  
This proves~\eqref{eq: left limits for left-continuous},
and~\eqref{eq: right limits for left-continuous} follows from Theorem~\ref{thm: Minimal regularity of generating functions}. 
\end{proof}

Importantly, the left-continuity of the generating function implies local connectedness, as detailed below. 
This property is of great interest for boundary behavior of growing hulls (cf.~\cite{Chen-Rohde:SLE_driven_by_symmetric_stable_processes, Peltola-Schreuder:Loewner_traces_driven_by_Levy_processes}).
Note  that local connectedness of the image of the generating function $\eta$ is a priori a different property than local connectedness of the boundaries of the associated domains $\bH \setminus K_t$: 
Example~\ref{eg: Belyaev} gives an example of a right-continuous function $\eta$ generating a Loewner chain, 
whose driving function is continuous, and
for which $\eta$ itself is not locally connected but the boundaries of the associated domains $\bH \setminus K_t$ are still locally connected.
See Theorem~\ref{thm: left continuity implies locally connectedness} concerning the case where $\eta$ has both left and right limits.

\begin{thm} 
\label{thm: locally path-connected} 
Let $\boldsymbol{K} = (K_t)_{t \geq 0}$ be a family of locally growing hulls generated by a left-continuous function $\eta\colon [0, \infty) \to \overline{\bH}$. 
Then, for each $t \geq 0$, the set $\eta[0, t] \cup \bR$ is connected and locally path-connected. 
\end{thm}

Note that any connected and locally path-connected set is path-connected.
In fact, this theorem can be seen as a special case of a more general topological result. 
Its proof only requires two ingredients; that $\eta[0, t] \cup \bR$ is compact and connected, 
and that $\eta$ has both left and right limits --- see Theorem~\ref{thm: left continuity implies locally connectedness}.

\begin{proof}
Fix $t \geq 0$. It suffices to prove the asserted properties for $L_R := \eta[0, t] \cup [-R,R]$ for large $R > 0$. 
First, $L_R$ is compact and connected whenever $\eta[0, t] \subset B(0,R)$. 
Second, $\eta \colon [0, \infty) \to \overline{\bH}$ is c\`agl\`ad by Proposition~\ref{prop: left-cont implies right-cont}.
Hence, Theorem~\ref{thm: left continuity implies locally connectedness} below implies that $L_R$ is also locally connected.
Third, \cite[Theorem~6.7.2]{Sagan:Space_filling_curves} now implies that $L_R$ is connected and locally path-connected, as claimed. 
\end{proof}

Theorem~\ref{thm: locally path-connected} is a generalization of~\cite[Corollary~2.14]{Peltola-Schreuder:Loewner_traces_driven_by_Levy_processes}, where we considered the hulls themselves. 
The next result generalizes~\cite[Proposition~1.2]{Peltola-Schreuder:Loewner_traces_driven_by_Levy_processes}, where we proved the local connectedness of the frontiers of the 
hulls\footnote{Theorem~\ref{thm: left continuity implies locally connectedness} was proven for the special case of a Loewner chain driven by a symmetric stable pure jump process in~\cite[Proposition~7.2]{Chen-Rohde:SLE_driven_by_symmetric_stable_processes}. 
In that case, one could use the property that the hulls have empty interiors~\cite[Theorem~1.3(i)]{Guan-Winkel:SLE_and_aSLE_driven_by_Levy_processes} in combination with a result from complex analysis due to Warschawski from the 1950s.}.  
It could be regarded as an extension of part of the Hahn-Mazurkiewicz theorem~\cite[Theorem~2, page~256]{Kuratowski:Topology}. 
We can follow the same proof as in~\cite[Proposition~1.2]{Peltola-Schreuder:Loewner_traces_driven_by_Levy_processes}, which we include in Appendix~\ref{app: left continuity implies locally connectedness} 
for convenience, showcasing the relevance of the right and left limits for $\eta$.

\begin{restatable}{thm}{LocConnThm} 
\label{thm: left continuity implies locally connectedness}
Let $\boldsymbol{K} = (K_t)_{t \geq 0}$ be a family of locally growing hulls generated by a c\`agl\`ad function $\eta\colon [0, \infty) \to \overline{\bH}$. 
Then, for each $t \geq 0$, the set $\eta[0, t] \cup \bR$ is locally connected.
\end{restatable}

\begin{rem}
The comb space in Example~\ref{subsec: comb space} is not locally connected. 
Thus, by Theorem~\ref{thm: left continuity implies locally connectedness} and Proposition~\ref{prop: left-cont implies right-cont}, the associated generating function $\eta\colon [0, 4] \to \overline{\bH}$ 
cannot be left-continuous. Indeed, $\eta$ is not left-continuous at $t = 3$ which is the first time, when $\eta[0, t] \cup \bR$ is not locally connected.
\end{rem}

A natural question is whether these properties are in fact equivalent: Does the local connectedness of the graph of a generating function imply that this generating function is left-continuous? This is not true in general. 
By Carath\'eodory's theorem, e.g.,~\cite[Theorem~2.1 and Corollary~2.17]{Pommerenke:Boundary_behaviour_of_conformal_maps}, local connectedness implies that all grown points on the boundary are unique.  
The problem arises when alternative left limits are swallowed by the hulls, see Example~\ref{eg: Belyaev}. 
If we exclude this phenomenon, i.e.,~the hulls $(K_t)_{t \geq 0}$ have empty interior, then local connectedness and left-continuity do become equivalent.

\begin{prop}
\label{prop: empty interior locally connected equiv cadlag generated}
Let $\boldsymbol{K} = (K_t)_{t \geq 0}$ be a family of locally growing hulls.
Assume that for all $t \geq 0$, the set $K_t$ has an empty interior and its boundary $\bdry K_t$ is locally connected. 
Then, the hulls $\boldsymbol{K}$ are generated by a c\`agl\`ad function.
\end{prop}

\begin{proof}
By Carath\'eodory's theorem (e.g.~\cite[Theorem 2.1]{Pommerenke:Boundary_behaviour_of_conformal_maps}), the radial limit~\eqref{eq: expressing grown principal point} exists for all time. 
So, by Theorem~\ref{thm: existence right-continuous generating function} the hulls 
$\boldsymbol{K}$ are generated by a function $\eta\colon [0, \infty) \to \overline{\bH}$, and by Theorem~\ref{thm: Minimal regularity of generating functions}, $\eta$ has unique right limits. 
It thus remains to show that $\eta$ is left-continuous. 
Let $t > 0$. In order to prove the left-continuity of $\eta$ at $t$, it is sufficient to show that for every strictly increasing sequence $(s_n)_{n \in \bZnn}$ converging to $t$, there is a subsequence such that
\begin{align*}
\lim_{k \to \infty} \eta(s_{n_k}) = \eta(t).
\end{align*}
By the growth of the hulls, we have $\eta(s_n) \in K_{s_n} \subset K_t$ for all $n \in \bZnn$. 
Because $K_t$ is compact, there exists a subsequence $(s_{n_k})_{k \in \bZnn}$ such that $( \eta(s_{n_k}) )_{k \in \bZnn}$ has a limit $z \in K_t$. 
By Corollary~\ref{cor: left limts are grown points}, this point $z$ is a grown point at time $t$. 
Then, because $z \in \bdry K_t$ and $\bdry K_t$ is locally connected by assumption, 
Carath\'eodory's theorem shows that the only grown point at time $t$ in $\bdry K_t$ is the unique limit
\begin{align*}
	\hspace*{150pt}
	z = \eta(t) = \lim_{y \to 0+} g_t^{-1}(W(t-) + i y). 
	\hspace*{150pt}
	\qedhere
\end{align*}
\end{proof}


\appendix

\bigskip{}
\section{Proof of Theorem~\ref{thm: left continuity implies locally connectedness}}
\label{app: left continuity implies locally connectedness}
We consider crossings of annuli 
\begin{align*}
\overline{\mathbb{A}}_0 = \overline{\mathbb{A}}_0(z_0,r_0,R_0) = \{ w \in \bC \; | \; r_0 \leq |w - z_0| \leq R_0 \} . 
\end{align*}
The key to guarantee local connectedness is that the function $\eta$ possesses both left and right limits, which implies that annulus crossings are controlled.

\begin{proof}[{Proof of Theorem~\ref{thm: left continuity implies locally connectedness}}]
Let $t \geq 0$.
We prove that $\eta[0, t] \cup \bR$ is locally connected by contradiction. 

Suppose that for some $t \geq 0$, the set $A_t := \eta[0, t] \cup \bR$ is not locally connected. 
Then, there exists a point $z \in A_t$, radius $r > 0$, and points $z_n \to z$ as $n \to \infty$ such that all of $z_n$ and $z$ lie in different connected components of $U(z,r) := A_t \cap B(z,r)$. 
We may furthermore assume that all of the points $z_n$ are inside $B(z,\frac{r}{10})$. 
Since $A_t = \eta[0, t] \cup \bR$ is connected,  
the points $z_n \in A_t \cap B(z,\frac{r}{10})$ must all be connected together in $A_t$ outside of $U(z,r)$. 
In particular, the set $A_t$ makes infinitely many distinct crossings across the annulus $\smash{\overline{\mathbb{A}}(z,\tfrac{r}{2},r)}$. 
We shall prove that this is impossible by the left-continuity of $\eta$ and the right-continuity of $\gamma\colon [0, \infty) \to \bR$ defined by $\gamma(t) = \eta(t+)$.

To this end, fix a point $z_0 \in \overline{\bH}$ and two radii $0 < r_0 < R_0$. 
Consider the time 
\begin{align*}
T_0 = T(z_0,r_0,R_0) 
:= \; & \inf \{ s \geq 0 \; | \; \textnormal{$A_t$ makes infinitely many distinct crossings across $\smash{\overline{\mathbb{A}}_0}$} \} \\
= \; & \inf \{ s \geq 0 \; | \; \textnormal{$A_t \cap \smash{\overline{\mathbb{A}}_0}$ has infinitely many connected components $(S_j)_{j \in J}$ such that} \\
& \qquad\qquad\quad
\textnormal{for all $j \in J$, we have $S_j \cap \bdry B(z_0,r_0) \neq \emptyset$ and $S_j \cap \bdry B(z_0,R_0) \neq \emptyset$} \} .
\end{align*}
Suppose $T_0 < \infty$. Then, the following contradictory properties hold.

\begin{itemize}[leftmargin=*]
\item First, the set $A_{T_0}$ cannot make infinitely many distinct crossings across $\smash{\overline{\mathbb{A}}_0}$.
Indeed, if this is the case, then for any strictly smaller time $s < T_0$, the set 
\begin{align*}
A_{T_0} \setminus A_s = ( \eta[0, T_0] \cup \bR ) \setminus ( \eta[0, s] \cup \bR )
\subset \eta(s, T_0]
\end{align*}
would make infinitely many crossings across the annulus $\smash{\overline{\mathbb{A}}_0}$.
This violates the left-continuity of $\eta$, since there exists $\delta = \delta(r_0, R_0) > 0$ such that, taking $s = T_0 - \delta$, we arrive at a contradiction:
\begin{align*}
\sup_{ u,v \in ( T_0 - \delta , T_0 ] } | \eta(u) - \eta(v) | < \tfrac{1}{2}(R_0 - r_0) .
\end{align*}
Thus, the set $A_{T_0}$ must make finitely many distinct crossings across $\smash{\overline{\mathbb{A}}_0}$.

\item Second, consider a sequence $t_n \to T_0+$ as $n \to \infty$ such that each set $A_{t_n}$ makes infinitely many distinct crossings across $\smash{\overline{\mathbb{A}}_0}$.
If the set $A_{T_0}$ only makes finitely many distinct crossings across $\smash{\overline{\mathbb{A}}_0}$, then the set 
\begin{align*}
A_{t_n} \setminus A_{T_0}
= ( \eta[0, t_n] \cup \bR ) \setminus ( \eta[0, T_0] \cup \bR )
\subset \overline{\gamma(T_0, t_n]}
\end{align*}
makes infinitely many crossings across the annulus $\smash{\overline{\mathbb{A}}_0}$.
However, this violates the right-continuity of $\gamma$, since there exists $\delta = \delta(r_0, R_0) > 0$ such that, taking $t_n = T_0 + \delta$, we arrive at a contradiction:
\begin{align*}
\sup_{ u,v \in [ T_0, T_0  + \delta ) } | \gamma(u) - \gamma(v) | < \tfrac{1}{2}(R_0 - r_0) .
\end{align*}
Hence, the set $A_{T_0}$ must make infinitely many distinct crossings across $\smash{\overline{\mathbb{A}}_0}$.
\end{itemize}
In summary, for any $z_0 \in \overline{\bH}$ and $0 < r_0 < R_0$, we have $T(z_0,r_0,R_0) = \infty$, which shows in particular that $\infty = T(z,\tfrac{r}{2},r) \leq t$, which is impossible.
Hence, $A_t$ is locally connected for all $t \geq 0$, as claimed.  
\end{proof}

\bigskip{}
\bibliographystyle{annotate}
\newcommand{\etalchar}[1]{$^{#1}$}

\end{document}